\documentclass[reqno,12pt,letterpaper]{amsart}

\usepackage{amsmath,amssymb,amsthm,graphicx,mathrsfs,url}
\usepackage[usenames,dvipsnames]{color}
\usepackage[colorlinks=true,linkcolor=Red,citecolor=Green]{hyperref}
\usepackage{amsxtra}
\usepackage{graphicx}
\usepackage{tikz}
\usetikzlibrary{positioning}
\usepackage{xcolor} 
\relax

\numberwithin{equation}{section}

\newcommand{\vertiii}[1]{{\left\vert\kern-0.25ex\left\vert\kern-0.25ex\left\vert #1 
		\right\vert\kern-0.25ex\right\vert\kern-0.25ex\right\vert}}

\newcommand{\be}{\begin{equation}}
	\newcommand{\ee}{\end{equation}}

\newcommand{\ov}{\overline}

\newcommand{\e}{{\mathrm{e} } }

\renewcommand{\Im}{\mathop{\rm Im}\nolimits}

\theoremstyle{plain}

\newtheorem{thm}{Theorem}[section]
\newtheorem{prop}{Proposition}[section]
\newtheorem{cor}[thm]{Corollary}

\newtheorem{lem}[prop]{Lemma}
\newtheorem{definition}[prop]{Definition}

\theoremstyle{definition}

\newtheorem{rem}[prop]{Remark}

\numberwithin{equation}{section}

\def\squarebox#1{\hbox to #1{\hfill\vbox to #1{\vfill}}}

\newcommand{\dd}{\mathrm{d}}

\newcommand{\noi}{\noindent}

\usepackage{amsxtra}

\ifx\pdfoutput\undefined
\DeclareGraphicsExtensions{.pstex, .eps}
\else
\ifx\pdfoutput\relax
\DeclareGraphicsExtensions{.pstex, .eps}
\else
\ifnum\pdfoutput>0
\DeclareGraphicsExtensions{.pdf}
\else
\DeclareGraphicsExtensions{.pstex, .eps}
\fi
\fi
\fi

\title[Quasi-invariant measures for the energy critical NLS]
{Quasi-invariance of Gaussian measures for the $3d$ energy critical nonlinear Schr\" odinger equation}
\author{Chenmin Sun}
\address{CNRS, Université Paris-Est Créteil, Laboratoire d’Analyse et de Mathématiques appliquées, UMR 8050 du CNRS, 94010 Créteil cedex, France}
\email{chenmin.sun@cnrs.fr}
	
\author{Nikolay Tzvetkov}
\address{Ecole Normale Sup\'erieure de Lyon, UMPA, UMR CNRS-ENSL 5669, 46, all\'ee d'Italie, 69364-Lyon Cedex 07, France}
\email{nikolay.tzvetkov@ens-lyon.fr}

\usepackage{amssymb}
\usepackage{amsmath, amsthm, amsopn, amsfonts}

\def\11{{\rm 1~\hspace{-1.4ex}l} }
\def\R{\mathbb R}
\def\C{\mathbb C}
\def\Z{\mathbb Z}
\def\N{\mathbb N}

\def\T{\mathbb T}

\def\fk{\mathfrak}

\begin{document}    
	
%	\definecolor{backgroundcolor}{RGB}{199, 238, 206} %%调节背景颜色2
%	\pagecolor{backgroundcolor} %%调节背景颜色3	
	
 \dedicatory{ Dedicated to Professor Carlos Kenig for his 70's birthday with admiration}
	
	\begin{abstract}
	We consider the $3d$ energy critical nonlinear Schr\" odinger equation with data distributed according to the Gaussian measure with covariance operator $(1-\Delta)^{-s}$, where $\Delta$ is the Laplace operator and $s$ is sufficiently large. We prove that the flow sends full measure sets to full measure sets. We also discuss some simple applications. This extends a previous result by Planchon-Visciglia and the second author from $1d$ to higher dimensions. 
	\end{abstract}   
	
	\maketitle 
	%	\setlength{\parskip}{0.3em}  
	%	\tableofcontents

	\section{Introduction}
\subsection{Motivation}	
The seminal paper \cite{F} initiated the study of Hamiltonian PDE's with initial data distributed according to the Gibbs measure which is constructed from the Hamiltonian functional. The Gibbs measure construction is strongly inspired by earlier developments in quantum field theory (see e.g. \cite{GJ,Simon}). These Gibbs measures are absolutely continuous with respect to suitable Gaussian measures (or shifts of such Gaussian measures). They are at least formally invariant under the corresponding Hamiltonian flow and therefore the underlying Gaussian measure (or its shift) are quasi-invariant under the flow. 
		
In dimensions greater than or equal to two, in order to consider initial data distributed according to the Gibbs measure a  renormalization of the equation under consideration is required, see e.g. \cite{Bo96, BDNY, FOSW, DNY1, ORT, OT}. 
Such renormalizations have strong motivations from Physics but they also make the results not so natural from a classical PDE perspective. A notable exception is the cubic  nonlinear Schr\" odinger equation for which a gauge transform links the (truncated) equation and its renormalized version. 
		
One may also observe that full Gibbs measure sets cover a very tiny part of the phase space of a Hamiltonian PDE and also that the Gibbs measure plays no role in the dynamics of most of the initial distributions of the initial data. %Observe that this is  in sharp contrast with Langevin type dynamics where the (same) Gibbs measure plays a truly distinguished role because it attracts all initial distributions.

Motivated by the above observations, in recent years there has been an activity aiming to show that a more general set of gaussian measures are quasi-invariant under Hamiltonian PDE's, see  \cite{BT, deb, forl, FS, forl_tol, GLT1,GLT2, 3DNLW, PTV1, PTV2, OS,OST,OTz4NLS,OTT,2DNLS, sigma,SXT}.  Such results allow to give a statistical description of the Hamiltonian flow for a larger class of initial distributions of the initial data (out of equilibrium dynamics). In particular, one obtains results for data of arbitrary Sobolev regularity  while the Gibbs measures live in low regularity Sobolev spaces. Moreover, no renormalization of the equation is required (even if renormalized energies may be used in the proof, see \cite{2DNLS,3DNLW,SXT}).  It is also worth to observe that it looks that  the question of  quasi-invariance of gaussian measures for Hamiltonian PDE's does not seem to have an analogue in the context of dissipative PDE's. 

Most of the results quoted in the previous paragraph are dealing with $1d$ models. The only results in dimensions $\geq 2$ are \cite{2DNLS,3DNLW,SXT}. They are dealing with non linear wave equations. The approach used in these works, based on renormalized energies, does not apply to the nonlinear Schr\" odinger equation (NLS) because of the lack of explicit smoothing in the equation. Our goal here is to resolve this issue and prove the quasi-invariance of Gaussian measures supported by sufficiently regular functions under the NLS flow in  higher dimensions. Our approach is based on normal form reductions as in \cite{OS,OST,OTz4NLS} combined with a soft analysis initiated in \cite{PTV1}. The main new idea in this paper is the identification of a remarkable cancellation of the worse pairing when estimating the divergence of the Hamiltonian vector field with respect to a weighted Gaussian measure (see Section \ref{Section:singular} below). The weight is naturally produced by the normal form reduction and therefore is related to the nature of the resonant set  (while is the wave equation case the weight is related to the potential energy). 
This remarkable  cancellation is certainly related to the Hamiltonian structure and hopefully may be used in other contexts. 
Our result is only giving qualitative quasi-invariance for sufficiently regular initial distributions. Therefore several challenging issues remain open (see the remarks after the statement of the main result). 
%%%%%%%%%%%%%%%%%%%%%%%
%%%%%%%%%%%%%%%%%%%%%%%%%%
\subsection{Main result} 
In this work we will study the most challenging model to which we succeeded to make work our approach. 
Therefore, we consider  the defocusing energy-critical NLS
\begin{align}\label{NLS} 
i\partial_tu +\Delta u=|u|^{4}u,\quad (t,x) \in \mathbb{R}\times\T^3,
\end{align}
where  $\T^3:=\R^3/(2\pi\Z)^3$.  Equation \eqref{NLS} is a Hamiltonian system with the conserved mass and energy: 
$$ 
M[u]:=\int_{\T^3} |u|^2 \mathrm{d}x,\quad H[u]:= \frac{1}{2} \int_{\T^3}|\nabla u|^2\mathrm{d}x+\frac{1}{6}\int_{\T^3}|u|^{6}\mathrm{d}x.
$$
These conservation laws allow to construct relatively easily global weak solutions of \eqref{NLS} in the Sobolev spaces $H^1(\T^3)$ via basic compactness arguments. 
Unfortunately such techniques are not suitable to prove uniqueness and propagation of higher Sobolev regularities. 
Thanks to the remarkable work by Ionescu-Pausader \cite{IoPau} (based on the previous contributions \cite{BaGe,Bo_JAMS,CKSTT,HTTz,KeMe}) we know that \eqref{NLS} is globally well posed in $H^{\sigma}(\T^3)$, $\sigma\geq 1$. Namely, for every $u_0\in H^{\sigma}(\T^3)$, $\sigma\geq 1$ there exists a unique global solution of \eqref{NLS} in $C(\R;H^{\sigma}(\T^3))$ such that $u(0,x)=u_0(x).$
Let us denote by $\Phi(t)$ the Ionescu-Pausader flow of \eqref{NLS}. 

When studying the statistical properties of \eqref{NLS}, we assume that the initial data are distributed according to the Gaussian probability measure $\mu_s$, formally defined as 	``$\frac{1}{\mathcal{Z}}\mathrm{e}^{-\frac{1}{2}\vertiii{u}_{H^s}^2}\dd u$'', induced by the random Fourier series
\begin{align}\label{randominitial}  
\phi^{\omega}(x):=\sum_{k\in\Z^3}\frac{g_k(\omega)}{\sqrt{1+|k|^{2s}}}\mathrm{e}^{ik\cdot x}\, ,
\end{align}
where $(g_k)_{k\in\Z^3}$ are independent, identically distributed standard complex Gaussian random variables and the equivalent Sobolev norm $\vertiii{\cdot}_{H^s}$ is defined in \eqref{def:Hsequiv}. Thanks to the Kakutani theorem we know that for at least $s\geq 10$ the measure $\mu_s$ is absolutely continuous with respect to the Gaussian measure with covariance operator $(1-\Delta)^{-s}$ 
(see e.g.  \cite{2DNLS} for such an application of the Kakutani theorem). It is well-known that
$$ 
				H^{(s-\frac{3}{2})-}(\T^3):=\bigcap_{\sigma<s-\frac{3}{2}}H^{\sigma}(\T^3)
		$$
		is of full $\mu_s$ measure and $\mu_s(H^{s-\frac{3}{2}}(\T^3))=0$.
For bigger $s$, typical functions/distributions have bigger Sobolev regularity.
 
  Thanks to \cite{IoPau}, when $s>\frac{5}{2}$, the flow $\Phi(t)$ of \eqref{NLS} exists globally on $H^{\sigma}(\T^3)$ for any $1\leq \sigma<s-\frac{3}{2}$. In particular, a unique global solution exists for any initial data on $H^{(s-\frac{3}{2})-}(\mathbb{T}^3)$,  $s>\frac{5}{2}$.
Our main result reads as follows. 
\begin{thm}\label{thm:main}
Assume that $s\geq 10$. Then $\mu_s$ is quasi-invariant under $\Phi(t)$. More precisely,  for every $t\in \R$, $(\Phi(t))_*\mu_s\ll \mu_s\ll (\Phi(t))_*\mu_s$, where  $(\Phi(t))_*\mu_s$ is the push forward of $\mu_s$ by $\Phi(t)$.
\end{thm}
In the statement above, the notation $\mu\ll \nu$ for two measures $\mu,\nu$ 
%temporarily 
means that $\mu$ is absolutely continuous with respect to $\nu$.

In the proof of Theorem \ref{thm:main} below, we retain us of using arithmetic arguments such as the divisor bound. Therefore the result of  Theorem \ref{thm:main} remains valid for irrational tori with essentially the same proof. 

In view of \cite{aizen}, it seems hopeless to construct a Gibbs measure for \eqref{NLS} (and any other energy critical problem). This gives a further motivation for studying quasi-invariant Gaussian measures for \eqref{NLS} or any other model for which the Gibbs measure construction fails. 

The result of  Theorem \ref{thm:main} remains true (with a simpler proof) for the cubic $3d$ NLS
$$
i\partial_tu +\Delta u=|u|^{2}u,\quad (t,x) \in \mathbb{R}\times\T^3,
$$
and also for the $2d$ NLS with an arbitrary polynomial defocusing nonlinearity.

As already mentioned,  Theorem \ref{thm:main} only gives qualitative quasi-invariance. It would be interesting to obtain quantitative bounds on the resulting Radon-Nikodym derivatives. Such quantitative bounds were obtained in some previous works on the subject, the most notable being the paper by Forlano-Tolomeo \cite{forl_tol} where such quantitative informations on the Radon-Nikodym derivative are used in order to perform the Bourgain globalization argument, i.e. quasi-invariance is used in order to construct the flow. The  Forlano-Tolomeo argument is performed for a $1d$ model and it would be very interesting to extend it to higher dimensions. In particular, it would be interesting to decide whether  Theorem \ref{thm:main} holds in the supercritical regime, i.e. for some $s<\frac 5 2$ (in this regime the existence of the flow should rely on a probabilistic well-posedness in the spirit of \cite{NaSta}). But at this stage it is even not clear to us how to prove   Theorem \ref{thm:main} in the natural subcritical range  $s>\frac 5 2$. By using the dispersive effects we can relax slightly the assumption $s\geq 10$ but it would still be away for the natural subcritical assumption  $s>\frac 5 2$. In summary, much more remains to be understood concerning the transport of $\mu_s$ under the NLS flow and its connection with the probabilistic well-posedness theory. 
%%%%%%%%%%%%%%%%%%%%%%%%%%%%%%%%%%%%%%%%%%%%%%%%%%%%%%%%%%%%%%%%%%		
\subsection{Applications} In this section we present two simple corollaries of Theorem~\ref{thm:main}. Recall that the random field \eqref{randominitial} is a stationary Gaussian process on $\T^2$. In particular, for each fixed $x\in\T^2$, $\phi^\omega(x)$ is a complex Gaussian random variable with law $\mathcal{N}_{\mathbb{C}}(0,\sigma^2)$, where
$
\sigma^2=\sum_{k\in\Z^3}\frac{1}{1+|k|^{2s}}.
%<\infty. 
$
Consequently, the probability density of $\phi^{\omega}(x)$ is $\frac{1}{2\pi \sigma}\e^{-\frac{|y|^2}{2\sigma}}\dd y$ on $\mathbb{C}=\R^2$. In particular, the law of $\phi^{\omega}(x)$ is absolutely continuous with respect to the Lebesgue measure. A natural question is to study the regularity of the law for the random variable $u(t,x)$, evolved by \eqref{NLS} with the initial data $\phi^{\omega}$.  This type of problems has been intensively studied in the field of Stochastic analysis. For many classes of stochastic (partial) differential equations, the regularity of laws of solutions can be obtained via the  Malliavin Calculus (see the book of Nualart \cite{Nualart} and references therein). The Malliavin Calculus was originally developed by P.~Malliavin \cite{Malliavin} to bring a new proof of H\"ormander's theorem for hypoelliptic operators. We do not intend to include any element of the Malliavin Calculus in this article, but to give a simple application of the quasi-invariance property to obtain the absolutely continuity for the law of solutions of NLS with random initial data, which can be viewed as a pointwise version of the quasi-invariance property of the NLS equation displayed by Theorem~\ref{thm:main}.
\begin{cor}\label{cor1}
Assume that $s\geq 10$ and fix  $(t_0,x_0) \in \mathbb{R}\times\T^3$. Let $u(t,x,\omega)$ be the solution of \eqref{NLS} with data \eqref{randominitial}.
Then the law of the complex random variable $\omega\mapsto u(t_0,x_0,\omega)$ has a density with respect to the Lebesgue measure on $\C$. 
\end{cor}
In order to prove Corollary~\ref{cor1}, we observe that we need to study the composition of $\Phi(t)$ and the evaluation map $u\mapsto u(t_0,x_0)$. Then it suffices to apply 
Theorem~\ref{thm:main} for $\Phi(t)$ and the observation before the statement of Corollary~\ref{cor1} for the evaluation map. It is likely that the Malliavin Calculus can be useful to get regularity properties of the densities appearing in the statement of Corollary~\ref{cor1}. In Corollary~\ref{cor1}, one may replace the evaluation map by other finite dimensional projections.
For instance, one may show that for every $k\in \Z^3$, the law of the Fourier coefficient $\widehat{u}(t,k,\omega)$ has a density with respect to  the Lebesgue measure on $\C$. 

Let us also emphasize that the Malliavin Calculus methods can be applied to prove quasi-invariance for maps from infinite dimensional gaussian spaces to finite dimensional spaces, while in 
Theorem~\ref{thm:main} we deal with the more complex situation of a map from an infinite dimensional gaussian space to itself.

Another simple consequence of Theorem~\ref{thm:main} is the following $L^1$-stability result.
\begin{cor}\label{stabi} 
Assume that $s\geq 10$. Let $f_1,f_2\in L^1(\dd\mu_s)$ and $\Phi(t)$ the flow of \eqref{NLS}. Then for any $t\in\R$, the transports of measures $f_1(u)\dd\mu_s(u)$, $f_2(u)\dd\mu_s(u)$ by $\Phi(t)$ are given by $F_1(t,u)\dd\mu_s(u)$ and $F_2(t,u)\dd\mu_s(u)$  respectively, for suitable $F_1(t,\cdot),F_2(t,\cdot)\in L^1(\dd\mu_s)$. Moreover,
		$$ \|F_1(t,\cdot)-F_2(t,\cdot)\|_{L^1(\dd\mu_s)}=\|f_1-f_2\|_{L^1(\dd\mu_s)}.
		$$
	\end{cor}
One may prove Corollary \ref{stabi}  by performing the computations from	 \cite{STz1}. A more direct proof can be given by observing that $\Phi(t)$ is a measurable map and therefore the total variation distance between $F_1(t,u)\dd\mu_s(u)$  and $F_2(t,u)\dd\mu_s(u)$ is smaller than the total variation distance between   $f_1(u)\dd\mu_s(u)$  and $f_2(u)\dd\mu_s(u)$. This implies that 
$$ 
\|F_1(t,\cdot)-F_2(t,\cdot)\|_{L^1(\dd\mu_s)}\leq \|f_1-f_2\|_{L^1(\dd\mu_s)}.
$$
Using the reversibility of the NLS flow we get the inverse inequality. 

The remaining part of this paper is devoted to the proof of Theorem~\ref{thm:main}. 
In Section~\ref{Sec:modifiedenergy} we perform the normal form reduction, we define accordingly suitable weighted Gaussian measures and we state the key energy estimates.
In Section~\ref{proof:quasi} we perform the soft analysis leading from the energy estimates to the quasi-invariance result stated in  Theorem~\ref{thm:main}. 
In Section~\ref{prelim} we introduce our basic counting tool and the Wiener chaos estimate useful for our purposes.  
In Section~\ref{graficd} we decompose the divergence of the Hamiltonian vector field with respect to the weighted Gaussian measures in several pieces according to the possible pairings. 
In Section~\ref{energyI} we estimates the contributions of the first generation. 
Section~\ref{Section:singular} deals with the most singular contribution resulting from pairings between different generations. 
This is the most delicate part of our analysis containing the remarkable algebraic cancellations mentioned above. 
In Section~\ref{Section:rest} we treat the remainder terms in which the singular pairings are not presented. 
Finally in an Appendix we prove some approximation results for \eqref{NLS}, crucially exploited in Section~\ref{proof:quasi}. Let us emphasize that because of the critical nature of the Cauchy problem for \eqref{NLS}, the approximation argument is much more delicate compared to the previous literature on quasi-invariant Gaussian measures for Hamiltonian PDE's. 
%%%%%%%%
\subsection*{Acknowledgments} This work is partially supported by the ANR project Smooth ANR-22-CE40-0017.
The authors would like to thank the anonymous reviewer for his/her careful reading and  constructive comments. The authors would like to thank Alexis Knezevitch for pointing out an error in a previous version of the manuscript.
%%%%%%%	%%%%%%%%%%%%%%%%%%%%%%%%%%%%%%%%%%%%%%%%%%%%%%%%%%%%%%%%%	
%%%%%%%%%%%%%%%%%%%%%%%%%%%%%%%%%%%%%%%%%%%%%%%%%%%%%%%%%%%%%%%%%%%%%%%%%%%%%%%
\section{Modified energy and the weighted Gaussian measure}\label{Sec:modifiedenergy}
\subsection{An approximated system}
Fix a radial cutoff function $\chi\in C_c^{\infty}(\R^3)$ such that $\chi\equiv 1$ on $[-\frac{1}{2},\frac{1}{2}]^3$ and supp$(\chi)\subset \{|x|< 1\}$. For $N\in \N$, set $\chi_N(\cdot):=\chi(N^{-1}\cdot)$ and $S_N=\chi_N(\sqrt{-\Delta})$ the smooth frequency truncation and $\Pi_N=\mathbf{1}_{\sqrt{-\Delta}\leq N}$ the sharp frequency truncation. By definition,
$$ S_N\Pi_N=\Pi_NS_N=S_N,\quad S_N^*=S_N. $$
The advantage of using the operator $S_N$ is that $S_N$ is uniformly bounded on $L^p(\T^3)$ for $1<p<\infty$, which is crucial when taking the limit of the approximated system in the energy critical case.
Similar to the situation in \cite{BTT}, we consider the following smoothly approximated NLS equation
\begin{align}\label{NLS:SN}
	\begin{cases}
		&i\partial_tu_N+\Delta u_N=S_N(|S_Nu_N|^4S_Nu_N),\\
		&u_N|_{t=0}=u_0\in H^{\sigma}(\T^3).
	\end{cases}
\end{align}
As in \cite{BTT}, the solution of \eqref{NLS:SN} can be decomposed as two components on $\mathcal{E}_N:=\Pi_NL^2(\T^3)$ and $\mathcal{E}_N^{\perp}:=(\mathrm{Id}-\Pi_N)L^2(\T^3)$. 
This naturally leads to a splitting of $\mu_s$ as $\dd\mu_s=d\mu_{s,N}\otimes d\mu_{s,N}^{\perp}$ for every $N\in\N$, where  $\mu_{s,N}$ is a measure on $\mathcal{E}_N$ while  $\mu_{s,N}^{\perp}$  is a measure on  $\mathcal{E}_N^{\perp}$. 
The finite-dimensional part of \eqref{NLS:SN} on $\mathcal{E}_N$ is a Hamiltonian system (see \cite [Lemma~8.1] {BTT}), while the infinite-dimensional part is the linear evolution $\e^{it\Delta}$.
Thanks to the Cauchy-Lipchitz  theorem and the defocusing nature, the solution of \eqref{NLS:SN} is global and we denote by $\Phi_N(t)$ its flow, which can be factorized as $(\widetilde{\Phi}_N(t),\e^{it\Delta})$ on $\mathcal{E}_N\times \mathcal{E}_N^{\perp}$, where $\widetilde{\Phi}_N(t)$ is the restriction of $\Phi_N(t)$ on the finite-dimensional space $\mathcal{E}_N$, which is a Hamiltonian flow on $\mathcal{E}_N$.
By convention, we denote $\Phi(t)$ by $\Phi_\infty(t)$. 
%%%%%%%%%%%%%%%%%%%%%%%%%%%%%%%%%%%%%%%%%%%%%%%
\subsection{Poincar\'e-Dulac normal form and the modified energy}
To construct suitable weighted measures for our study, we must identify a modified energy functional.  Consider a smooth solution $u_N(t)$ of \eqref{NLS:SN}. We introduce a new unknown:
\[
v(t) = \e^{-it\Delta}u_N(t).
\]
Expanding $v(t)$ in the Fourier series, we have:
\[
v(t,x) = \sum_{k \in \Z^3} v_k(t)\mathrm{e}^{ik\cdot x},
\]
from which it follows that $v_k(t)$ satisfies the equation:
\begin{align}\label{NLSk}
	i\partial_tv_k(t) = \chi_N(k)\sum_{k_1 - k_2 + k_3 - k_4 + k_{5} = k  }\e^{-it\Omega(\vec{k})}\cdot \left(\prod_{j=1}^5\chi_N(k_j)\right)\cdot v_{k_1}(t)\ov{v}_{k_2}(t)\cdots v_{k_{5}}(t),
\end{align}
where
\[
\Omega(\vec{k}) = \sum_{j=1}^{5}(-1)^{j-1}|k_j|^2 - |k|^2
\]
is the resonant function.

To construct the modified energy, it is more convenient to use an equivalent of the Sobolev norm for $s\geq 0$:
\begin{align}\label{def:Hsequiv}
\vertiii{f}_{H^s(\T^3)}^2:=\sum_{k\in\Z^3}(1+|k|^{2s})|\widehat{f}(k)|^2. %% New command for the three lines norms
\end{align}
A simple computation using symmetry of indices yields
\begin{align}\label{deriveeHs}
	\frac{1}{2}\frac{d}{dt}\vertiii{v(t)}_{H^s}^2
	=&-\frac{1}{6}\Im\sum_{k_1-k_2+\cdots -k_{6}=0  } \psi_{2s}(\vec{k})\e^{-it\Omega(\vec{k})}
	\Big(\prod_{j=1}^6\chi_{N}(k_j)\Big)
	v_{k_1}\ov{v}_{k_2}\cdots \ov{v}_{k_{6}},
\end{align}
where in the above expression, we abuse the notation slightly and denote
\[
\psi_{2s}(\vec{k})=\sum_{j=1}^{6}(-1)^{j-1}|k_j|^{2s},\quad 
\Omega(\vec{k})=\sum_{j=1}^{6}(-1)^{j-1}|k_j|^2.
\]
The basic estimate for $\psi_{2s}(\vec{k})$ is 
$$
|\psi_{2s}(\vec{k})|\lesssim  |k_{(1)}|^{2s-2}(|k_{(3)}|^2+|\Omega(\vec{k})|),
$$
where $|k_{(1)}|\geq |k_{(2)}|\geq \cdots \geq |k_{(6)}|$ is rearrangement of $k_1,\cdots,k_6$ and $k_1-k_2+\cdots -k_{6}=0$ (see Lemma \ref{monto_gordo} below).
Note that each $v_{k_j}$ will be accompanied with $\chi_N(k_j)$, and the capital $N$ plays no role in our analysis, we will simply write 
$ w_{k_j}:=\chi_N(k_j)v_{k_j}
$ in the sequel. Note that
\begin{align}\label{eq:wk} i\partial_tw_k=\chi_N(k)^2\sum_{k_1-k_2+k_3-k_4+k_5=k}\e^{-it\Omega(\vec{k})}\cdot w_{k_1}\ov{w}_{k_2}\cdots w_{k_5}.
\end{align}

In order to truncate the level set of the resonant function, we further introduce the symmetric factor
\[
\lambda(\vec{k})=\Big(\sum_{j=1}^6|k_j|^2\Big)^{\frac{1}{2}}.
\]

As the resonant function $\Omega(\vec{k})$ takes integer values\footnote{This fact is not essential for our result and the proof. Though we keep to work on the rational torus for convenience. }, we will decompose the set of indices $k_1,\cdots,k_{6}$ according to the level set of $\Omega(\vec{k})$. In order to perform the differentiations by parts in time, we further write 
	\begin{align}
		\frac{1}{2}\frac{d}{dt}\vertiii{v(t)}_{H^s}^2=&-\frac{1}{6}\Im\sum_{k_1-k_2+\cdots-k_{6}=0  }\chi\Big(\frac{\Omega(\vec{k})}{\lambda(\vec{k})^{\delta_0}}\Big)\psi_{2s}(\vec{k})\e^{-it\Omega(\vec{k})}w_{k_1}\ov{w}_{k_2}\cdots \ov{w}_{k_{6}}\notag \\
		&-\frac{1}{6}\Im\sum_{k_1-k_2+\cdots-k_{6}=0
		}\Big(1-\chi\Big(\frac{\Omega(\vec{k})}{\lambda(\vec{k})^{\delta_0}}\Big)\Big)\frac{\psi_{2s}(\vec{k}) }{-i\Omega(\vec{k})}\partial_t\Big(\e^{-it\Omega(\vec{k})}w_{k_1}\ov{w}_{k_2}\cdots\ov{w}_{k_{6}}\Big)\notag\\
		&+\frac{1}{6}\Im\sum_{k_1-k_2+\cdots-k_{6}=0
		}\Big(1-\chi\Big(\frac{\Omega(\vec{k})}{\lambda(\vec{k})^{\delta_0}}\Big)\Big)\frac{\psi_{2s}(\vec{k}) }{-i\Omega(\vec{k})}\e^{-it\Omega(\vec{k})}\partial_t(w_{k_1}\ov{w}_{k_2}\cdots\ov{w}_{k_{6}}),
		 \label{deriveeHs1}
	\end{align}
	where $0<\delta_0<\frac{2}{3}$ is close to $\frac{2}{3}$ (here we denote by $\chi$ a standard bump function from $\R$ to $\R$).
	Motivated by the above formula, we define the modified energy (with $w=\chi_N(\sqrt{-\Delta})v$)
	\begin{equation}\label{modifiedenergy}  \mathcal{E}_{s,t}(v):=\frac{1}{2}\vertiii{v}_{H^s(\T^3)}^2+\mathcal{R}_{s,t}(w),
	\end{equation}
	where
	$$ \mathcal{R}_{s,t}(w):=\frac{1}{6}\Im \sum_{k_1-k_2+\cdots-k_{6}=0  }\Big(1-\chi\Big(\frac{\Omega(\vec{k})}{\lambda(\vec{k})^{\delta_0}}\Big)\Big) \frac{\psi_{2s}(\vec{k}) }{-i\Omega(\vec{k}) }\e^{-it\Omega(\vec{k})}w_{k_1}\ov{w}_{k_2}\cdots \ov{w}_{k_{6}}.
	$$
	Changing back to the variable $u$, the modified energy is 
	\begin{equation}\label{Est}
		\mathcal{E}_{s,t}(v)= E_{s,N}(u):=\frac{1}{2}\vertiii{u}_{H^s(\T^3)}^2+R_{s,N}(u),
	\end{equation}
	where
	\begin{equation}\label{Rst}
		R_{s,N}(u):=\frac{1}{6}\Im \sum_{k_1-k_2+\cdots-k_{6}=0}\Big(1-\chi\Big(\frac{\Omega(\vec{k})}{\lambda(\vec{k})^{\delta_0}}\Big)\Big) \frac{\psi_{2s}(\vec{k}) }{-i\Omega(\vec{k}) }\cdot
		\Big(\prod_{j=1}^6 \chi_N(k_j)
		\Big)\cdot 
		\widehat{u}_{k_1}\ov{\widehat{u}}_{k_2}\cdots \ov{\widehat{u}}_{k_{6}}\;.
	\end{equation}
	We define $R_{s}(u)$ as $R_{s,N}(u)$ without $\prod_{j=1}^6 \chi_N(k_j)$. Sometimes  $R_{s}(u)$ will be denoted by $R_{s,\infty}(u)$. We similarly define 
	$E_{s}(u)$ which may also be denoted by $E_{s,\infty}(u)$. 
	
	The modified energy \eqref{Est} will play a crucial role in our analysis. We refer to \cite{visciglia} for a survey on the use of modified energies in the analysis of dispersive PDE's.

Then from \eqref{deriveeHs1} and the equation  \eqref{NLSk} of $u_N(t)$, and symmetry of indices, we have (with $w_k=\chi_N(k)v_k$)
\begin{align}\label{dmodifiedenergy}
\frac{d}{dt}E_{s,N}(u_N(t))=	\frac{d}{dt}\mathcal{E}_{s,t}(v) 
	&= -\frac{1}{6}\Im\sum_{k_1-k_2+\cdots-k_{6}=0  } \chi\Big(\frac{\Omega(\vec{k})}{\lambda(\vec{k})^{\delta_0}}\Big) \psi_{2s}(\vec{k}) \e^{-it\Omega(\vec{k})} w_{k_1}\overline{w}_{k_2}\cdots \overline{w}_{k_{6}} \notag \\
	&+ \frac{1}{2}\Im\sum_{k_1-k_2+\cdots-k_{6}=0 }\Big(1-\chi\Big(\frac{\Omega(\vec{k})}{\lambda(\vec{k})^{\delta_0}}\Big)\Big)\frac{\psi_{2s}(\vec{k})}{\Omega(\vec{k})} \notag \\
	&\qquad \times \sum_{\substack{k_1=p_1-p_2+\cdots+p_{5} }} \e^{-it\big(\Omega(\vec{k})+\Omega(\vec{p})\big)}\chi_N(k_1)^2 w_{p_1}\overline{w}_{p_2}\cdots w_{p_{5}}\overline{w}_{k_2}\cdots \overline{w}_{k_{6}} \notag \\
	&- \frac{1}{2}\Im\sum_{k_1-k_2+\cdots -k_{6}=0}\Big(1-\chi\Big(\frac{\Omega(\vec{k})}{\lambda(\vec{k})^{\delta_0}}\Big)\Big)\frac{\psi_{2s}(\vec{k})}{\Omega(\vec{k})} \notag \\
	&\qquad \times \sum_{k_2=q_1-q_2+\cdots +q_{5}} \e^{-it\big(\Omega(\vec{k})-\Omega(\vec{q})\big)}\chi_N(k_2)^2 w_{k_1}\overline{w}_{q_1}\cdots \overline{w}_{q_{5}}w_{k_3}\cdots \overline{w}_{k_{6}},
\end{align}
where
\begin{align*}
	\Omega(\vec{p})=\sum_{j=1}^{5}(-1)^{j-1}|p_j|^2-|k_1|^2, \quad
	\Omega(\vec{q})=\sum_{j=1}^{5}(-1)^{j-1}|q_j|^2-|k_2|^2. 
\end{align*}
%%%%%%%%%%%%%%%%%%%%%%%%%%%%%%%%%%%%%%%%%%%%%%%%%%%%%%%%%%%%%%%%%%%%%%%%%%%%%%%%%%%%%%%%%%%%%%%%%%%%%%%%%
	
\subsection{The weighted measure}
Using the modified energy, we define the weighted Gaussian measure for given $R> 1$
\begin{align}\label{rhostr} 
	\dd\rho_{s,R,N}(u) =\chi_R(\|u\|_{H^{\sigma}})\cdot \e^{-R_{s,N}(u)} \dd\mu_{s,N}(u),\quad \dd\ov{\rho}_{s,R,N}(u):=\dd\rho_{s,R,N}\otimes \dd\mu_{s,N}^{\perp},
\end{align}
where the functional $R_{s,N}(u)$ is defined by \eqref{Rst} and $\chi_R(\cdot)=\chi(R^{-1}\cdot)$ is a cutoff.
\begin{prop}[Local existence of the weighted measure]\label{weightedmeasure}
Let $s \geq 10, R \geq 1$, 
$\sigma<s-\frac{3}{2}$, close to $s-\frac{3}{2}$
%\sigma\in\big(s-2,s-\frac{3}{2}\big), 
and $N \in \N$. Then for any $p \in [1,\infty)$, there exists a uniform in $N$ constant $C(p,s,R) > 0$, such that
\[
\big\| \chi_R(\|u\|_{H^{\sigma}}) \cdot \e^{|R_{s,N}(u)|}  \big\|_{L^p(\dd\mu_s)} \leq C(p,s,R).
\]
Moreover, for fixed $R > 0$,
\[
\lim_{N\rightarrow\infty}\big\| \chi_R(\|u\|_{H^{\sigma}}) \e^{-R_{s,N}(u)} - \chi_R(\|u\|_{H^{\sigma}})\e^{-R_{s}(u)}  \big\|_{L^p(\dd\mu_s)} = 0.
\]
\end{prop}
Recall that $\Phi_N(t)$ is the flow of \eqref{NLS:SN} while $\Phi_\infty(t)=\Phi(t)$ is the flow of \eqref{NLS}.  Another key proposition is the following weighted energy estimate:
\begin{prop}[Weighted energy estimate]\label{energyestimate}
Let $s\geq 10, R \geq 1$, 
%\sigma\in\big(s-2,s-\frac{3}{2}\big), 
$\sigma<s-\frac{3}{2}$, close to $s-\frac{3}{2}$ and $N\in\N\cup\{\infty\}$.  Set
$$ 
Q_{s,N}(u)=\frac{\dd}{\dd t}E_{s,N}(\Phi_N(t)u)|_{t=0}
$$
and denote by $B_R^{H^\sigma}$ the centered ball in $H^{\sigma}(\T^3)$ of radius $R$.  Then there exist constants $C(s,R)>0$ and 
$\beta\in(0,1)$, such that for all $p\in[2,\infty)$ and  $N\in \N\cup\{\infty\}$,
$$
 \|\mathbf{1}_{B_R^{H^{\sigma}}}(u)\cdot Q_{s,N}(u) \|_{L^p(\dd\mu_s)}\leq C(s,R )p^{\beta}.
 $$
 Thanks to Proposition~\ref{weightedmeasure}, we have also for all $N\in\N\cup\{\infty\}, p\in[1,\infty)$,
 $$  
 \|\mathbf{1}_{B_R^{H^{\sigma} }}(u)\cdot Q_{s,N}(u) \|_{L^p(\dd\ov{\rho}_{s,R,N})}\leq C(s,R )p^{\beta}.
$$
\end{prop} 
The proof of above two propositions will occupy the main part of the article. 
To prove the quasi-invariance of the full system, we need to pass to the limit $N\rightarrow \infty$ in  the approximated equation \eqref{NLS:SN}. 
This will be done in the next section. 
%%%%%%%%%%%%%%%%%%%%%%%%
%%%%%%%%%%%%%%%%%%%%%%%%%%%%%%%%%%%%%%%%%%%%%%%%%
\section{Proof of the quasi-invariance assuming energy estimates}\label{proof:quasi}
In this section we prove Theorem~\ref{thm:main}, assuming Proposition~\ref{weightedmeasure} and Proposition~\ref{energyestimate}.
\subsection{Approximation theory for the energy-critical NLS}
\begin{prop}\label{uniformboundHsigma}
Assume that $\sigma\geq 1$. There exists a constant $\Lambda(R,T)>0$, depending only on $T>0$, $R>0$ and $\sigma\geq 1$, such that for any $u\in B_R^{H^{\sigma}}$,
$$ \sup_{|t|\leq T}\|\Phi(t)u\|_{H^{\sigma}}+\sup_{|t|\leq T}\|\Phi_N(t)u\|_{H^{\sigma}}\leq \Lambda(R,T),\quad \forall\, N\in\N.
$$
\end{prop}
%%%%%%%%%%%%%%%
\begin{prop}\label{approximation} 
Assume that $\sigma\geq 1$.	Let $K$ be a compact subset of $H^{\sigma}(\T^3)$ and $T>0$. Then uniformly in $|t|\leq T$ and $u\in K$, $$\lim_{N\rightarrow\infty}\|\Phi_N(t)u-\Phi(t)u\|_{H^{\sigma}}=0.
	$$
	\end{prop}
Observe that since $\Phi_N(t)$ and $\Phi(t)$ are continuous we have that for any $|t|\leq T$ and $N\in \N$, $\Phi_N(t)(K), \Phi(t)(K)$ are also compacts in $H^{\sigma}(\T^3)$.  The proof of Proposition~\ref{uniformboundHsigma}, Proposition~\ref{approximation} will be given in the appendix.
%%%%%%%%%	
\subsection{Proof of quasi-invariance}
First, we prove:
	\begin{lem}\label{Quasi-N}
	Let $T\geq 1$. Let $A\subset B_R^{H^{\sigma}}$ be a Borel measurable set. Then there exist $\epsilon_0>0$ and $C_{s,R,T}>0$, such that for all $N\in\N$ and $|t|\leq T$, $\mu_s(\Phi_N(t)(A))\leq C_{s,R,T}\cdot \mu_s(A)^{\frac{1-\epsilon_0}{4}}$.
	\end{lem}
	
	\begin{proof}
Let $\Lambda(R,T)>0$ be the constant in Proposition \ref{uniformboundHsigma}, such that for all $R>0, N\in \N\cup\{\infty\}$,
$$ 
\Phi_N(t)(B_R^{H^{\sigma}})\subset B_{\Lambda(R,T)}^{H^{\sigma}},\quad |t|\leq T.
$$
Denote $R_1:=\Lambda(\Lambda(R,T),T)$, and we consider the weighted measure
\begin{align*} 
\dd\ov{\rho}_{s,R_1,N}(u)=&\dd\rho_{s,R_1,N}(u)\otimes \dd\mu_{s,N}^{\perp}\\=& \chi_{R_1}(\|u\|_{H^{\sigma}})\frac{1}{\mathcal{Z}_N}\e^{-E_{s,N}(u)}\Big(\prod_{|k|\leq N}\dd\widehat{u}_k\Big)\otimes \dd\mu_{s,N}^{\perp},
\end{align*}
where $\mathcal{Z}_N>0$ is the normalizing constant appearing in the finite-dimensional truncation of the Gaussian measure
$$ \dd\mu_{s,N}(u)=\frac{1}{\mathcal{Z}_N}\e^{-\frac{1}{2}\sum_{|k|\leq N}(1+|k|^{2s})|\widehat{u}_k|^2 }\Big(\prod_{|k|\leq N}\dd \widehat{u}_k\Big). 
$$
For $A\subset B_R^{H^{\sigma}}$, from Proposition \ref{uniformboundHsigma},  for any $|t_1|,|t_2|\leq T$ and $N\in\N$,
$$ \Phi_N(t_2)\circ\Phi_N(t_1)(A)\subset B_{R_1}^{H^{\sigma}}.
$$ 
In particular, for any $u\in A$, $|t|\leq 2T$, $\|\Phi_N(t)u\|_{H^{\sigma}}\leq R_1$.
Now for $|t_0|\leq T, |t|\leq 1$, using that for  $\chi_{R_1}(\|\Phi_N(t)u\|_{H^{\sigma}})\equiv 1$ for $u\in A$, as in  \cite{OTz4NLS,2DNLS, sigma}, we can obtain the following change of variable formula 
\begin{align*}
\ov{\rho}_{s,R_1,N}(\Phi_N(t_0+t)(A))=\int_{A}\frac{1}{\mathcal{Z}_N}\e^{-E_{s,N}(\Pi_N \Phi_N(t_0+t)u)}\prod_{|k|\leq N}\dd\widehat{u}_k\, 
\dd\mu_{s,N}^{\perp}(u).
\end{align*}
Observe that 
$$ 
Q_{s,N}(u)=\frac{\dd}{\dd t}E_{s,N}(\Phi_N(t)u)|_{t=0}=\frac{\dd}{\dd t}E_{s,N}(\Pi_N\Phi_N(t)u)|_{t=0}\,.
$$
Taking the time derivative of the above equality and evaluate it at $t=0$, we obtain the identity
\begin{align*}
\frac{\dd}{\dd t}\ov{\rho}_{s,R_1,N}(\Phi_N(t_0+t)(A))|_{t=0}=&
%- \int_{\Pi_N^{\perp}(A)}d\mu_{s,N}^{\perp}(u)\int_{\Pi_N(A)}Q_{s,N}(\Phi_N(t_0)u)\frac{1}{\mathcal{Z}_N}\e^{-E_{s,N}(\Phi_N(t_0)u)}\prod_{|k|\leq N}d\widehat{u}_k
-\int_{A}\frac{1}{\mathcal{Z}_N}
Q_{s,N}(\Phi_N(t_0)u)\,
\e^{-E_{s,N}(\Pi_N \Phi_N(t_0)u)}\prod_{|k|\leq N}\dd\widehat{u}_k \,\dd\mu_{s,N}^{\perp}(u)
\\
=&
-\int_{\Phi_N(t_0)(A)}\frac{1}{\mathcal{Z}_N}
Q_{s,N}(u)\,\e^{-E_{s,N}(\Pi_N u)}\prod_{|k|\leq N}\dd\widehat{u}_k \,\dd\mu_{s,N}^{\perp}(u),
%-\int_{\Pi_N^{\perp}\Phi_N(t_0)(A)}d\mu_{s,N}^{\perp}(u)\int_{\Pi_N\Phi_N(t_0)(A)}Q_{s,N}(u)\frac{1}{\mathcal{Z}_N}\e^{-E_{s,N}(u)}\prod_{|k|\leq N}d\widehat{u}_k,
\end{align*}
where we again used the change of variable formula. As for $u\in \Phi_N(t_0)(A)$, $1\leq \chi_{R_1}(\|u\|_{H^{\sigma}})$, we obtain the inequality
\begin{align*}
\Big|\frac{\dd}{\dd t}\ov{\rho}_{s,R_1,N}&(\Phi_N(t)(A))|_{t=t_0} \Big|\\ \leq & 
%\int_{\Pi_N^{\perp}\Phi_N(t_0)(A)}d\mu_{s,N}^{\perp}(u)\int_{\Pi_N\Phi_N(t_0)(A)}Q_{s,N}(u)\chi_{R_1}(\|u\|_{H^{\sigma}})\frac{1}{\mathcal{Z}_N}\e^{-E_{s,N}(u)}\prod_{|k|\leq N}d\widehat{u}_k
\int_{\Phi_N(t_0)(A)}
\chi_{R_1}(\|u\|_{H^{\sigma}})
\frac{1}{\mathcal{Z}_N}
|Q_{s,N}(u)|\,\e^{-E_{s,N}(\Pi_N u)}\prod_{|k|\leq N}\dd\widehat{u}_k \,\dd\mu_{s,N}^{\perp}(u)
\\
=
&\int_{\Phi_N(t_0)(A)}|Q_{s,N}(u)|\dd\ov{\rho}_{s,R_1,N}(u)\\
\leq &\|Q_{s,N}(u)\|_{L^p(\dd\ov{\rho}_{s,R_1,N})}\cdot \ov{\rho}_{s,R_1,N}(\Phi_N(t_0)(A))^{1-\frac{1}{p}}.
\end{align*}
Thanks to the last assertion of Proposition \ref{energyestimate},  the function 
$$ F(t):=\ov{\rho}_{s,R_1,N}(\Phi_N(t)(A)),
$$
satisfies the inequality
$$ 
F'(t)\leq C_{s,R}\cdot p^{\beta} F(t)^{1-\frac{1}{p}},\quad \forall |t|\leq T,\; \quad p<\infty 
$$
Integrating the differential inequality above, we obtain that
$$ F(t)\leq \big(F(0)^{\frac{1}{p}}+C_{s,R}\cdot p^{-(1-\beta)}t\big)^p\leq F(0)\e^{C_{R,s}tp^{\beta}F(0)^{-\frac{1}{p}}}.
$$
Without loss of generality, we assume that $F(0)=\overline{\rho}_{s,R_1,N}(A)>0$. Thanks to Proposition \ref{weightedmeasure},  $F(0)\leq M_{R,s}$ for some constant $M_{R,s}>0$. If $\log F(0)\geq 1$, we simply take $p=2$ to get
$$ F(t)\leq F(0)^{1-\epsilon_0}M_{R,s}^{\epsilon_0}\cdot  \e^{C_{R,s}t2^{\beta}\e^{-\frac{1}{2}} }
$$ 
for any $\epsilon_0\in(0,1)$.
Otherwise $\log F(0)<1$, we choose
$$ p=2+\log\Big(\frac{1}{F(0)}\Big).
$$
So we conclude that there exists $\epsilon_0\in(0,1)$, such that
$$ F(t)\leq C_{R,s,T}F(0)^{1-\epsilon_0},\quad \forall |t|\leq T,
$$
namely
$$ \ov{\rho}_{s,R_1,N}(\Phi_N(t)(A))\leq C_{R_1,s,T} \ov{\rho}_{s,R_1,N}(A)^{1-\epsilon_0},\quad \forall |t|\leq T.
$$
Finally,
as $\Phi_N(t)(A)\subset B_{R_1/2}^{H^{\sigma}}$,
$$ \mu_s(\Phi_N(t)(A))=\int_{\Phi_N(t)(A)}\chi_{R_1}(\|u\|_{H^{\sigma}})\cdot \e^{R_{s,N}(u)}\dd\ov{\rho}_{s,R,N}(u).
$$
By Cauchy-Schwarz and the $L^2$-integrability of $\chi_{R_1}(\|u\|_{H^{\sigma}})\cdot\e^{R_{s,N}(u)}$ with respect to $\dd\mu_s$ (Proposition \ref{weightedmeasure}), 
\begin{align}\label{measurepower1}   \mu_s(\Phi_N(t)(A))\leq &\|\chi_{R_1}(\|u\|_{H^{\sigma}})\e^{R_{s,N}(u)}\|_{L^2(\dd\ov{\rho}_{s,R,N})}\cdot \ov{\rho}_{s,R,N}(\Phi(t)(A))^{\frac{1}{2}}\notag \\
\leq &\|\chi_{R_1}(\|u\|_{H^{\sigma}})\e^{|R_{s,N}(u)|}\|_{L^1(\dd\mu_s)}^{\frac{1}{2}}\cdot \sqrt{C_{R_1,s,T}}\cdot \ov{\rho}_{s,R_1,N}(A)^{\frac{1-\epsilon_0}{2}} \notag \\
\leq &C_{R_1,s,T}'\cdot \ov{\rho}_{s,R_1,N}(A)^{\frac{1-\epsilon_0}{2}},
\end{align}
Again, since $A\subset B_R^{H^{\sigma}}\subset B_{R_1}^{H^{\sigma}}$, 
$$ \ov{\rho}_{s,R_1,N}(A)\leq \|\chi_{R_1}(\|u\|_{H^{\sigma}})\;\e^{R_{s,N}(u)}\|_{L^2(\dd\mu_s)}\mu_s(A)^{\frac{1}{2}}\leq C_{R_1,s}\cdot \mu_s(A)^{\frac{1}{2}}.
$$
Plugging into \eqref{measurepower1}, we complete the proof of Lemma \ref{Quasi-N}.
\end{proof}
%%%%%%%%%%%%%
\begin{proof}[Proof of Theorem \ref{thm:main}]
Let $T>0$, we first show that for any compact set $K\subset B_{R}^{H^{\sigma}}$, $|t|\leq T$,
$$ \mu_s(\Phi(t)(K))\leq C_{s,2R,T}\cdot \mu_s(K)^{\frac{1-\epsilon_0}{4}}.
$$
%where $\epsilon_0>0,$ $R_1:=\Lambda(\Lambda(R,T),T)$ are as in the proof of Lemma \ref{Quasi-N}. 
Indeed, applying the approximation theory (Proposition \ref{approximation}) to the set $\Phi(t)(K)$, which is compact, for any small $\epsilon>0$, there exists $N_0\in\N$, such that for all $N\geq N_0$, $$\Phi_N(-t)(\Phi(t)(K))\subset \Phi(-t)(\Phi(t)(K))+B_{\epsilon}^{H^{\sigma}}=K+B_{\epsilon}^{H^{\sigma}},$$  thus $\Phi(t)(K)\subset \Phi_N(t)(K+B_{\epsilon}^{H^{\sigma}})$, consequently,
\begin{align*}
 \mu_s(\Phi(t)(K))\leq \mu_s(\Phi_N(t)(K+B_{\epsilon}^{H^{\sigma}}) ).
\end{align*} 
Since for small $\epsilon>0$, $K+B_{\epsilon}^{H^{\sigma}}\subset B_{2R}^{H^{\sigma}}$,  Lemma \ref{Quasi-N} implies that 
 \begin{align}\label{measure:ineq}
  \mu_s(\Phi(t)(K))\leq \mu_s(\Phi_N(t)(K+B_{\epsilon}^{H^{\sigma}} ) )\leq C_{s,2R,T}\cdot\mu_s(K+B_{\epsilon}^{H^{\sigma}})^{\frac{1-\epsilon_0}{4}}.
 \end{align}
 We are going to take the limit $\epsilon\rightarrow 0$ in the inequality above, using the fact that $\mu_s$ is regular. Before doing that, we have to show that for any open set $G\supset K$, there exists $\epsilon>0$, such that
 $$ G\supset K+B_{\epsilon}^{H^{\sigma}}.
 $$
 Since  $K$ is compact, for any open set $G\supset K$, there exist finitely many balls $B_1,\cdots B_m$ of $H^{\sigma}$ such that 
  $$ K\subset \bigcup_{j=1}^mB_j\subset \bigcup_{j=1}^m2B_j\subset G, 
  $$
  where $2B_j$ is the ball with the same center as $B_j$ and with radius twice of $B_j$.
  In particular, there exists $\epsilon_1>0$, such that for all $0<\epsilon<\epsilon_1$,
  $$ K+B_{\epsilon}^{H^{\sigma}}\subset G.
  $$ 
  To see this, we take $\epsilon_1<\frac{1}{4}\min\{\mathrm{radius}(B_j): j=1,\cdots,m \}$. Then for any $u\in K+B_{\epsilon_1}^{\sigma}$, there exists $u_0\in K$, such that $\|u-u_0\|_{H^{\sigma}}<\epsilon_1$. As $K$ is covered by $B_j$, there is a ball, say $B_{1}$ with center $u_1$, such that $\|u_0-u_1\|_{H^{\sigma}}<\mathrm{radius}(B_1)$. Hence $u\in 2B_1\subset G$.

 Recall that Gaussian measures are regular, namely, for any Borel set $A$
 \begin{align*}  \mu_s(A)=&\inf\{\mu_s(G):\; G\supset A,\; G \text{ open in } H^{\sigma} \}\\
 	=&\sup\{\mu_s(F):\; F\subset A,\; F \text{ compact and Borel in } H^{\sigma} \}
 \end{align*}
 we can take $\epsilon\rightarrow 0$ on the right hand side of \eqref{measure:ineq} to obtain the estimate
\begin{align} \label{measure:ineq'}
\mu_s(\Phi(t)(K))\leq  C_{s,2R,T}
 \cdot  \mu_s(K)^{\frac{1-\epsilon_0}{4}},
\end{align}
as desired.

Finally we assume that $A\subset B_{R}^{H^{\sigma}}$ is an arbitrary Borel set. Since $\Phi(t)$ is a continuous bijection on $H^{\sigma}(\T^3)$, $\Phi(t)(A)$ is also a Borel set (view $\Phi(t)(A)=(\Phi(-t))^{-1}(A)$). Thus there exists a sequence of compact sets $K_n(t)\subset \Phi(t)(A)$, such that
$$ \mu_s(\Phi(t)(A))=\lim_{n\rightarrow\infty}\mu_s(K_n(t)).
$$ 
For fixed $|t|\leq T$, set $F_n(t)=\Phi(-t)(K_n(t))$, by the bijectivity of $\Phi(t)$, $K_n(t)=\Phi(t)(F_n(t))$. Since $F_n(t)$ are also compact (Proposition \ref{approximation}), we deduce that
$$ \mu_s(K_n(t))=\mu_s(\Phi(t)(F_n(t)))\leq C_{R_1,s,T}\cdot \mu_s(F_n(t))^{\frac{1-\epsilon_0}{4}}.
$$
Again from $K_n(t)=\Phi(t)(F_n(t)
)\subset \Phi(t)(A)$ and the bijectivity, $F_n(t)\subset A$, thus 
$$ \mu_s(K_n(t))\leq C_{R_1,s,T}\cdot \mu_s(A)^{\frac{1-\epsilon_0}{4}}.
$$
Letting $n\rightarrow\infty$, we deduce that
$$ \mu_s(\Phi(t)(A))\leq C_{R_1,s,T}\cdot \mu_s(A)^{\frac{1-\epsilon_0}{4}}.
$$
In particular, if $\mu_s(A)=0$, we must have $\mu_s(\Phi(t)(A))=0$. This proves the quasi-invariance property of $\mu_s$ along the flow $\Phi(t)$.
\end{proof}
%%%%%%%%%%%%%%%%%%%%%%%%%%%%%%%%%%%%%%%%%%%%%%%%%%%%%
%%%%%%%%%%%%%%%%%%%%%%%%%%%%%%%%%%%%%%%%%%%%%%%%%%%%%%%
%%%%%%%%%%%%%%%%%%%%%%%%%%%%%%%%%%%%
%%%%%%%%%%%%%%%%%%%%%%%%%%%%%%%%%%%%%
\section{Preliminaries for the energy estimates}\label{prelim}
In this section, we summarize several frequently used preliminary results as well as some notations. 
\subsection{Deterministic tools}	
For a given set of frequencies $k_1,k_2,\cdots,k_m$, we denote $k_{(1)},k_{(2)},\cdots,k_{(m)}$ a non-increasing rearrangement such that
$$ |k_{(1)}|\geq |k_{(2)}|\geq \cdots\geq |k_{(m)}|.
$$
Similarly, for a given set of dyadic integers $N_1,N_2,\cdots,N_m$, we denote $N_{(1)},N_{(2)},\cdots,N_{(m)}$ a non-increasing rearrangement such that
$$ N_{(1)}\geq N_{(2)}\geq\cdots\geq N_{(m)}.
$$
We have the following estimate on the the function $\psi_s$ which measures the lack of conservation of $H^s$ based quantities. 
\begin{lem}\label{monto_gordo}
Set
\[
\psi_{2s}(\vec{k})=\sum_{j=1}^{6}(-1)^{j-1}|k_j|^{2s},\quad 
\Omega(\vec{k})=\sum_{j=1}^{6}(-1)^{j-1}|k_j|^2.
\]
Then for $k_1-k_2+k_3-k_4+k_5-k_6=0$,
$$
 |\psi_{2s}(\vec{k})|\lesssim |k_{(1)}|^{2s-2}[|\Omega(\vec{k})|+|k_{(3)}|^2].
 $$
 \end{lem}
 \begin{proof}
 We can suppose that $|k_{(3)}|\ll |k_{(2)}|$, otherwise the estimate reduces to the straightforward bound
 $
  |\psi_{2s}(\vec{k})|\lesssim |k_{(1)}|^{2s}.
 $
 Essentially, there are two different cases : $k_{(1)}=k_1$,  $k_{(2)}=k_2$ and  $k_{(1)}=k_1$,  $k_{(2)}=k_3$. In the second case we can again use the bound  
 $
  |\psi_{2s}(\vec{k})|\lesssim |k_{(1)}|^{2s}.
 $
 Let is now suppose that $k_{(1)}=k_1$,  $k_{(2)}=k_2$.  By the mean-value theorem, 
$$ 
\big| |k_1|^{2s}-|k_2|^{2s}\big| \lesssim |k_{(1)}|^{2(s-1)}\big| |k_1|^{2}-|k_2|^{2}\big| \lesssim |k_{(1)}|^{2s-2}[|\Omega(\vec{k})|+|k_{(3)}|^2].
$$
This completes the proof of Lemma \ref{monto_gordo}.
\end{proof}
%%%%
For linear constraints, we denote
\begin{align}\label{not:linearconstraint} \fk{h}_{k_1^{\iota_1}k_2^{\iota_2}\cdots k_{m}^{\iota_m}}:=\mathbf{1}_{\iota_1k_1+\iota_2k_2+\cdots+\iota_mk_m=0},
\end{align}
where $\iota_j\in\{+1,-1\}$, identified also as $\{+,-\}$, the signature of frequencies $k_1,\cdots,k_m$. For example, $$\fk{h}_{k_1^+k_2^-k_3^+k_4^-k_5^+k_6^-}=\mathbf{1}_{k_1-k_2+k_3-k_4+k_5-k_6=0}.$$	
Sometimes in the computation we exploit this notation to shorten the expression. 

We will frequently use the following elementary counting bound:
	\begin{lem}\label{threevectorcounting} 
		Assume that $n\geq 2$ and given dyadic numbers $N_1, N_2,\cdots N_n$. Then uniformly in $K\in\Z^3$, $\kappa\in\R$ and $\iota_j\in\{+1,-1\}$, we have
		$$ \sum_{\substack{k_1,k_2,\cdots,k_n\\
	    \iota_ik_i+\iota_jk_j\neq 0,\forall i\neq j	
	 }   }\mathbf{1}_{\iota_1k_1+\iota_2k_2+\cdots+\iota_nk_n=K }\cdot  
		 \mathbf{1}_{  \iota_{1} |k_1|^2+\cdots+\iota_{n }|k_n|^2=\kappa  } 
 \Big(\prod_{j=1}^n \mathbf{1}_{|k_j|\sim N_j }\Big)\lesssim N_{(2)}^2\prod_{j=3}^nN_{(j)}^3,
		$$
where we adopt the convention that when $n=2$, the bound on the right hand side is $N_{(2)}^2$.
\end{lem}
\begin{rem}
The counting bound stated here is very rough but it already fits our need. By using some arithmetic, one can improve it when $n\geq 3$ or $n=2$ and $\iota_1=\iota_2$. 
We refer to Lemma 4.5 of \cite{DNY2} for such an improvement . The estimate of Lemma \ref{threevectorcounting} has the advantage to hold with the same (trivial) proof on a general torus. 
\end{rem}
Next we recall the following conditional Wiener chaos estimate for multi-linear expression of complex Gaussian random variables. In the sequel we adopt the notation $z^{+}=z$ and $z^-=\ov{z}$ for a complex number $z\in\C$. 
	\begin{lem}[Wiener chaos estimate]\label{conditionalWiner} 
		Consider the multi-linear expression of Gaussian:
		$$ F(\omega',\omega)=\sum_{k_{1},\cdots,k_{n}}c_{k_{1},\cdots,k_{n}}(\omega')\cdot \prod_{j=1}^{n}g_{k_{j} }  ^{\iota_{j}}(\omega),
		$$
		where the random variables $c_{k_{1},\cdots,k_{n} }(\omega')$ are independent of complex standard i.i.d. Gaussians $g_{k_j}(\omega) $. Then for any $p\geq 2$, we have
		$$ \|F(\omega',\omega)\|_{L_{\omega}^p}\leq Cp^{\frac{n}{2}}\|F(\omega',\omega) \|_{L_{\omega}^2}.$$
	\end{lem}
We state the Wiener chaos estimate in above form  since later on we will use Lemma \ref{conditionalWiner} for $L^p$ estimates for high-frequency Gaussians conditioning to some $\sigma$-algebra generated by low-frequency Gaussians (see also \cite{STz2} for a statement involving conditional expectation). 
Starting from \cite{Bringmann},  in recent years such conditioned Wiener chaos estimates were extensively use in the field of random dispersive PDE's. 
	%%%%%%%%%%%%%%%%%%%%%%%%%%%%%%%%%%%%%%%%%%%%%%%%%%%%%%%%%%%%%%%%%%%%
%%	%%%%%%%%%%%%%%%%%%%%%%%%%%%%%%%%%%%%%%%%%%%%%%%%%%%%%%%%%%%%%%%%%%%%%%%%%%%%%%%%%%%%%%%%%%%%%%%%%%%%%%%%%%%%%%%%%%%%%%%%%%%%%%%%%%%%%%%%%%%%%%%%%%%%%%%%%%%%%%%%%%%%%%%%%%%%%%%%%%%%%%%%%%%%%%%%%%%%%%%%%%%%%%%%%%%%%%%%%%%%%%%%%%%%%%%%%%%%%%%%%%%%%	
\section{Decomposion of the differential of the modified energy}\label{graficd}	
Recall from \eqref{dmodifiedenergy} that
$$ 
Q_{s,N}(w)=\Im\big(-\frac{1}{6}\mathcal{R}_0(w)+\frac{1}{2}\mathcal{R}_{1}(w)-\frac{1}{2}\mathcal{R}_{2}(w)\big), 
$$
where
\begin{equation}\label{defR0} 
\mathcal{R}_0(w):=\sum_{k_1-k_2+\cdots-k_{6}=0  }
		\chi\Big(\frac{\Omega(\vec{k})}{\lambda(\vec{k})^{\delta_0}}\Big)
		\psi_{2s}(\vec{k})w_{k_1}\ov{w}_{k_2}\cdots \ov{w}_{k_{6}},
	\end{equation}
	\begin{align}\label{def:R11} 
		\mathcal{R}_{1}(w):=\sum_{\substack{k_1-k_2+\cdots-k_{6}=0\\
	k_1=p_1-p_2+p_3-p_4+p_5	
	 }} &\Big(1-	\chi\Big(\frac{\Omega(\vec{k})}{\lambda(\vec{k})^{\delta_0}}\Big)\Big)
		\frac{\psi_{2s}(\vec{k})}{\Omega(\vec{k})} 
		\chi_N(k_1)^2
		w_{p_1}\ov{w}_{p_2}\cdots w_{p_{5}}\ov{w}_{k_2}\cdots \ov{w}_{k_6}
	\end{align}
	and
	\begin{align}\label{defR12} 
		\mathcal{R}_{2}(w):=\sum_{\substack{ k_1-k_2+\cdots -k_{6}=0 \\
	k_2=q_1-q_2+q_3-q_4+q_5	
	} }
		\Big(1-	\chi\Big(\frac{\Omega(\vec{k})}{\lambda(\vec{k})^{\delta_0}}\Big)\Big)	\frac{\psi_{2s}(\vec{k})}{\Omega(\vec{k})} \chi_N(k_2)^2w_{k_1}\ov{w}_{q_1}\cdots \ov{w}_{q_{5}}w_{k_3}\cdots \ov{w}_{k_{6}}.
	\end{align}
	Comparing to the estimate for $\mathcal{R}_0(w)$, the major difficulty in estimating $\mathcal{R}_1(w),\mathcal{R}_2(w)$ is the existence of pairing contributions between different generations ($w_{k_j}$ and $w_{p_j}$ (or  $w_{k_j}$ and $w_{q_j}$)).
Roughly speaking, the singular pairing contributions in $\mathcal{R}_{1}(v)$ are (up to symmetry)
	\begin{itemize}
		\item $|k_1|\sim |k_2|\gg |k_3|+|k_4|+|k_5|+|k_6|$, $|k_1|\sim |k_2|\gg |p_2|+|p_3|+|p_4|+|p_5|$ and $p_1=k_2$;
		\item $|k_1|\sim |k_3|\gg |k_2|+|k_4|+|k_5|+|k_6|, |k_1|\sim |k_3|\gg |p_1|+|p_3|+|p_4|+|p_5|$ and $p_2=k_3$.
	\end{itemize}
Now we identify these pairing contributions precisely:
	\begin{align}\label{defLambda11} 
		\Lambda_{1,1}:=\Big\{(p_1,\cdots,p_5,k_2,\cdots,k_6):&\sum_{j=1}^5(-1)^{j-1}p_j+\sum_{i=2}^6(-1)^{i-1}k_i=0,\notag \\
		& k_{2}=p_{1},\ \sum_{i\in\{3,4,5,6\} }|k_{i}|\leq |k_1|^{\theta}+ |k_{2}|^{\theta},\; \sum_{j\in\{2,3,4,5\} }|p_{j}|\leq |k_1|^{\theta}+|k_{2}|^{\theta}  \Big\}
	\end{align}
	and
	\begin{align}\label{defLambda12} 
		\Lambda_{1,2}:=\Big\{(p_1,\cdots,p_5,k_2,\cdots,k_6):&\sum_{j=1}^5(-1)^{j-1}p_j+\sum_{i=2}^6(-1)^{i-1}k_i=0, \notag \\
		& k_{3}=p_{2},\; \sum_{i\in\{2,4,5,6\} }|k_{i}|\leq |k_1|^{\theta}+ |k_{3}|^{\theta},\; \sum_{j\in\{1,3,4,5\} }|p_{j}|\leq |k_1|^{\theta}+ |k_{3}|^{\theta}  \Big\},
	\end{align}
	where $0<\theta<\frac{\delta_0}{2}<\frac{1}{3}$ is close to $\frac{1}{3}$. We define correspondingly
	\begin{align}\label{defS11}
		&\mathcal{S}_{1,1}(w):=\sum_{\Lambda_{1,1} }\chi_N(k_1)^2|w_{k_2}|^2\Big(1-	\chi\Big(\frac{\Omega(\vec{k})}{\lambda(\vec{k})^{\delta_0}}\Big)\Big)\frac{\psi_{2s}(\vec{k})}{\Omega(\vec{k})} w_{k_3}\ov{w}_{k_4}w_{k_5}\ov{w}_{k_6}\cdot 
		\ov{w}_{p_2}w_{p_3}\ov{w}_{p_4} w_{p_5}
	\end{align}
	and
	\begin{align}\label{defS12}
		&\mathcal{S}_{1,2}(w):=\sum_{\Lambda_{1,2}}\chi_N(k_1)^2|w_{k_3}|^2
		\Big(1-	\chi\Big(\frac{\Omega(\vec{k})}{\lambda(\vec{k})^{\delta_0}}\Big)\Big)
		\frac{\psi_{2s}(\vec{k})}{\Omega(\vec{k})}\ov{w}_{k_2}\ov{w}_{k_4}w_{k_5}\ov{w}_{k_6}\cdot w_{p_1}w_{p_3}\ov{w}_{p_4}w_{p_5}.
	\end{align}

Similarly, the pairing contributions in $\mathcal{R}_{2}$ are (up to symmetry)
	\begin{align}\label{defS21}
		&\mathcal{S}_{2,1}(w):=\sum_{\Lambda_{2,1} }\chi_N(k_2)^2|w_{k_1}|^2\Big(1-	\chi\Big(\frac{\Omega(\vec{k})}{\lambda(\vec{k})^{\delta_0}}\Big)\Big)\frac{\psi_{2s}(\vec{k})}{\Omega(\vec{k})} w_{k_3}\ov{w}_{k_4}w_{k_5}\ov{w}_{k_6}\cdot
		\ov{w}_{q_3}w_{q_2} \ov{w}_{q_5}w_{q_4}
	\end{align}
	and
	\begin{align}\label{defS22}
		&\mathcal{S}_{2,2}(w):=\sum_{\Lambda_{2,2}}\chi_N(k_2)^2|w_{k_4}|^2\Big(1-	\chi\Big(\frac{\Omega(\vec{k})}{\lambda(\vec{k})^{\delta_0}}\Big)\Big)\frac{\psi_{2s}(\vec{k})}{\Omega(\vec{k})}w_{k_1}w_{k_3}w_{k_5}\ov{w}_{k_6} \cdot \ov{w}_{q_1}\ov{w}_{q_3}\ov{w}_{q_5}w_{q_4}.
	\end{align}
	where
	\begin{align}\label{defLambda21} 
		\Lambda_{2,1}:=\Big\{(k_1,q_1,\cdots,q_5,k_3,\cdots,k_6):&\sum_{j=1}^5(-1)^{j}q_j+\sum_{i\in\{1,3,4,5,6\}}(-1)^{i-1}k_i=0,\notag \\
		& k_{1}=q_{1},\; \sum_{i\in\{3,4,5,6\} }|k_{i}|\leq |k_{1}|^{\theta}+|k_2|^{\theta},\; \sum_{j\in\{2,3,4,5\} }|q_{j}|\leq |k_{1}|^{\theta}+|k_2|^{\theta}  \Big\},
	\end{align}
	and
	\begin{align}\label{defLambda22} 
		\Lambda_{2,2}:=\Big\{(k_1,q_1,\cdots,q_5,k_3,\cdots,k_6):&\sum_{j=1}^5(-1)^{j}q_j+\sum_{i\in\{1,3,4,5,6 \}}(-1)^{i-1}k_i=0, \notag \\
		& k_{4}=q_{2},\; \sum_{i\in\{1,3,5,6\} }|k_{i}|\leq |k_2|^{\theta}+ |k_{4}|^{\theta},\; \sum_{j\in\{1,3,4,5\} }|q_{j}|\leq |k_2|^{\theta}+ |k_{4}|^{\theta}  \Big\}.
	\end{align}
	
	\begin{tikzpicture}[scale=1]
		\tikzset{level distance=25pt};
		\tikzset{sibling distance=20pt};
		[
		level 1/.style={sibling distance=10pt},
		%  level 1/.style={level post sep=0.5cm},
		level 2/.style={level post sep=10pt}, level distance=10pt,
		%	level 3/.style={sibling distance=2em}, level distance=1cm
		] 
		\draw (1,0) circle (3.0pt);
		\draw (0,0).. controls (0.3,0.3) and (0.8,0.3) .. (0.95,0);
		\node at (0,0) [left] {};
		\node at (1,0) [right] {$k_6$};
		\coordinate (root) {}  [fill]  circle (3.0pt)
		child {  
			[fill]  circle (3.0pt)
			child {  [fill]  circle (3.0pt)
				node[left]{$p_1$} 
				%		child {}
				%		child {}
				%	edge from parent
				%	node[left] {a}
			} 
			child { circle (3.0pt) 
			}
			child { [fill] circle (3.0pt)         
			}
			child {    circle (3.0pt) 
				%	child {}
				%	child {}
			}
			child {  [fill]  circle (3.0pt) 
				%	child {}
				%	child {}
			}
			%	child {
			%	}
		}
		child {    circle (3.0pt) 
			%	child {}
			%	child {}
		}
		child {  [fill]  circle (3.0pt) 
			%	child {}
			%	child {}
		}
		child {    circle (3.0pt) 
			%	child {}
			%	child {}
		}
		child {  [fill]  circle (3.0pt) 
			%	child {}
			%	child {}
		}
		;
		%	\node at (root)[right]{$\mathfrak{r}$};
		\node at (root-1)[left] {$k_1$};
		\node at (root-2)[right] {$k_2$};
		\node at (0,-3) { $\mathcal{S}_{1,1}$: $p_1$ is paired with $k_2$};
	\end{tikzpicture}
	\hspace{2.5cm} %%%%%想將兩個tikzpicture并列，第一個的\end{tikzpicture}命令與第二個\begin{tikzpicture}命令之間不能有換行，只能加入一個行距的命令
	\begin{tikzpicture}[scale=1]
		\tikzset{level distance=25pt};
		\tikzset{sibling distance=20pt}; 
		[
		level 1/.style={sibling distance=2em},
		level 2/.style={sibling distance=0.5em}, level distance=0.5cm,
		%	level 3/.style={sibling distance=2em}, level distance=1cm
		] 
		\draw (1,0) circle (3.0pt);
		\draw (0,0).. controls (0.3,0.3) and (0.8,0.3) .. (0.95,0);
		\node at (0,0) [left] {};
		\node at (1,0) [right] {$k_6$};
		\coordinate (root) {}  [fill]  circle (3.0pt)
		child {  [fill]  circle (3.0pt)
			child {  [fill]  circle (3.0pt)
				%		child {}
				%		child {}
				%	edge from parent
				%	node[left] {a}
			}
			child { circle (3.0pt) 
				node[left]{$p_2$}
			}
			child { [fill] circle (3.0pt) 
			}
			child {  circle (3.0pt)
			}
			child { [fill] circle (3.0pt)         
			}
			%	child {
			%	}
		}
		child {    circle (3.0pt) 
			%	child {}
			%	child {}
		}
		child { [fill] circle (3.0pt)
		}
		child {  circle (3.0pt)
		}
		child {  [fill]  circle (3.0pt) 
			%	child {}
			%	child {}
		}
		;
		%	\node at (root)[right]{$\mathfrak{r}$};
		\node at (root-1)[left] {$k_1$};
		\node at (root-3)[right] {$k_3$};
		\node at (0,-3) { $\mathcal{S}_{1,2}$: $p_2$ is paired with $k_3$};	
	\end{tikzpicture}

	\begin{tikzpicture}[scale=1]
		\tikzset{level distance=25pt};
		\tikzset{sibling distance=20pt}; 
		[
		level 1/.style={sibling distance=3em},
		level 2/.style={sibling distance=1.5em}, level distance=1cm,
		%	level 3/.style={sibling distance=2em}, level distance=1cm
		] 
		\draw (1,0) circle (3.0pt);
		\draw (0,0).. controls (0.3,0.3) and (0.8,0.3) .. (0.95,0);
		\node at (0,0) [left] {};
		\node at (1,0) [right] {$k_6$};
		\coordinate (root) {}  [fill]  circle (3.0pt)
		child {  [fill]  circle (3.0pt) 
			node [left] {$k_1$}
			%	child {}
			%	child {}
		}
		child {    circle (3.0pt)
			child {    circle (3.0pt)
				node [right] {$q_1$}
				%		node[left]{$\mathfrak{l}''$}
				%		child {}
				%		child {}
				%	edge from parent
				%	node[left] {a}
			}
			child { [fill] circle (3.0pt) 
			}
			child {  circle (3.0pt)         
			}
			child { [fill] circle (3.0pt)
			}
			child {  circle (3.0pt) 
			}
			%	child {
			%	}
		}
		child { [fill] circle (3.0pt)
		}
		child {  circle (3.0pt) 
		}
		child {  [fill]  circle (3.0pt) 
			%	child {}
			%	child {}
		}
		;
		%	\node at (root)[right]{$\mathfrak{r}$};
		%	\node at (root-1)[left] {$\mathfrak{n}_0$};
		%	\node at (root-2)[right] {$\mathfrak{l}'$};
		\node at (root-2)[left] {$k_2$};
		\node at (0,-3) { $\mathcal{S}_{2,1}$: $q_1$ is paired with $k_1$};
	\end{tikzpicture}
	\hspace{2.5cm} %%%%%想將兩個tikzpicture并列，第一個的\end{tikzpicture}命令與第二個\begin{tikzpicture}命令之間不能有換行，只能加入一個行距的命令
	\begin{tikzpicture}[scale=1] 
		\tikzset{level distance=25pt};
		\tikzset{sibling distance=20pt}; 
		[
		level 1/.style={sibling distance=3em},
		level 2/.style={sibling distance=1.5em}, level distance=1cm,
		%	level 3/.style={sibling distance=2em}, level distance=1cm
		] 
		\draw (1,0) circle (3.0pt);
		\draw (0,0).. controls (0.3,0.3) and (0.8,0.3) .. (0.95,0);
		\node at (0,0) [left] {};
		\node at (1,0) [right] {$k_6$};
		\coordinate (root) {}  [fill]  circle (3.0pt)
		child {  [fill]  circle (3.0pt) 
			%	child {}
			%	child {}
		}
		child {    circle (3.0pt)
			child {    circle (3.0pt)
				%		node[left]{$\mathfrak{l}''$}
				%		child {}
				%		child {}
				%	edge from parent
				%	node[left] {a}
			}
			child { [fill] circle (3.0pt) 
				node[left] {$q_2$}
			}
			child {  circle (3.0pt)         
			}
			child { [fill] circle (3.0pt)
			}
			child {  circle (3.0pt) 
			}
			%	child {
			%	}
		}
		child { [fill] circle (3.0pt)
		}
		child {  circle (3.0pt) 
			node[right] {$k_4$}
		}
		child {  [fill]  circle (3.0pt) 
			%	child {}
			%	child {}
		}
		;
		%	\node at (root)[right]{$\mathfrak{r}$};
		%	\node at (root-1)[left] {$\mathfrak{n}_0$};
		%	\node at (root-2)[right] {$\mathfrak{l}'$};
		\node at (root-2)[left] {$k_2$};
		\node at (0,-3) { $\mathcal{S}_{2,2}$: $q_2$ is paired with $k_4$};
	\end{tikzpicture} 

\begin{center}
	\begin{minipage}{0.8\textwidth}  % Adjust the width as needed
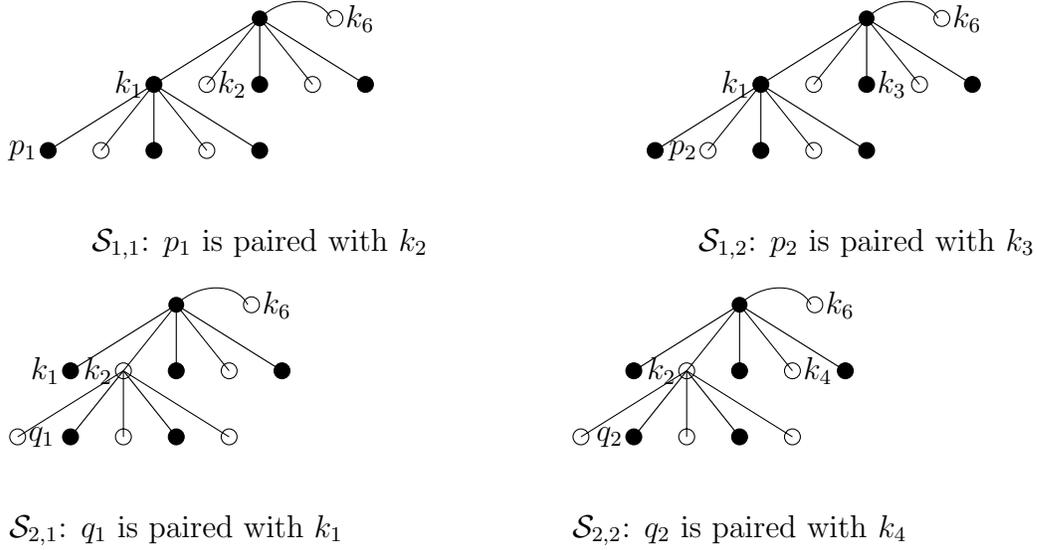
		\textbf{Figure 1:} Configurations of trees with two generations. The filled nodes have a $+$ sign, indicating the absence of a conjugate bar of the Fourier mode, while the unfilled nodes have a $-$ sign, signifying the presence of a conjugate bar of the Fourier mode. Each parent node generates five children nodes, displaying alternating signs. Up to symmetry, there are four distinct pairings across the two generations, giving the contributions $\mathcal{S}_{i,j}, i,j=1,2$.
	\end{minipage}
\end{center}
\vspace{0.3cm}

	By symmetry, we have
	\begin{align}\label{defR13} 
		\mathcal{R}_1(w)=9\mathcal{S}_{1,1}(w)+4\mathcal{S}_{1,2}(w)+\mathcal{R}_{1,3}(w),
	\end{align}
	and
	\begin{align}\label{defR23} 
		\mathcal{R}_{2}(w)=9\mathcal{S}_{2,1}(w)+4\mathcal{S}_{2,2}(w)+\mathcal{R}_{2,3}(w),
	\end{align}
	where in the expression of remainders $\mathcal{R}_{1,3}(w)$ 
	we have either $|k_{(3)}|\gtrsim |k_{(1)}|^{\theta}$ or $|k_{(3)}|\lesssim |k_{(1)}|^{\theta}$ and the dominating frequencies are either non-paired or paired within the same generation.	
	Here $k_{(1)}\cdots,k_{(10)}$ is a rearrangement of leaves $p_1,p_2,p_3,p_4,p_5,k_2,k_3,k_4,k_5,k_6$ such that $|k_{(1)}|\geq |k_{(2)}|\geq \cdots \geq |k_{(10)}|$.
	We define similarly the remainder $\mathcal{R}_{2,3}(w)$.
	More precisely, we distinguish three different types in $\mathcal{R}_{1,3}(w)$ (and  in $\mathcal{R}_{2,3}(w)$) with the corresponding constraints in the sum
	$$ \sum_{\substack{k_1-k_2+k_3-k_4+k_5-k_6=0\\ k_1=p_1-p_2+p_3-p_4+p_5\\}}(\cdots):
	$$
	\begin{itemize}
		\item Type A: $\sum_{j=3}^{10}|k_{(j)}|>|k_{(1)}|^{\theta}+|k_{(2)}|^{\theta}$.
		\item Type B: $\sum_{j=3}^{10}|k_{(j)}|\leq |k_{(1)}|^{\theta}+|k_{(2)}|^{\theta}$ and $\{k_{(1)},k_{(2)}\}\subset \{k_2,k_3,k_4,k_5,k_6\}$ or $\{k_{(1)},k_{(2)}\}\subset \{p_1,p_2,p_3,p_4,p_5\}$.
		\item Type C: $\sum_{j=3}^{10}|k_{(j)}|\leq |k_{(1)}|^{\theta}+|k_{(2)}|^{\theta}$, $k_{(1)}\neq k_{(2)}$ and $$k_{(1)}\in\{k_2,k_3,k_4,k_5,k_6\},\quad k_{(2)}\in\{p_1,p_2,p_3,p_4,p_5\}$$ or $$k_{(2)}\in\{k_2,k_3,k_4,k_5,k_6\},\quad k_{(1)}\in\{p_1,p _2,p_3,p_4,p_5\}.
		$$
	\end{itemize}
	Recall that we have to estimate the $L^p(d\mu_s)$ norm of
	$$Q_{s,N}(w)=-\frac{1}{6}\Im\mathcal{R}_0(w)+\frac{1}{2}\Im\mathcal{R}_1(w)-\frac{1}{2}\Im\mathcal{R}_2(w).
	$$
	In Section \ref{energyI}, we estimate the first generation contribution $\mathcal{R}_0(w)$, and in Section \ref{Section:singular}, we estimate the pairing contributions
	$$ \Im(\mathcal{S}_{1,j}-\mathcal{S}_{2,j}),\quad j=1,2. 
	$$ 
	Finally in Section \ref{Section:rest}, we finish the estimate for remainders in the second generation $\mathcal{R}_{1,3}(w),\mathcal{R}_{2,3}(w)$.
%%%%%%%%%%%%%%%%%%%%%%%%%%%%%%%%%%%%%%%%%%%%%%%%%%%%%%%%%%%%%

	\section{Energy estimate I: the first generation}\label{energyI}
Denote
\begin{align}\label{R1} 
	&\mathcal{R}(w):=\sum_{k_1-k_2+\cdots-k_6=0 }\Big(1-	\chi\Big(\frac{\Omega(\vec{k})}{\lambda(\vec{k})^{\delta_0}}\Big)\Big) \frac{\psi_{2s}(\vec{k})}{\Omega(\vec{k})}w_{k_1}\ov{w}_{k_2}\cdots \ov{w}_{k_6},\\ 
	&\mathcal{R}_0(w):=\sum_{k_1-k_2+\cdots-k_6=0}
	\chi\Big(\frac{\Omega(\vec{k})}{\lambda(\vec{k})^{\delta_0}}\Big)
	\psi_{2s}(\vec{k})w_{k_1}\ov{w}_{k_2}\cdots \ov{w}_{k_6}.
\end{align}

\begin{prop}\label{firstgeneration} 
	Assume that $\delta_0<\frac{2}{3}$. There exists $\beta\in(0,1)$, such that for any $R>0$ and $p\in[2,\infty)$ we have
	\begin{align*}
		\|\mathbf{1}_{B_R^{H^{\sigma}}}(w)\mathcal{R}(w)\|_{L^p(\dd\mu_s)}+\|\mathbf{1}_{B_R^{H^{\sigma}}}(w)\mathcal{R}_0(w)\|_{L^p(\dd\mu_s)}\leq C(R)p^{\beta}.
	\end{align*}
\end{prop}
\begin{proof}
Before proceeding to the estimates, we first observe that without loss of generality, we may assume that there is no pairing between frequencies with different signatures in the sum  defining $\mathcal{R}_0(w)$ or $\mathcal{R}(w)$. Indeed, if this is the case, say $k_1=k_2$ are paired, then the resonant function will degenerate to $|k_3|^2-|k_4|^2+|k_5|^2-|k_6|^2$ and the energy weight $\psi_{2s}(\vec{k})$ will degenerate to $|k_3|^{2s}-|k_4|^{2s}+|k_5|^{2s}-|k_6|^{2s}$. Therefore, this pairing contribution in $\mathcal{R}_0(w)$ or $\mathcal{R}(w)$ reduces to some power of $\|w\|_{L^2}^2$ times a similar term with two less degrees of homogeneity\footnote{For over pairing contributions, the over-paired part can be controlled by a power of $\|w\|_{L^2}^2$.}, and the  treatment of such a reduced term is similar (simpler) than $\mathcal{R}_0(w)$ or $\mathcal{R}(w)$. So in the sequel of the proof, we implicitly assume that there is no pairing between frequencies with different signatures in all the sums.
	
$\bullet$ {\bf Estimate of $\mathcal{R}_0(w)$}:
	
Pick $\alpha\in(0,1)$, close enough to $1$, we split $\mathcal{R}_0(w)$ as $\mathrm{I}+\mathrm{II}$, where
$$ \mathrm{I}:=\sum_{\substack{ |k_{(3)}|>|k_{(1)}|^{\alpha} \\k_1-k_2+\cdots-k_6=0} }
\chi\Big(\frac{\Omega(\vec{k})}{\lambda(\vec{k})^{\delta_0}}\Big)\psi_{2s}(\vec{k})w_{k_1}\ov{w}_{k_2}\cdots \ov{w}_{k_6}
$$
and
$$ \mathrm{II}:=\sum_{\substack{ |k_{(3)}|\leq |k_{(1)}|^{\alpha} \\k_1-k_2+\cdots-k_6=0} }
\chi\Big(\frac{\Omega(\vec{k})}{\lambda(\vec{k})^{\delta_0}}\Big)
\psi_{2s}(\vec{k})w_{k_1}\ov{w}_{k_2}\cdots \ov{w}_{k_6}.
$$
To estimate I, we only exploit deterministic analysis. The order of $w_{k_j},\ov{w}_{k_j}$ plays no significant role in the analysis. Therefore,  without loss of generality, we assume that in the sum, $|k_1|\geq |k_2|\geq |k_3|\geq |k_4|\geq |k_5|\geq |k_6|$. Taking the absolute value in the sum, we have
	\begin{align*}
		\mathrm{I}\lesssim & \sum_{\substack{|k_3|>|k_1|^{\alpha}\\
				k_1-k_2+\cdots-k_6=0,
		}} 
		\mathbf{1}_{|\Omega(\vec{k})|\lesssim |k_1|^{\delta_0}}\;\mathbf{1}_{|k_1|\geq|k_2|\geq \cdots\geq |k_6|}\cdot |k_1|^{2s-2}(|k_3|^2+|\Omega(\vec{k})|)|w_{k_1}\cdots w_{k_6}|\\
		\lesssim &\sum_{\substack{N_1\geq N_2\geq\cdots N_6\\ N_3\gtrsim N_1^{\alpha} }}\mathrm{I}_{N_1,\cdots N_6},
	\end{align*}
	where the summations are performed on the dyadic values of $N_1,\cdots N_6$ and
$$ 
\mathrm{I}_{N_1,\cdots N_6}:=\sum_{\substack{|k_3|>|k_1|^{\alpha}\\
		k_1-k_2+\cdots-k_6=0,
}} 
	N_1^{2s-2}N_3^2\cdot \mathbf{1}_{|\Omega(\vec{k})|\lesssim N_1^{\delta_0}}\cdot  \prod_{j=1}^6\mathbf{1}_{|k_j|\sim N_j}|w_{k_j}|,
$$
provided that  $\alpha>\frac{1}{3}$ thanks to the restriction $0<\delta_0<\frac{2}{3}$. Using the Cauchy-Schwarz inequality in the $k_1, k_2$ summations, we can write
$$
\mathrm{I}_{N_1,\cdots N_6}\lesssim N_1^{2s-2}N_3^2 \| P_{N_1} w\|_{L^2} \| P_{N_2} w\|_{L^2}\prod_{j=3}^6\big(\sum_{k_j\in\Z} \mathbf{1}_{|k_j|\sim N_j}|w_{k_j}|\big),
$$
where $P_N$ is the frequency projector to $|k|\sim N$.  Therefore, for  $N_3\gtrsim N_1^{\alpha}$  and $w\in B_R^{H^{\sigma}} $, we have a crude estimate
 $$
\mathrm{I}_{N_1,\cdots N_6}\lesssim  R^6\, N_1^{2s-2}N_3^2 \, N_1^{-2\sigma} N_3^{\frac{3}{2}-\sigma}
\lesssim_R N_1^{\frac{3}{2}+2s-2\sigma}N_1^{-\alpha\sigma}
\,,
$$
which is conclusive as far as $(2+\alpha)\sigma>2s+\frac{3}{2}$.  The last restriction is easily satisfied by taking $\alpha$ close to $1$ as far as $s>\frac{15}{2}$.
	
Next, we estimate $\mathrm{II}$. We decompose II dyadically as
	$\sum_{N_1,\cdots, N_6}\mathrm{II}_{N_1,\cdots,N_6}$, where
	\begin{align*}
		\mathrm{II}_{N_1,\cdots,N_6}:=\sum_{\substack{ |k_{(3)}|\leq |k_{(1)}|^{\alpha} \\k_1-k_2+\cdots-k_6=0} }
		\chi\Big(\frac{\Omega(\vec{k})}{\lambda(\vec{k})^{\delta_0}}\Big)
		\psi_{2s}(\vec{k})w_{k_1}\ov{w}_{k_2}\cdots \ov{w}_{k_6}\prod_{j=1}^6\mathbf{1}_{|k_j|\sim N_j}.
	\end{align*}
	For this contribution, we mainly rely on the Wiener chaos estimates. 
	
	Without loss of generality, we assume that $N_1\sim N_2\sim N_{(1)},N_{(3)}=N_3$ (since the analysis of cases $N_1\sim N_3\sim N_{(1)}, N_2\sim N_{(3)}$ or $N_1\sim N_3\sim N_{(1)}, N_5\sim N_{(3)}$ are similar or simpler). 
	
	Denote $\mathcal{B}_{\ll N_1}$ the $\sigma$-algebra generated by Gaussians $(g_{k}(\omega))_{|k|\leq N_1/100}$. Note that we have the constraint $N_3\leq N_1^{\alpha}$ for some $\alpha<1$, close to 1, and we only need to consider the contribution where $N_1$ is sufficiently large so that $N_1^{\alpha}\ll \frac{N_1}{100}$. Consequently, $\mathbf{1}_{|k_j|}w_{k_j}, j=3,4,5,6$ are all $\mathcal{B}_{\ll N_1}$ measurable and the random function 
	$$
	\sum_{|k|\sim N_1}\frac{g_k(\omega)}{\sqrt{1+|k|^{2s}}}\e^{ik\cdot x}$$ is independent of $\mathcal{B}_{\ll N_1}$, we have
	\begin{align*}
		\|\mathrm{II}_{N_1,\cdots,N_6}\cdot\mathbf{1}_{B_R^{H^{\sigma} }}(w)\|_{L^p(\dd\mu_s)}\leq &\|\mathrm{II}_{N_1,\cdots,N_6}\cdot\mathbf{1}_{
	B_R^{H^{\sigma}}	
	}(\mathbf{P}_{\leq N_1/100}w)\|_{L^p(\dd\mu_s)}\\
		\leq &\big\| \|\mathrm{II}_{N_1,\cdots,N_6}\|_{L^p(\dd\mu_s|\mathcal{B}_{\ll N_1})}\cdot \mathbf{1}_{B_R^{H^{\sigma} }}(\mathbf{P}_{\leq N_1/100}w)\big\|_{L^{\infty}(\dd\mu_s)},
	\end{align*}
	where $\mathbf{P}_{\leq N_1/100}$ is the frequency projection to $|k|\leq N_1/100$ and $L^p(\dd\mu_s|\mathcal{B}_{\leq N_1/100})$ means the $L^p$ norm conditioned to the $\sigma$-algebra $\mathcal{B}_{\leq N_1/100}$.  
	By the conditional Wiener-chaos estimate (see Lemma \ref{conditionalWiner} ), we have
	\begin{align}\label{WienrchaosTermII} 
		\|\mathrm{II}_{N_1,\cdots,N_6}\|_{L^p(\dd\mu_s|\mathcal{B}_{\ll N_1})}\lesssim &\,\,  p\|\mathrm{II}_{N_1\cdots,N_6}\|_{L^2(\dd\mu_s|\mathcal{B}_{\ll N_1})}\notag \\
		\lesssim & \,\,p (N_1N_2)^{-s}\Big(\sum_{\substack{|k_1|\sim N_1\\|k_2|\sim N_2}}\Big|\!\!\!\sum_{\substack{k_3,k_4,k_5,k_6\\
				k_3-k_4+k_5-k_6=k_2-k_1\\ |\Omega(\vec{k})|\lesssim N_1^{\delta_0} }}\!\!\!\!\!\! \psi_{2s}(\vec{k})w_{k_3}\ov{w}_{k_4}w_{k_5}\ov{w}_{k_6}\prod_{j=3}^6\mathbf{1}_{|k_j|\sim N_j}
		\Big|^2 \Big)^{\frac{1}{2}}.
	\end{align}
	By Cauchy-Schwarz,
	\begin{align*}
		&\sum_{\substack{|k_1|\sim N_1\\|k_2|\sim N_2}}\Big|\!\sum_{\substack{k_3,k_4,k_5,k_6\\
				k_3-k_4+k_5-k_6=k_2-k_1\\ |\Omega(\vec{k})|\lesssim N_1^{\delta_0} }}\!\!\!\!\!\! \psi_{2s}(\vec{k})w_{k_3}\ov{w}_{k_4}w_{k_5}\ov{w}_{k_6}\prod_{j=3}^6\mathbf{1}_{|k_j|\sim N_j}
		\Big|^2\\
		\leq &\Big(\sum_{\substack{|k_1|\sim N_1\\|k_2|\sim N_2}}\sum_{\substack{k_3,k_4,k_5,k_6\\
				k_3-k_4+k_5-k_6=k_2-k_1\\ |\Omega(\vec{k})|\lesssim N_1^{\delta_0} }}\!\!\!\!\!\! |\psi_{2s}(\vec{k})|^2|w_{k_6}|^2\prod_{j=3}^6\mathbf{1}_{|k_j|\sim N_j}
		\Big)\\
		&\hspace{3cm}\times\sup_{|k_1|,|k_2|\sim N}\sum_{k_3-k_4+k_5-k_6=k_2-k_1}|w_{k_3}w_{k_4}w_{k_5}|^2\prod_{j=3}^6\mathbf{1}_{|k_j|\sim N_j}.
	\end{align*}
	Since $|\psi_{2s}(\vec{k})|^2\lesssim N_1^{4(s-1)}(N_3^4+|\Omega(\vec{k})|^2)$, the first sum on the right hand-side can be bounded by (below we implicitly insert the constraint $|k_j|\sim N_j$)
	\begin{align*}
		\sum_{|k_6|\sim N_6}|w_{k_6}|^2&\sum_{|\kappa|\lesssim N_1^{\delta_0}}\, \, \sum_{k_1,k_2,k_3,k_4,k_5}N_1^{4(s-1)}(N_3^4+\kappa^2)\,\,\fk{h}_{k^+_1k^-_2k^+_3k^-_4k^+_5k^-_6}\mathbf{1}_{|\Omega(\vec{k})|=\kappa}\\
		\lesssim &N_6^{-2\sigma}\|w\|_{H^{\sigma}}^2N_1^{4(s-1)}\big(N_3^4N_1^{\delta_0}\, N_2^{2}\,(N_3 N_4N_5)^3 +N_1^{3\delta_0}\, N_2^{2}\,(N_3N_4N_5)^3\big)\\
		\lesssim &\|w\|_{H^{\sigma}}^2N_6^{-2\sigma}N_1^{4s-4}(N_3^4N_1^{\delta_0}+N_1^{3\delta_0}) N_2^2 N_3^3 N_4^3N_5^3,
	\end{align*}
where we used Lemma \ref{threevectorcounting}  and recall the notation $\mathfrak{h}_{k^+_1k^-_2k^+_3k^-_4k^+_5k^-_6}$
introduced in \eqref{not:linearconstraint}.
Plugging into \eqref{WienrchaosTermII}, we obtain that
	\begin{align*}
		\|\mathrm{II}_{N_1,\cdots,N_6}\|_{L^p(\dd\mu_s|\mathcal{B}_{\ll N_1})}\lesssim &
		\,\,pN_1^{-1}(N_3^2N_1^{\frac{\delta_0}{2}}+N_1^{\frac{3\delta_0}{2}}) N_3^{\frac 3 2} N_3^{-\sigma}\prod_{j=3}^6\|w_{N_j}\|_{H^{\sigma}}\\
		\lesssim &\,\, p\big(N_1^{\frac{\delta_0}{2}-1}N_3^{\frac{7}{2}-\sigma}+N_1^{\frac{3\delta_0}{2}-1}N_3^{\frac{3}{2}-\sigma}\big)\|w\|_{H^{\sigma}}^4.
	\end{align*} 
	Since $\delta_0<\frac{2}{3}$ and $\sigma>\frac{7}{2}$, the above quantity can be controlled by
	$$ pN_1^{-(1-\frac{3\delta_0}{2})}\|w\|_{H^{\sigma}}^4.
	$$
	Hence
	$$ \|\mathrm{II}_{N_1,\cdots,N_6}\mathbf{1}_{
B_R^{H^{\sigma}}		
}(w)\|_{L^p(\dd\mu_s)}\lesssim pN_{(1)}^{-(1-\frac{3\delta_0}{2})}R^4.
	$$
	Here since we gain a negative power in $N_{(1)}$, by interpolating with the crude deterministic estimate
	$$ |\mathrm{II}_{N_1,\cdots,N_6}\mathbf{1}_{B_R^{H^{\sigma}}	}(w)|\lesssim N_{(1)}^{2s-2\sigma}\|w\|_{H^{\sigma}}^6\leq N_{(1)}^{2(s-\sigma)}R^6,
	$$
	we conclude the estimate for $\mathcal{R}_0(w)$.
	\vspace{0.3cm}
	
	$\bullet$ {\bf Estimate of $\mathcal{R}(w)$}:
	The estimate for $\mathcal{R}(w)$ is similar (simpler) to the estimate for $\mathcal{R}_0(w)$ and we only sketch the proof. Indeed, comparing to the estimate of $\mathcal{R}_0(w)$, the only difference is that the weight $\chi\Big(\frac{\Omega(\vec{k})}{\lambda(\vec{k})^{\delta_0}}\Big)$ is now replaced by 
	$$ \Big(1-\chi\Big(\frac{\Omega(\vec{k})}{\lambda(\vec{k})^{\delta_0}}\Big) \Big)\frac{1}{\Omega(\vec{k})}.
	$$
	We  similarly split $\mathcal{R}(w)$ similarly as $\mathrm{I}'+\mathrm{II}'$,  where
	$$ \mathrm{I}':=\sum_{\substack{ |k_{(3)}|>|k_{(1)}|^{\alpha} \\k_1-k_2+\cdots-k_6=0} }
	\Big(1-\chi\Big(\frac{\Omega(\vec{k})}{\lambda(\vec{k})^{\delta_0}}\Big) \Big)\frac{\psi_{2s}(\vec{k})}{\Omega(\vec{k})}w_{k_1}\ov{w}_{k_2}\cdots \ov{w}_{k_6},
	$$
	$$ \mathrm{II}':=\sum_{\substack{ |k_{(3)}|\leq |k_{(1)}|^{\alpha} \\k_1-k_2+\cdots-k_6=0} }
	\Big(1-\chi\Big(\frac{\Omega(\vec{k})}{\lambda(\vec{k})^{\delta_0}}\Big) \Big)\frac{\psi_{2s}(\vec{k})}{\Omega(\vec{k})}
	w_{k_1}\ov{w}_{k_2}\cdots \ov{w}_{k_6}
	$$
	and we invoke the inequalities  
	$$
\sum_{N_{(1)}^{\delta_0}\lesssim|\kappa|\lesssim N_{(1)}^{2}}\frac{1}{|\kappa|}\lesssim \log(N_{(1)}),\quad
\Big|\frac{\psi_{2s}(\vec{k})}{\Omega(\vec{k})}\Big|\lesssim |k_{(1)}|^{2s-2} (|k_{(3)}|^2+1).
$$
%{\color{blue}{	The estimate of $\mathrm{I}'$ can be realized by the same deterministic analysis as for $\mathrm{I}$ in the estimate of $\mathcal{R}_0(w)$, with the sum $\sum_{|\kappa|\lesssim N_{(1)}^{\delta_0}}$ by
%	$\sum_{N_{(1)}^{\delta}\gtrsim|\kappa|\lesssim N_{(1)}^{2}}\frac{1}{|\kappa|}$, and we will end up with a better bound $$N_{(1)}^{2(s-1-\sigma)-\alpha(\sigma-\frac{5}{2})}\log(N_{(1)})\|w\|_{H_x^{\sigma}}^6
%	$$
%}}
%	when estimating its dyadic component $\mathrm{I}_{N_1,\cdots,N_6}'$. For the estimate of $\mathrm{II}'$, the $L^p(d\mu_s|\mathcal{B}_{\ll N_{(1)}})$ of its dyadic component $\mathrm{II}_{N_1,\cdots,N_6}'$ will be bounded by 
%	$$ p\big(N_{(1)}^{-1+\epsilon}N_{(3)}^{3-\sigma}+N_{(1)}^{\frac{\delta_0}{2}-1+\epsilon}N_{(3)}^{1-\sigma+\epsilon}\big)\|w\|_{H^{\sigma}}^4,
%%	$$
	%which is conclusive within the range of constraints for $\delta_0$ and $\sigma$.
The proof of Proposition~\ref{firstgeneration} is now complete.  
\end{proof}	
%%%%%%%%%%%%%%%%%%%%%%%%%%%%%%%%%%%%%%%%%%%%%%%%%%%%%%%%%%%%%%%%%%%%%%%%%
\section{Energy estimate II: the pairing contributions in the second generation}\label{Section:singular} 
In this section, we estimate the singular contributions. Recall the definition of $\mathcal{S}_{i,j}$ in \eqref{defS11}-\eqref{defS22}. 
	\begin{prop}\label{prop:Sij}
		There exist $C>0$ and $\beta=\beta(\theta,s)\in(0,1)$, such that for $j\in\{1,2\}$, $R\geq 1$ and $p\in[2,\infty)$, we have
		$$ \|\Im\big(\mathcal{S}_{1,j}(w)-\mathcal{S}_{2,j}(w)\big)\mathbf{1}_{B_R^{H^{\sigma}}	(w)}\|_{L^p(\dd\mu_s)}\leq Cp^{\beta}R^{10}.
		$$
	\end{prop}
	
	\begin{rem}\label{Firstcancellation} 
	To explain the difficulty, we remark that the singular contributions $\mathcal{S}_{i,j}(w)$ ($i,j\in\{1,2\}$) prevents us to use Wiener-chaos estimate to gain the square root cancellation. Nevertheless, it turns out that there is an extra cancellation when one takes the imaginary part of $\mathcal{S}_{i,j}(w)$. To understand the hidden cancellation, for $\mathcal{S}_{1,1}(v)$, one can think about the sum is taken over $|k_3|,\cdots,|k_6|,|p_2|,\cdots,|p_5|=O(1)$, then
		$$ \frac{\psi_{2s}(\vec{k})}{\Omega(\vec{k})}\approx \frac{|k_1|^{2s}-|k_2|^{2s}}{|k_1|^2-|k_2|^2}, 
		$$
		and the second sum in the definition of $\mathcal{S}_{1,1}$ is completely decoupled and we have
		\begin{align*}
			\mathcal{S}_{1,1}(w)= &-\sum_{k_1,k_2}\chi_N(k_1)^2|w_{k_2}|^2\frac{|k_1|^{2s}-|k_2|^{2s}}{|k_1|^2-|k_2|^2}\Big|\sum_{\substack{|k_3|+|k_4|+|k_5|+|k_6|\leq |k_2|^{\theta}\\
					k_3-k_4+k_5-k_6=k_2-k_1 } } w_{k_3}\ov{w}_{k_4}w_{k_5}\ov{w}_{k_6}
			\Big|^2\\
			+&\text{ error },
		\end{align*}
		where the main contribution is obviously real.
		\end{rem}
		%%%
\begin{rem}
However, it turns out that the cancellation described in Remark \ref{Firstcancellation} alone is not enough to conclude, as the error term in the formula above is not negligible if we estimate individually $\mathcal{S}_{1,j}(w)$ and $\mathcal{S}_{1,j}(w)$, $j=1,2$. 
What saves us is that these expressions there is some symmetric structure 
%in $\mathcal{S}_{1,1}$ and $\mathcal{S}_{2,1}$ so that $\mathcal{S}_{1,1}(w)-\mathcal{S}_{2,1}(w)$ enjoys 
so that we can exploit some extra probabilistic cancellation and a deterministic smoothing. 
%Besides the cancellations explained in Remark \ref{Firstcancellation},  it turns out that there is a further cancellation for $\Im\big(\mathcal{S}_{1,j}(v)-\mathcal{S}_{2,2}(v)\big)$. 
Let us explain theses points with more details.
%To see this, 
With the identification of  $(q_3,q_2,q_5,q_4)=(p_2,p_3,p_4,p_5)$ (without changing $k_j$), we observe that $\Lambda_{2,1}=\Lambda_{1,1}$ and
	\begin{align}\label{S21new} 
		\mathcal{S}_{2,1}(w)=\sum_{\Lambda_{1,1}}\chi_N(k_2)^2|w_{k_1}|^2\frac{\psi_{2s}(\vec{k})}{\Omega(\vec{k})}\Big(1-\chi\Big(\frac{\Omega(\vec{k})}{\lambda(\vec{k})^{\delta_0} }\Big)\Big) w_{k_3}\ov{w}_{k_4}w_{k_5}\ov{w}_{k_6}\cdot \ov{w}_{p_2}w_{p_3}\ov{w}_{p_4}w_{p_5}.
	\end{align}
	Therefore,
	\begin{align}\label{cancellationS11S21} 
	\Im\mathcal{S}_{1,1}(w)-\Im\mathcal{S}_{2,1}(w) =&\Im\sum_{\Lambda_{1,1}}(\chi_N(k_1)^2|w_{k_2}|^2-\chi_N(k_2)^2|w_{k_1}|^2)\frac{\psi_{2s}(\vec{k})}{\Omega(\vec{k})}\Big(1-\chi\Big(\frac{\Omega(\vec{k})}{\lambda(\vec{k})^{\delta_0} }\Big)\Big)\notag\\
		&\hspace{2.5cm}\times w_{k_3}\ov{w}_{k_4}w_{k_5}\ov{w}_{k_6}\cdot \ov{w}_{p_2}w_{p_3}\ov{w}_{p_4}w_{p_5}.
	\end{align}
	For $j=2$, due to the special position, there is no such cancellation in $\Im\big(\mathcal{S}_{1,2}(w)-\mathcal{S}_{2,2}(w)\big)$. Indeed, 
	\begin{align*}
		\mathcal{S}_{2,2}(w)=&\sum_{k_2,k_4}\chi_N(k_2)^2|w_{k_4}|^2
		\!\!\!\!\!\!\!\!	\sum_{\substack{k_1+k_3+k_5-k_6=k_2+k_4\\
				q_1+q_3+q_5-q_4=k_2+k_4\\
				|k_1|+|k_3|+|k_5|+|k_6|\leq |k_2|^{\theta}+|k_4|^{\theta}\\
				|q_1|+|q_3|+|q_4|+|q_5|\leq |k_2|^{\theta}+ |k_4|^{\theta} 
		} }\!\!\!\!\!\!\!\!\!\!\frac{\psi_{2s}(\vec{k})}{\Omega(\vec{k})}
		\Big(1-\chi\Big(\frac{\Omega(\vec{k})}{\lambda(\vec{k})^{\delta_0}}\Big)\Big) w_{k_1}w_{k_3}w_{k_5}\ov{w}_{k_6}\cdot \ov{w}_{q_1}\ov{w}_{q_3}\ov{w}_{q_5}w_{q_4}
	\end{align*}
	By switching the indices $(k_1,k_3,k_5)$ with $(k_2,k_4,k_6)$ and identifying $(q_1,q_3,q_4,q_5)$ as $(p_1,p_3,p_4,p_5)$ in $\Lambda_{2,2}$, we deduce that
	$$ \mathcal{S}_{2,2}(w)=\ov{\mathcal{S}}_{1,2}(w),
	$$
	where we used the fact that $\frac{\psi_{2s}(\vec{k})}{\Omega(\vec{k})}$ is invariant by the switching of indices $(k_1,k_3,k_5)$ and $(k_2,k_4,k_6)$.  Therefore,
	\begin{align}\label{ImS12S22}
		\Im\mathcal{S}_{1,2}(w)-\Im\mathcal{S}_{2,2}(w)=-2\Im\mathcal{S}_{2,2}(w).
	\end{align} 
	The good news is that in the expression $\Im\mathcal{S}_{1,2}(v)$, we only need to exploit the first cancellation explained in Remark \ref{Firstcancellation}, since the resonant function $\Omega(\vec{k})\approx |k_1|^2+|k_3|^2\sim |k_{(1)}|^2$ has  a significantly larger size which provides a smoothing effect. 
\end{rem}	
%%%%%%%%%%%%%%%%%%%%%%%%%%%%%%%%%%%%%%%
%%%%%%%%%%%%%%%%%
\begin{proof}[Proof of Proposition \ref{prop:Sij}] We separate the analysis for $j=1$ and $j=2$.\\
		
		\noi
		$\bullet$ {\bf Estimate for $j=1$:}
		
		Set
		$$ \Psi(\vec{k}):=\frac{\psi_{2s}(\vec{k})}{\Omega(\vec{k})}\Big(1-\chi\Big(\frac{\Omega(\vec{k})}{\lambda(\vec{k})^{\delta_0}}\Big)\Big)-\frac{|k_1|^{2s}-|k_2|^{2s}}{|k_1|^2-|k_2|^2}\Big(1-\chi\Big(\frac{|k_1|^2-|k_2|^2}{(|k_1|^2+|k_2|^2)^{\delta_0/2}}\Big)\Big).
		$$
		We need an elementary lemma:
		\begin{lem}\label{cancellationerror} 
			On $\Lambda_{1,1}$ defined in \eqref{defLambda11}, for sufficiently large $|k_{(1)}|$, we have
			$$ |\Psi(\vec{k})|\lesssim \frac{|k_{(1)}|^{2s-2}|k_{(3)}|^2}{|\Omega(\vec{k})|}\mathbf{1}_{|\Omega(\vec{k})|\gtrsim |k_{(1)}|^{\delta_0}},
			$$
			where we recall that in the definition of $\Lambda_{1,1}$, $\theta<\frac{\delta_0}{2}$. 
		\end{lem}				
		\begin{proof}
			Note that on $\Lambda_{1,1}$, 
			$\{k_1,k_2 \}=\{k_{(1)},k_{(2)} \}$ and
			$|k_{(3)}|^2\lesssim |k_{(1)}|^{2\theta}\ll \lambda(\vec{k})^{\delta_0}$. Thanks to the support property of $\chi$,  if $|\Omega(\vec{k})|\ll |k_{(1)}|^{\delta_0}\sim \lambda(\vec{k})^{\delta_0}$, we must have
			$$ \Psi(\vec{k})=-\frac{|k_1|^{2s}-|k_2|^{2s} }{|k_1|^2-|k_2|^2}\Big(1-\chi\Big(\frac{|k_1|^2-|k_2|^2 }{(|k_1|^2+|k_2|^2 )^{\delta_0/2} } \Big)\Big),
			$$
			and $||k_1|^2-|k_2|^2|\gtrsim (|k_1|^2+|k_2|^2)^{\delta_0/2}\sim \lambda(\vec{k})^{\delta_0}$, otherwise $\Psi(\vec{k})=0$. Thus $$|k_{(1)}|^{\delta_0}\lesssim  ||k_1|^2-|k_2|^2|=|\Omega(\vec{k})|-O(|k_{(3)}|^2),$$ which contradicts to the fact that $|k_{(3)}|^2\ll |k_{(1)}|^{\delta_0}$.
		 Therefore, 
			we assume that $|\Omega(\vec{k})|\gtrsim |k_{(1)}|^{\delta_0}$ and consequently $|k_1|\neq |k_2|$ in the sequel.

			Set 
			$$
			G=|k_3|^2-|k_4|^2+|k_5|^2-|k_6|^2, \quad F=|k_3|^{2s}-|k_4|^{2s}+|k_5|^{2s}-|k_6|^{2s}
			$$
			 and write
			\begin{align*}
				\psi_{2s}(\vec{k})=\frac{|k_1|^{2s}-|k_2|^{2s}}{|k_1|^2-|k_2|^2}(\Omega(\vec{k})-G)+F.
			\end{align*}
			Hence
			\begin{align*}
				\Psi(\vec{k})=&\frac{|k_1|^{2s}-|k_2|^{2s}}{|k_1|^2-|k_2|^2}\Big[\Big(1-\frac{G}{\Omega(\vec{k})}\Big)\Big(1-\chi\Big(\frac{\Omega(\vec{k})}{\lambda(\vec{k})^{\delta_0}}\Big)\Big)-\Big(1-\chi\Big(\frac{|k_1|^2-|k_2|^2}{(|k_1|^2+|k_2|^2)^{\delta_0/2}}\Big)\Big) \Big]\\
				+&\frac{F}{\Omega(\vec{k})}\Big(1-\chi\Big(\frac{\Omega(\vec{k})}{\lambda(\vec{k})^{\delta_0}}\Big)\Big).
			\end{align*}
			Since $|F|\lesssim |k_{(3)}|^{2s}, |G|\lesssim |k_{(3)}|^2$ and $$\Big|\frac{|k_1|^{2s}-|k_2|^{2s}}{|k_1|^2-|k_2|^2}\Big|\lesssim |k_{(1)}|^{2s-2},$$
			we deduce that
			\begin{align}\label{mean-value}
				|\Psi(\vec{k})|\lesssim \frac{|k_{(1)}|^{2s-2}|k_{(3)}|^2}{|\Omega(\vec{k})|}\Big(1-\chi\Big(\frac{\Omega(\vec{k})}{\lambda(\vec{k})^{\delta_0}}\Big)\Big)+|k_{(1)}|^{2s-2}\Big|\chi\Big(\frac{\Omega(\vec{k})}{\lambda(\vec{k})^{\delta_0}}\Big)-\chi\Big(\frac{|k_1|^2-|k_2|^2}{(|k_1|^2+|k_2|^2)^{
						\delta_0/2}} \Big) \Big|.
			\end{align}
			The first term on the right hand side of \eqref{mean-value} satisfies the claimed bound. It remains to evaluate the second one. 
			By the mean value theorem, there exists $\alpha\in[0,1]$ such that
			\begin{align*}
				&\chi\Big(\frac{\Omega(\vec{k})}{\lambda(\vec{k})^{\delta_0}}\Big)-\chi\Big(\frac{|k_1|^2-|k_2|^2}{(|k_1|^2+|k_2|^2)^{
						\delta_0/2}} \Big)
				=\chi'(\xi_{\alpha})\Big(\frac{\Omega(\vec{k})}{\lambda(\vec{k})^{\delta_0}}-\frac{|k_1|^2-|k_2|^2}{(|k_1|^2+|k_2|^2)^{\delta_0/2} } \Big),
			\end{align*}
			where
			$$ \xi_{\alpha}=\frac{\Omega(\vec{k})}{\lambda(\vec{k})^{\delta_0}}-\alpha\Big(\frac{\Omega(\vec{k})}{\lambda(\vec{k})^{\delta_0}}-\frac{|k_1|^2-|k_2|^2}{(|k_1|^2+|k_2|^2)^{\delta_0/2} } \Big).
			$$
Thanks to the support properties of  $\chi'$, when the second term on the right hand side of \eqref{mean-value} is non zero, we must have $|\xi_{\alpha}|\sim 1$. In this case, a direct computation yields
			\begin{align*}
				\frac{\Omega(\vec{k})}{\lambda(\vec{k})^{\delta_0}}-\frac{|k_1|^2-|k_2|^2}{(|k_1|^2+|k_2|^2)^{\delta_0/2} }=\frac{\Omega(\vec{k})}{\lambda(\vec{k})^{\delta_0}}\cdot\frac{(|k_1|^2+|k_2|^2)^{\delta_0/2}-\lambda(\vec{k})^{\delta_0}}{(|k_1|^2+|k_2|^2)^{\delta_0/2}}+\frac{G}{(|k_1|^2+|k_2|^2)^{\delta_0/2}}.
			\end{align*}
			Note that $|(|k_1|^2+|k_2|^2)^{\delta_0/2}-\lambda(\vec{k})^{\delta_0}|\lesssim |k_{(3)}|^{\delta_0}$,
			we deduce that 
			$$ \Big|\frac{\Omega(\vec{k})}{\lambda(\vec{k})^{\delta_0}}-\xi_{\alpha}\Big|\lesssim \alpha\frac{|k_{(3)}|^{\delta_0}}{\lambda(\vec{k})^{\delta_0}}\Big|\frac{\Omega(\vec{k})}{\lambda(\vec{k})^{\delta_0}}-\xi_{\alpha}\Big|
			+
			\alpha\frac{|k_{(3)}|^{\delta_0}}{\lambda(\vec{k})^{\delta_0}}|\xi_\alpha|
			+\frac{|G|}{\lambda(\vec{k})^{\delta_0}}.
			$$
			As $|G|\lesssim |k_{(3)}|^{2}\lesssim \lambda_1(\vec{k})^{2\theta}\ll \lambda(\vec{k})^{\delta_0}$ for large enough $|k_{(1)}|$, we deduce that 
			$$ \Big|\frac{\Omega(\vec{k})}{\lambda(\vec{k})^{\delta_0}}-\xi_{\alpha}\Big|\lesssim |k_{(3)}|^2\lambda(\vec{k})^{-\delta_0}\lesssim \lambda(\vec{k})^{-\delta_0+2\theta}\ll 1.
			$$
			Since $|\Omega(\vec{k})|\gtrsim |k_{(1)}|^{\delta_0}$, the second term on the right hand side of \eqref{mean-value} is bounded by
			$$ \mathbf{1}_{|\Omega(\vec{k})|\sim |k_{(1)}|^{\delta_0}}\cdot \frac{|k_{(1)}|^{2s-2}|k_{(3)}|^2 }{\lambda(\vec{k})^{\delta_0}}\sim \mathbf{1}_{|\Omega(\vec{k})|\sim |k_{(1)}|^{\delta_0}}\cdot \frac{|k_{(1)}|^{2s-2}|k_{(3)}|^2 }{|\Omega(\vec{k})| }.
			$$
			This completes the proof of Lemma \ref{cancellationerror}.
		\end{proof}		
The key observation is that
		\begin{align*}
			&\sum_{\Lambda_{1,1} }(\chi_N(k_1)^2|w_{k_2}|^2-\chi_N(k_2)^2|w_{k_1}|^2)\Big[\frac{\psi_{2s}(\vec{k})}{\Omega(\vec{k})}\Big(1-\chi\Big(\frac{\Omega(\vec{k})}{\lambda(\vec{k})^{\delta_0}}\Big) \Big)-\Psi(\vec{k})\Big] \\ 
			&\hspace{2cm}\times w_{k_3}\ov{w}_{k_4}w_{k_5}\ov{w}_{k_6}\cdot
			\ov{w}_{p_2}w_{p_3}\ov{w}_{p_4} w_{p_5}\\
			=&\sum_{k_1,k_2}(\chi_N(k_1)^2|w_{k_2}|^2-\chi_N(k_2)^2|w_{k_1}|^2)\frac{|k_1|^{2s}-|k_2|^{2s}}{|k_1|^2-|k_2|^2}\Big(1-\chi\Big(\frac{|k_1|^2-|k_2|^2}{(|k_1|^2+|k_2|^2)^{\delta_0/2}}\Big)\Big)\\
			&\hspace{1cm}\times\Big|\sum_{\substack{k_3-k_4+k_5-k_6=k_2-k_1\\
					|k_3|+|k_4|+|k_5|+|k_6|\leq |k_1|^{\theta}+|k_2|^{\theta}  } } w_{k_3}\ov{w}_{k_4}w_{k_5}\ov{w}_{k_6} \Big|^2
		\end{align*}
		is real-valued and it disappears when taking the imaginary part. 
		
		Therefore it suffices to show that there exists $\beta=\beta(s,\theta,\delta_0)\in (0,1)$, such that
		\begin{align}
			\|J(w)\mathbf{1}_{\|w\|_{H^{\sigma}}\leq  R}\|_{L^p(\dd\mu_s)}\lesssim p^{\beta}R^{10}.
		\end{align}
		where
		\begin{align}\label{Jdefinition}
			J(w):=\sum_{\Lambda_{1,1}}\Psi(\vec{k})(\chi_N(k_1)^2|w_{k_2}|^2-\chi_N(k_2)^2|w_{k_1}|^2) w_{k_3}\ov{w}_{k_4}w_{k_5}\ov{w}_{k_6} \cdot
			\ov{w}_{p_2}w_{p_3}\ov{w}_{p_4} w_{p_5}.
		\end{align} 
		Since in the above expression, the contribution of $k_1=k_2$ is zero, below we always implicitly assume that $k_1\neq k_2$.

		For dyadic numbers $N_1,N_2,N_3,N_4,N_5,N_6,M_2,M_3,M_4,M_5$, we decompose accordingly $w_{k_j}^{N_j}=w_{k_j}\mathbf{1}_{|k_j|\sim N_j}$ and $w_{p_i}^{M_i}=v_{p_i}\mathbf{1}_{|p_i|\sim M_i}$. It suffices to show that
		\begin{align}\label{dyadicJbound}
			\|J_{N_1,\cdots,N_6;M_2,\cdots,M_5}(w)\mathbf{1}_{\|w\|_{H^{\sigma}}\leq R}\|_{L^p(\dd\mu_s)}\lesssim p^{\beta}N_{(1)}^{-\gamma}R^{10}
		\end{align}
	for some $\beta\in(0,1)$ and $\gamma>0$,
		where $J_{N_1,\cdots,N_6;M_2,\cdots,M_5}$ is the same expression as $J(w)$ by replacing the inputs $w_{k_j},w_{p_i}$ to $w_{k_j}^{N_j},w_{p_i}^{M_i}$.
		By definition of $\Lambda_{1,1}$, we have $N_1\sim N_2$ and $$N_3+\cdots+N_6+M_2+\cdots+M_5\lesssim N_1^{\theta}.$$

		By Lemma \ref{cancellationerror} and the fact that $|\Omega(\vec{k})|\gtrsim N_{(1)}^{\delta_0}>N_{(1)}^{2\theta}\gtrsim N_{(3)}^2$, a crude deterministic estimate leads to
		\begin{align*}
			|J_{N_2,\cdots,M_5}(w)|\lesssim &\sum_{\Lambda_{1,1}}\frac{N_1^{2(s-1)}N_{(3)}^2 }{|\Omega(\vec{k})|}\mathbf{1}_{k_1\neq k_2}(|w^{N_2}_{k_2}|^2+|w^{N_1}_{k_1}|^2)|w^{N_3}_{k_3}\cdots w^{N_6}_{k_6}|\cdot |w^{M_2}_{p_2}\cdots w^{M_5}_{p_5}|\\
			\lesssim &
			\sum_{\Lambda_{1,1}}N_1^{2(s-1)} \mathbf{1}_{k_1\neq k_2}(|w^{N_2}_{k_2}|^2+|w^{N_1}_{k_1}|^2)|w^{N_3}_{k_3}\cdots w^{N_6}_{k_6}|\cdot |w^{M_2}_{p_2}\cdots w^{M_5}_{p_5}|,
		\end{align*}
		and the right hand side ca be bounded by
		\begin{align*} 
			&N_1^{2(s-1)} (\|w_{k_1}^{N_1}\|_{l^2}^2+\|w_{k_2}^{N_2}\|_{l^2}^2)\|w_{k_3}^{N_3}\|_{l^1}\cdots\|w_{k_6}^{N_6}\|_{l^1}\|w_{p_2}^{M_2}\|_{l^1}\cdots\|w_{p_5}^{M_5}\|_{l^1}\\ \lesssim &N_1^{2(s-1)}N_1^{-2\sigma}(\|w_{k_1}^{N_1}\|_{h^{\sigma}}^2+\|w_{k_2}^{N_2}\|_{h^{\sigma}}^2)\|w_{k_3}^{N_3}\|_{h^{\sigma}}\cdots\|w_{p_5}^{M_5}\|_{h^{\sigma}}\cdot(N_3\cdots M_5)^{-\sigma+\frac{3}{2}}\\
			\lesssim & N_1^{2(s-1-\sigma)}(N_3\cdots M_5)^{-\sigma+\frac{3}{2}}\|w\|_{H^{\sigma}}^{10}.
		\end{align*}
		Therefore, we obtain the first bound
		\begin{align}\label{Jbounddeterministic} 
			|J_{N_1,\cdots,M_5}(w)|\mathbf{1}_{\|w\|_{H^{\sigma}}\leq R}\lesssim N_1^{2(s-1-\sigma)}
			(N_3\cdots M_5)^{-(\sigma-\frac{3}{2})} R^{10}.
		\end{align}
		As $\sigma<s-\frac{3}{2}$, we need to improve the above bound using the Wiener chaos estimate. We further split
		$$ J_{N_1,\cdots,M_5}(w):=\widetilde{J}_{N_1,\cdots,M_5}(w)+R_{N_1,\cdots,M_5}(w),
		$$
		where
		\begin{align}\label{tildeJ} 
			\widetilde{J}_{N_1,\cdots,M_5}(w):=\sum_{\Lambda_{1,1}}\Psi(\vec{k})&\Big[\chi_N(k_1)^2\big(|w^{N_2}_{k_2}|^2-\frac{1}{1+|k_2|^{2s} }   
					\big)-\chi_N(k_2)^2\big(|w^{N_1}_{k_1}|^2-
					\frac{1}{1+|k_1|^{2s}}
					\big)\Big]\notag \\ \times& w^{N_3}_{k_3}\ov{w}^{N_4}_{k_4}w^{N_5}_{k_5}\ov{w}^{N_6}_{k_6}  \cdot
			\ov{w}^{M_2}_{p_2}w^{M_3}_{p_3}\ov{w}^{M_4}_{p_4} w^{M_5}_{p_5},
		\end{align}
		and
		\begin{align*}
			R_{N_1,\cdots,M_5}:= \sum_{\Lambda_{1,1}}\Psi(\vec{k})\Big(\frac{\chi_N(k_1)^2}{
				1+|k_2|^{2s}
			}-\frac{\chi_N(k_2)^2}{
				1+|k_1|^{2s}  
			}\Big) 
			w^{N_3}_{k_3}\ov{w}^{N_4}_{k_4}w^{N_5}_{k_5}\ov{w}^{N_6}_{k_6}  \cdot
			\ov{w}^{M_2}_{p_2}w^{M_3}_{p_3}\ov{w}^{M_4}_{p_4} w^{M_5}_{p_5}.
			%v_{k_3}\ov{v}_{k_4}v_{k_5}\ov{v}_{k_6} \cdot \ov{v}_{p_2}v_{p_3}\ov{v}_{p_4} v_{p_5}. 	
		\end{align*} 
		For $|k_1|\sim |k_2|\sim N_{(1)}\lesssim N$, by the mean-value theorem and the fact that
		$\chi_N(k_1)-\chi_N(k_2)$ takes the form $\widetilde{\chi}(|k_1|^2/N^2)-\widetilde{\chi}(|k_2|^2/N^2)$, we have
		$$ \Big|\frac{\chi_N(k_1)^2}{ 1+|k_2|^{2s}
	 }-\frac{\chi_N(k_2)^2}{1+|k_1|^{2s} 
		}\Big|\lesssim \frac{|\Omega(\vec{k})|+|k_{(3)}|^2}{|k_{(1)}|^{2(s+1)}}.
		$$
		Thus the remainder term $R_{N_1,\cdots,M_5}$ can be estimated by the previous deterministic manipulation (recall that $|\Omega(\vec{k})|\gtrsim N_{(1)}^{\delta_0}>N_{(1)}^{2\theta}$ in the sum)
		\begin{align}
			&N_1^{2(s-1)+2\theta-2(s+1)}\sum_{\Lambda_{1,1}}\mathbf{1}_{k_1\neq k_2} \frac{|\Omega(\vec{k})|+N_1^{2\theta}}{|\Omega(\vec{k})|}|w_{k_3}^{N_3}\cdots w_{p_5}^{M_5}|\notag  \\
			\lesssim &N_1^{-4+2\theta}\cdot N_1^{3}\|w_{k_3}^{N_3}\|_{l^1}\cdots \|w_{p_5}^{M_5}\|_{l^1}\notag\\
			\lesssim & N_1^{-1+2\theta}\|w_{k_3}^{N_3}\|_{h^{\sigma}}\cdots \|w_{p_5}^{M_5}\|_{h^{\sigma}}(N_3\dots M_5)^{-\sigma+\frac{3}{2}} \notag \\
			\lesssim &N_1^{-1+2\theta}(N_3\cdots M_5)^{-(\sigma-\frac{3}{2})}\|w\|_{H^{\sigma}}^8\lesssim N_{(1)}^{-1+2\theta}R^8, \label{outputRN1M5}
		\end{align}
		which is conclusive as far as  $\theta<\frac{1}{2}$ and $\sigma>\frac{3}{2}$.
		
		We next estimate  $\widetilde{J}_{N_1,\cdots,M_5}$. We will not make use of the cancellation in the difference 
		$$
\chi_N(k_1)^2\big(|w^{N_2}_{k_2}|^2-\frac{1}{1+|k_2|^{2s} }  
\big)-\chi_N(k_2)^2\big(|w^{N_1}_{k_1}|^2-
		\frac{1}{1+|k_1|^{2s}}
\big),
		$$
		so we treat separately in the same manner the contribution of each term.

		%To estimate $\widetilde{J}_{N_1,\cdots,M_5},$ note that under the measure $\mu_s$, 
		%\begin{align*}  \widetilde{J}_{N_1,\cdots,M_5}(w)=\sum_{\substack{\Lambda_{1,1}\\ |k_1|\sim N_1\\|k_2|\sim N_2 }} \Big(\frac{|g_{k_2}|^2-1}{\langle k_2\rangle^{2s}}-\frac{|g_{k_1}|^2-1}{\langle k_1\rangle^{2s}}\Big)\Psi(\vec{k})  w^{N_3}_{k_3}\ov{w}^{N_4}_{k_4}w^{N_5}_{k_5}\ov{w}_{k_6}^{N_6} \cdot
		%	\ov{w}^{M_2}_{p_2}w_{p_3}^{M_3}\ov{w}_{p_4}^{M_4} w_{p_5}^{M_5}.
		%\end{align*}
		%Here we will not make use of the cancellation between $\frac{|g_{k_2}|^2-1}{\langle k_2\rangle^{2s}}$ and $\frac{|g_{k_1}|^2-1}{\langle k_1\rangle^{2s}}$. So we treat them separately in the same manner. 
		Let $\mathcal{B}_{\ll N_1}$ be the $\sigma$-algebra generated by Gaussians $g_{k_j},|k_j|\ll N_1$ and $\mathbf{P}_{\ll N_1}$ the frequency projector to $|k|\ll N_1$. In particular,
		$g_{k_1},g_{k_2}$ for $|k_1|\sim N_1,|k_2|\sim N_2$ are independent of the $\sigma$-algebra $\mathcal{B}_{\ll N_1}$.
		We have
		\begin{align*} \|\widetilde{J}_{N_1,\cdots,M_5}(w)\mathbf{1}_{\|w\|_{H^{\sigma}}\leq R}\|_{L^p(\dd\mu_s)}^p\leq\mathbb{E}^{\mu_s}[\mathbb{E}^{\mu_s}[|\widetilde{J}_{N_1,\cdots,M_5}(w)\mathbf{1}_{\|\mathbf{P}_{\ll N_1}w\|_{H^{\sigma}}\leq R}|^p|\mathcal{B}_{\ll N_1}
			]
			].
		\end{align*}
		As $\mathbf{P}_{\ll N_1}w,w_{k_3}^{N_3},\cdots,w_{p_5}^{M_5} $ are $\mathcal{B}_{\ll N_1}$-measurable,	by the Wiener chaos estimate conditional to $\mathcal{B}_{\ll N_1}$, 
		\begin{align*}
			&\Big(\mathbb{E}^{\mu_s}[|\widetilde{J}_{N_1,\cdots,M_5}(w)\mathbf{1}_{\|\mathbf{P}_{\ll N_1}w\|_{H^{\sigma}}\leq R}|^p|\mathcal{B}_{\ll N_1}
			]\Big)^{\frac{1}{p}}\\
			\leq & Cp\Big(\sum_{|k_2|\sim N_2} \frac{1}{\langle k_2\rangle^{4s}}
			\Big|\sum_{k_3-k_4+k_5-k_6=p_2-p_3+p_4-p_5}\!\!\!\!\!\!\Psi(\vec{k})w_{k_3}^{N_3}\cdots w_{p_5}^{M_5} \Big|^2
			\Big)^{\frac{1}{2}} \cdot \mathbf{1}_{\|\mathbf{P}_{\ll N_1}w\|_{H^{\sigma}}\leq R}.
		\end{align*}
		By Lemma \ref{cancellationerror} and Cauchy-Schwarz, we bound the above expression by 
		\begin{align*}
			&CpN_2^{-2+2\theta}\Big(\sum_{\Lambda_{1,1}, k_1\neq k_2}\frac{1}{|\Omega(\vec{k})|^2}\Big)^{\frac{1}{2}}\|w_{k_3}^{N_3}\|_{l^2}\cdots\|w_{p_5}^{M_5}\|_{l^2}\cdot\mathbf{1}_{\|\mathbf{P}_{\ll N_1}w \|_{H^{\sigma}}\leq R }\\
			\lesssim & CpN_2^{-2+2\theta}\cdot N_2\|w_{k_3}^{N_3}\|_{h^{\sigma}}\cdots\|w_{p_5}^{M_5}\|_{h^{\sigma}}(N_3\cdots M_5)^{-\sigma+\frac{3}{2}}\cdot\mathbf{1}_{\|\mathbf{P}_{\ll N_1}w \|_{H^{\sigma}}\leq R}\\
			\lesssim &CpN_1^{-1+2\theta}R^8,
		\end{align*}
		provided that 
		%which is conclusive as far as $\theta<\frac{1}{2}$ and 
		$\sigma>\frac{3}{2}$.
		Therefore,
		$$ \|\widetilde{J}_{N_1,\cdots,M_5}(w)\mathbf{1}_{\|w\|_{H^{\sigma}}\leq R}\|_{L^p(\dd\mu_s)}\lesssim pN_1^{-1+2\theta}R^8.
		$$
		Combining with \eqref{outputRN1M5} and interpolating with \eqref{Jbounddeterministic}, we deduce that there exist constants $C>0$, $\beta=\delta(s,\theta)\in(0,1)$, such that for any $p\geq 2$ and $R\geq 1$, 
		\begin{align}\label{Jboundfinal} 
			\|J_{N_1,\cdots,M_5}(w)\mathbf{1}_{\|w\|_{H^{\sigma}}\leq R}\|_{L^p(\dd\mu_s)}\leq Cp^{\beta}N_{(1)}^{-1+2\theta}R^{10},
		\end{align}
		which is conclusive since $\theta<\frac{1}{3}$.
		\\
		
		\vspace{0.3cm}
		\noi
		$\bullet$ {\bf Estimate for $j=2$:}
		It suffices to estimate $\Im\mathcal{S}_{1,2}(w)$. Recall that	
$$
\mathcal{S}_{1,2}(w):=\sum_{\Lambda_{1,2}}\chi_N(k_1)^2|w_{k_3}|^2\Big(1-	\chi\Big(\frac{\Omega(\vec{k})}{\lambda(\vec{k})^{\delta_0}}\Big)\Big)\frac{\psi_{2s}(\vec{k})}{\Omega(\vec{k})}
\ov{w}_{k_2}\ov{w}_{k_4}w_{k_5}\ov{w}_{k_6}\notag \cdot w_{p_1}w_{p_3}\ov{w}_{p_4}w_{p_5},
$$
and on $\Lambda_{1,2}$, $|\Omega(\vec{k})|\sim |k_{(1)}|^2\gg \lambda(\vec{k})^{\delta_0}$, thus
$$ 
\frac{\psi_{2s}(\vec{k})}{\Omega(\vec{k})}\Big(1-\chi\Big(\frac{\Omega(\vec{k})}{\lambda(\vec{k})^{\delta_0}}\Big)\Big)=\frac{\psi_{2s}(\vec{k})}{\Omega(\vec{k})}.
$$
Set
$$ 
\widetilde{\Psi}(\vec{k}):=\frac{\psi_{2s}(\vec{k})}{\Omega(\vec{k})}-\frac{|k_1|^{2s}+|k_3|^{2s}}{|k_1|^2+|k_3|^2}.
$$
By mean-value theorem, we easily deduce that:
\begin{lem}\label{cancellationerror'} 
On $\Lambda_{1,2}$ defined in \eqref{defLambda12}, for sufficiently large $|k_{(1)}|$, we have
$$
 |\widetilde{\Psi}(\vec{k})|\lesssim \frac{|k_{(1)}|^{2s-2}|k_{(3)}|^{2}}{|\Omega(\vec{k})|}\mathbf{1}_{|\Omega(\vec{k})|\sim |k_{(1)}|^{2}}.
$$
\end{lem}
Since
$$
\sum_{k_1,k_3}\chi_N(k_1)^2|w_{k_3}|^2\frac{|k_1|^{2s}+|k_3|^{2s}}{|k_1|^2+|k_3|^2}\!\!\!\!
\sum_{\substack{k_2+k_4-k_5+k_6=k_1+k_3\\ p_1+p_3-p_4+p_5=k_1+k_3\\ |k_2|+|k_4|+|k_5|+|k_6|\leq |k_1|^{\theta}+|k_3|^{\theta}\\
|p_1|+|p_3|+|p_4|+|p_5|\leq |k_1|^{\theta}+|k_3|^{\theta} } }\!\!\!\!
\ov{w}_{k_2}\ov{w}_{k_4}w_{k_5}\ov{w}_{k_6}\cdot w_{p_1}w_{p_3}\ov{w}_{p_4}w_{p_5}
$$
equals to
$$
\sum_{k_1,k_3}\chi_N(k_1)^2|w_{k_3}|^2\frac{|k_1|^{2s}+|k_3|^{2s}}{|k_1|^2+|k_3|^2}\Big|\!\!\!\!\sum_{\substack{k_2+k_4-k_5+k_6=k_1+k_3\\
				|k_2|+|k_4|+|k_5|+|k_6|\leq |k_1|^{\theta}+|k_3|^{\theta} } }
		\!\!\!\!\!\!\
		%e^{it(|k_2|^2+|k_4|^2-|k_5|^2+|k_6|^2)}
		\ov{w}_{k_2}\ov{w}_{k_4}w_{k_5}\ov{w}_{k_6}
		\Big|^2,
		$$		
		which is real-valued, we deduce that
		\begin{align}\label{formula:ImS12}
			\mathrm{Im}(\mathcal{S}_{1,2}(w))=I(w):=&\sum_{\Lambda_{1,2}}\chi_N(k_1)^2|w_{k_3}|^2\widetilde{\Psi}(\vec{k})\ov{w}_{k_2}\ov{w}_{k_4}w_{k_5}\ov{w}_{k_6}\cdot w_{p_1}w_{p_3}\ov{w}_{p_4}w_{p_5}.
		\end{align}	
		As for the case $j=1$, we will split the sum into dyadic pieces. For dyadic numbers $N_j,M_j$, we decompose accordingly $w_{k_j}^{N_j}=w_{k_j}\mathbf{1}_{|k_j|\sim N_j}$ and $w_{p_i}^{M_i}=w_{p_i}\mathbf{1}_{|p_i|\sim M_i}$. It suffices to show that for some $\beta\in(0,1)$.
		\begin{align}\label{dyadicIbound}
			\|I_{N_1,\cdots,N_6;M_2,\cdots,M_5}(w)\mathbf{1}_{\|w\|_{H^{\sigma}}\leq R}\|_{L^p(\dd\mu_s)}\lesssim p^{\beta}N_{(1)}^{-\frac{1}{100}}R^{10},
		\end{align}
		where $I_{N_1,\cdots,N_6;M_2,\cdots,M_5}$ is the same expression as $I(w)$ by replacing the inputs $w_{k_j},w_{p_i}$ to $w_{k_j}^{N_j},w_{p_i}^{M_i}$.	
		
		It turns out that only deterministic estimates suffice. Indeed, from the fact that $|\Omega(\vec{k})|\sim N_{(1)}^2$, by Lemma \ref{cancellationerror'}, 
		we have
		$$ |I_{N_1,\cdots,M_5}(w)|\lesssim \frac{N_{(1)}^{2s-2}N_{(3)}^2}{N_{(1)}^2}\,\sum_{k_3}|w^{N_3}_{k_3}|^2\sum_{k_2,k_4,k_5,k_6}\sum_{p_1,p_3,p_4,p_5}|w^{N_2}_{k_2}w^{N_4}_{k_4}w^{N_5}_{k_5} w^{N_6}_{k_6}w^{M_1}_{p_1}w^{M_3}_{p_3}w^{M_4}_{p_4}w^{M_5}_{p_5}|,
		$$
		which is bounded by
		\begin{align*} &N_{(1)}^{2(s-2)}N_{(3)}^2\|w_{k_3}^{N_3}\|_{l^2}^2\|w_{k_2}^{N_2}\|_{l^1}\cdots \|w_{k_6}^{N_6}\|_{l^1}\|w_{p_1}^{M_1}\|_{l^1}\cdots\|w_{p_5}^{M_5}\|_{l^1}\\
			\lesssim & N_{(1)}^{2(s-2-\sigma)+2\theta}(N_2N_4N_5N_6M_1M_3M_4M_5)^{-(\sigma-\frac{3}{2})}\|w\|_{H^{\sigma}}^{10}\\
			\leq &N_{(1)}^{2(s-2-\sigma)+2\theta}\|w\|_{H^{\sigma}}^{10},	
		\end{align*}
provided that $\sigma>\frac{3}{2}$, which is conclusive when $\theta<\frac{1}{2}$ and $\sigma$ close enough to $s-\frac{3}{2}$. This completes the proof of Proposition \ref{prop:Sij}.	
\end{proof}	
%%%%%%%%%%%%%%%%%%%%%%%%%%%%%%%%%%%%%%%%%%%%%%%%%%%%%%%%%%%%%%%%%%%%%%%%%%
\section{Energy estimate III: Remainders in the second generation } \label{Section:rest}
In this section, we will estimate $\mathcal{R}_{1,3}(w),\mathcal{R}_{2,3}(w)$. More precisely, we have the following statement:
%defined in \eqref{defR13} in order to prove:
\begin{prop}\label{prop:R13}
Let $\theta<\frac{1}{3}$, close enough to $\frac{1}{3}$ and $\delta_0\in (2\theta, \frac{2}{3})$, close enough to $\frac{2}{3}$.  There exist $C>0$ and $\beta=\beta(\theta,s)\in(0,1)$, such that for $j\in\{1,2\}$, $R\geq 1$ and $p\in[2,\infty)$, we have
$$ 
\|\mathcal{R}_{1,3}(w) \mathbf{1}_{B_R^{H^{\sigma}}	(w)} \|_{L^p(\dd\mu_s)}+\|\mathcal{R}_{2,3}(w)\mathbf{1}_{B_R^{H^{\sigma}}	(w)} \|_{L^p(\dd\mu_s)}\leq Cp^{\beta}R^{10}.
$$
\end{prop}
Since the estimate for $\mathcal{R}_{2,3}(w)$ is similar, we only do it for $\mathcal{R}_{1,3}(w)$. Recall in the expression of $\mathcal{R}_{1,3}(w)$, we distinguish three types of contributions in the decomposition of the sum
$$ 
\sum_{\substack{k_1-k_2+k_3-k_4+k_5-k_6=0\\ k_1=p_1-p_2+p_3-p_4+p_5\\}}(\cdots).
$$
Recall that $k_{(1)}\cdots,k_{(10)}$ is a rearrangement of leaves $p_1,p_2,p_3,p_4,p_5,k_2,k_3,k_4,k_5,k_6$ such that $|k_{(1)}|\geq |k_{(2)}|\geq \cdots \geq |k_{(10)}|$.
	\begin{itemize}
		\item Type A: $\sum_{j=3}^{10}|k_{(j)}|>|k_{(1)}|^{\theta}+|k_{(2)}|^{\theta}$.
		\item Type B: $\sum_{j=3}^{10}|k_{(j)}|\leq |k_{(1)}|^{\theta}+|k_{(2)}|^{\theta}$ and $\{k_{(1)},k_{(2)}\}\subset \{k_2,k_3,k_4,k_5,k_6\}$ or $\{k_{(1)},k_{(2)}\}\subset \{p_1,p_2,p_3,p_4,p_5\}$.
		\item Type C: $\sum_{j=3}^{10}|k_{(j)}|\leq |k_{(1)}|^{\theta}+|k_{(2)}|^{\theta}$, $k_{(1)}\neq k_{(2)}$ and $$k_{(1)}\in\{k_2,k_3,k_4,k_5,k_6\},\quad k_{(2)}\in\{p_1,p_2,p_3,p_4,p_5\}$$ or $$k_{(2)}\in\{k_2,k_3,k_4,k_5,k_6\},\quad k_{(1)}\in\{p_1,p _2,p_3,p_4,p_5\},
		$$
		and $k_{(1)},k_{(2)}$ have different signatures.
	\end{itemize}
	Let us denote by $\Lambda_A$, $\Lambda_B$, $\Lambda_C$ the sets of indices $k_1, \ldots, k_6, p_1, \ldots, p_5$ that satisfy the linear constraints
	$$ k_1-k_2+k_3-k_4+k_5-k_6=0,\quad k_1=p_1-p_2+p_3-p_4+p_5
	$$
	and the conditions for Type A, B, C respectively. 
	Furthermore, we denote $\mathcal{R}_{1,3}^{(A)},\mathcal{R}_{1,3}^{(B)},\mathcal{R}_{1,3}^{(C)}$ the corresponding contributions to $\mathcal{R}_{1,3}(w)$.
	We will need the following elementary lemma. 
%%%%%%%%%%%	
\begin{lem}\label{basicdeterministic} 
		Assume that $f^{(j)}$ satisfies $f_{k_j}^{(j)}\mathbf{1}_{|k_j|\sim M_j}=f_{k_j}^{(j)}$ for $j=1,2,3,4,5,6$ with $M_j\in 2^{\N}$. Then
		\begin{multline*}
			%&
			\sum_{k_1,\cdots,k_6}\frac{\psi_{2s}(\vec{k})}{|\Omega(\vec{k})|}
			\Big(1-\chi\Big(\frac{\Omega(\vec{k})}{\lambda(\vec{k})^{\delta_0}}\Big)\Big)\cdot 
			\fk{h}_{k_1^+k_2^-k_3^+k_4^-k_5^+k_6^-}\cdot 
			\prod_{j=1}^6|f_{k_j}^{(j)}|
			\\ 
			\lesssim 
			%&
			M_{(1)}^{2s-2}\, M_{(3)}^2 (M_{(3)}M_{(4)}M_{(5)}M_{(6)})^{\frac{3}{2}}
			\prod_{j=1}^6\|f^{(j)}\|_{l^2},
		\end{multline*}
	where $M_{(1)}\geq M_{(2)}\geq\cdots\geq M_{(6)}$ is non-increasing rearrangement of dyadic integers $M_1,M_2,\cdots,M_6$.
	\end{lem}
	\begin{proof}
Since signature of $k_j$ plays no significant role in the proof, without loss of generality, we assume that $M_1\geq M_2\geq\cdots\geq M_6$. Since $|\psi_{2s}(\vec{k})|\lesssim M_{1}^{2s-2}(M_{3}^2+|\Omega(\vec{k})|)$, using the Cauchy-Schwarz inequality in the $k_1$ and $k_2$ summation, we obtain the bound 		
$$
M_{(1)}^{2s-2}\, M_{(3)}^2 \prod_{j=1}^2\|f^{(j)}\|_{l^2}\prod_{j=3}^6\|f^{(j)}\|_{l^1}.
$$	
It remains to use the 	Cauchy-Schwarz inequality  to pass from $l^1$ to $l^2$. This completes the proof of Lemma \ref{basicdeterministic}. 
\end{proof}
%%%%%%%%%%%%%%%%%%%%	
\begin{rem}
Since the crude bound is enough for our need, we do not make use of the denominator $\frac{1}{|\Omega(\vec{k})|}$ in the estimate above.
\end{rem}
\begin{proof}[Proof of Proposition \ref{prop:R13}]
		Since the proof follows from tedious estimates, we split it into three different parts, according to Type A,B,C.
		We will split the function $w$ into dyadic pieces, and we denote $w^K=\mathbf{P}_Kw$ in the sequel which means that $w_k^{K}=\mathbf{1}_{|k|\sim K}w_k$. 
		\begin{flushleft}
			$\bullet$ Estimate of Type A contribution:
		\end{flushleft}	
		
		We decompose the expression 
		$$ \mathcal{R}_{1,3}^{(A)}:=\sum_{\Lambda_A}\frac{\psi_{2s}(\vec{k})}{\Omega(\vec{k})}\Big(1-\chi\Big(\frac{\Omega(\vec{k})}{\lambda(\vec{k})^{\delta_0}}\Big)\Big)\chi_N(k_1)^2w_{p_1}\cdots w_{p_5}\ov{w}_{k_2}\cdots \ov{w}_{k_6}
		$$
		dyadically by
		\begin{align*}
			\sum_{M_1,\cdots,M_6, P_1,\cdots,P_5}\mathcal{R}_{1,3}^{(A)}(M_1,\cdots,P_5),
		\end{align*}
		where $\mathcal{R}_{1,3}^{(A)}(M_1,\cdots,P_5)$ is 
		$$ \sum_{\Lambda_A}\frac{\psi_{2s}(\vec{k})}{\Omega(\vec{k})}\Big(1-\chi\Big(\frac{\Omega(\vec{k})}{\lambda(\vec{k})^{\delta_0}}\Big)\Big)\chi_N(k_1)^2w^{P_1}_{p_1}\cdots w^{P_5}_{p_5}\ov{w}^{M_2}_{k_2}\cdots \ov{w}^{M_6}_{k_6}\cdot\mathbf{1}_{|k_1|\sim M_1}.
		$$
	 We denote by $N_{(1)}\geq N_{(2)}\geq\cdots N_{(10)}$ the non-increasing rearrangement of dyadic integers
		$$P_1,P_2,P_3,P_4,P_5,M_2,M_3,M_4,M_5,M_6.$$
		Note that the constraint $\sum_{j=3}^{10}|k_{(j)}|>|k_{(1)}|^{\theta}+|k_{(2)}|^{\theta}$ implies that $N_{(3)}\gtrsim N_{(1)}^{\theta}$ for non-zero terms $\mathcal{R}_{1,3}^{(A)}(M_1,\cdots,P_5)$. Write
$$
|\mathcal{R}_{1,3}^{(A)}(M_1,\cdots,P_5)|\leq  \sum_{k_1,\cdots,k_6}\frac{\psi_{2s}(\vec{k})}{|\Omega(\vec{k})|}\Big(1-\chi\Big(\frac{\Omega(\vec{k})}{\lambda(\vec{k})^{\delta_0}}\Big)\Big)
\cdot \fk{h}_{k_1^+k_2^-k_3^+k_4^-k_5^+k_6^-}\cdot
\prod_{j=1}^{6}|f_{k_j}^{(j)}|,
$$
where $f_{k_j}^{(j)}=w_{k_j}^{M_j}$ for $j=2,3,4,5,6$ and 
$$ 
f_{k_1}^{(1)}=\sum_{p_1-p_2+p_3-p_4+p_5=k_1}\mathbf{1}_{|k_1|\sim M_1}w_{p_1}^{P_1}\cdots w_{p_5}^{P_5}.
$$
Applying Lemma \ref{basicdeterministic}, we have
		\begin{align*}
			|\mathcal{R}_{1,3}^{(A)}(M_1,\cdots,P_5)|\lesssim &
			M_{(1)}^{2s-2}\, M_{(3)}^2 (M_{(3)}M_{(4)}M_{(5)}M_{(6)})^{\frac{3}{2}}
			%M_{(1)}^{2s-2}M_{(3)}^{\frac{5}{2}}\log M_{(1)}M_{(4)}^{\frac{1}{2}}(M_{(5)}M_{(6)}^{\frac{3}{2}})
			\prod_{j=1}^6\|f^{(j)}\|_{l^2}\\
			\lesssim & 
			M_{(1)}^{2s-2}\, M_{(3)}^2 (M_{(3)}M_{(4)}M_{(5)}M_{(6)})^{\frac{3}{2}}
			%M_{(1)}^{2s-2}(\log M_{(1)})M_{(3)}^{\frac{5}{2}}M_{(4)}^{\frac{1}{2}}(M_{(5)}M_{(6)})^{\frac{3}{2}}
			\|f^{(1)}\|_{l^2}\prod_{j=2}^6(M_j^{-\sigma}\|w^{M_j}\|_{H^{\sigma}}).
		\end{align*}
		where $M_{(1)}\geq M_{(2)}\geq \cdots\geq M_{(6)}$ is a non-increasing rearrangement of $M_1,M_2,\cdots,M_6$. By Cauchy-Schwarz, we have
		$$ \|f^{(1)}\|_{l^2}\lesssim (P_{(2)}P_{(3)}P_{(4)}P_{(5)})^{\frac{3}{2}}\prod_{j=1}^5(P_j^{-\sigma}\|w^{P_j}\|_{H^{\sigma}}),
		$$
		where $P_{(1)}\geq P_{(2)}\geq P_{(3)}\geq P_{(4)}\geq P_{(5)}$ is a non-increasing rearrangement of $P_1,P_2,P_3,P_4,P_5$. Thus we obtain that
		\begin{multline*}
		|\mathcal{R}_{1,3}^{(A)}(M_1,\cdots,P_5)|
		\\
		\lesssim 
			M_{(1)}^{2s-2}\, M_{(3)}^2 (M_{(3)}M_{(4)}M_{(5)}M_{(6)})^{\frac{3}{2}}
			(M_2\cdots M_6)^{-\sigma}\, P_{(1)}^{-\sigma}(P_{(2)}\cdots P_{(5)})^{\frac{3}{2}-\sigma}\|w\|_{H^{\sigma}}^{10}.
		\end{multline*}
Since we are in the regime $s\geq 10$ and $\sigma$ is close to $s-\frac{3}{2}$, we control  the right hand side by
$$ 
N_{(1)}^{2(s-1-\sigma)}\,N_{(3)}^{\frac{7}{2}-\sigma}\|w\|_{H^{\sigma}}^{10}\lesssim N_{(1)}^{2(s-1-\sigma)}\,N_{(1)}^{-\theta(\sigma-\frac{7}{2})}\|w\|_{H^{\sigma}}^{10}.
$$
For $s\geq 10$,  $\sigma$ close to $s-\frac{3}{2}$ and $\theta$ is close to $\frac{1}{3}$ the last expression can be estimated by $N_{(1)}^{-\epsilon_0(\theta,\sigma)}\|w\|_{H^{\sigma}}^{10}$ for some $\epsilon_0(\theta,\sigma)>0$ which is conclusive.
%%%%%%%%%%%%%%%%%%%%%%%%%%%%%%%%%%%%%%%%%%%%%%%%%%%%%%%%%%%%
		\vspace{0.3cm}
		
		\begin{flushleft}
			$\bullet$ Estimate of Type B contribution: 
		\end{flushleft}

		Denote $\Lambda_{B1}$ the set of $(k_1,\cdots,p_5)\in \Lambda_B$ such that  $k_{(1)},k_{(2)}\in \{k_2,k_3,k_4,k_5,k_6\}  $ and $\Lambda_{B2}$ the set of $(k_1,\cdots,p_5) \in\Lambda_B$ such that $k_{(1)},k_{(2)}\in\{p_1,p_2,p_3,p_4,p_5 \}$, and denote by $\mathcal{R}_{1,3}^{(B1)},\mathcal{R}_{1,3}^{(B2)}$ the corresponding multilinear expressions. 
		
		\noi
		$\bullet$ {\bf Subcase: Contribution $\mathcal{R}_{1,3}^{(B1)}$: }
		We first estimate $\mathcal{R}_{1,3}^{(B1)}$. By symmetry of indices, we may assume that $k_{(1)}=k_3, k_{(2)}=k_2$. Then other frequencies satisfy the constraint
		$$ \sum_{j=1}^5|p_j|+\sum_{j=4}^6|k_j|<|k_2|^{\theta}+|k_3|^{\theta}
		$$
		on $\Lambda_{B1}$.  
		We decompose $\mathcal{R}_{1,3}^{(B1)}$ by the dyadic sum
		$$ \sum_{M_1,\cdots,M_6, P_1,\cdots,P_5}\mathcal{R}_{1,3}^{(B1)}(M_1,\cdots,P_5)
		$$
		as in the estimate for Type (A) terms. Under the constraint of $\Lambda_B$ and our convention that $\{k_{(1)},k_{(2)}\}=\{k_3,k_2\}$, we must have $M_2\sim M_3\sim N_{(1)}$ and $\max\{M_1,M_4,M_5,M_6 \}\leq N_{(3)}$.
		
		Note that for the pairing part $k_2=k_3$ in $\mathcal{R}_{1,3}^{(B1)}(M_1,\cdots,P_5)$, we have $|\psi_{2s}(\vec{k})|\lesssim |k_{(3)}|^{2s}$, thus
		we can control it simply by 
		\begin{align}\label{innerpairingbound}   &\|w^{M_2}\|_{l^2}^2\cdot\!\!\!\!\!\sum_{\substack{k_1-k_4+k_5-k_6=0\\ k_1=p_1-p_2+p_3-p_4+p_5
			} }|k_{(3)}|^{2s}\prod_{j=1}^5|w_{p_j}^{P_j}|\prod_{j=4}^6|w_{k_j}^{M_j}|\notag \\
			\lesssim & N_{(1)}^{-2\sigma}N_{(3)}^{2s}\|w\|_{H^{\sigma}}^{10}\lesssim N_{(1)}^{2s\theta-2\sigma}\|w\|_{H^{\sigma}}^{10},
		\end{align}
		thanks to $\sigma>\frac{3}{2}$. As $\theta<\frac{1}{3}$, $s\geq 10$ and $\sigma$ is close enough to $s-\frac{3}{2}$, the right hand side is bounded by a negative power of $N_{(1)}$ times $\|w\|_{H^{\sigma}}^{10}$, which is conclusive.
		
		It remains to consider the non-pairing contribution in $\mathcal{R}_{1,3}^{(B1)}(M_1,\cdots,P_5)$.
		Recall that
		\begin{align*} &\mathcal{R}_{1,3}^{(B1)}(M_1,\cdots,P_5)\\= &\sum_{\substack{k_2\neq k_3\\
					|k_2|\sim M_2,|k_3|\sim M_3 } }\!\!\!\!\ov{w}_{k_2}^{M_2} w_{k_3}^{M_3}\!\!\!\!\!\!\sum_{\substack{k_4,k_5,k_6, p_1,\cdots,p_5
					\\ 
					p_1-\cdots+p_5-k_2+\cdots-k_6=0  } }\!\!\!\!\!\!\frac{\psi_{2s}(\vec{k})}{\Omega(\vec{k})}\Big(1-\chi\Big(\frac{\Omega(\vec{k})}{\lambda(\vec{k})^{\delta_0}}\Big)\Big)\ov{w}_{k_4}^{M_4}w_{k_5}^{M_5}\ov{w}_{k_6}^{M_6}w_{p_1}^{P_1}\cdots w_{p_5}^{P_5}.
		\end{align*}
		Denote $\mathcal{B}_{\ll M_{2}}$ the $\sigma$-algebra generated by $(g_k)_{|k|\leq M_2/100}$. Without loss of generality, we assume that $M_2\sim N_{(1)}$ is large enough such that $N_{(3)}\lesssim N_{(1)}^{\theta}\ll \frac{M_2}{100}$. Consequently, with respect to $\mu_s$, $w^{M_4},w^{M_5}, w^{M_6}, w^{P_1},\cdots, w^{P_5}$ are independent of $w^{M_2},w^{M_3}$.

		By conditional Wiener chaos estimate, we have
		\begin{align*}
			\|\mathcal{R}_{1,3}^{(B1)}(M_1,\cdots, P_5)\mathbf{1}_{B_R^{H^{\sigma}}	}(w)\|_{L^p(\dd\mu_s)}\leq &\|\mathcal{R}_{1,3}^{(B1)}(M_1,\cdots, P_5)\mathbf{1}_{B_R^{H^{\sigma}}	}(\mathbf{P}_{\ll  M_2}w)\|_{L^p(\dd\mu_s)}\\
			=&\big\| \|\mathcal{R}_{1,3}^{(B1)}(M_1,\cdots, P_5)\mathbf{1}_{B_R^{H^{\sigma}}	}(\mathbf{P}_{\ll M_2}w)\|_{L^p(\dd\mu_s|\mathcal{B}_{\ll M_2})} \big\|_{L^p(\dd\mu_s)}\\
			\lesssim &\, p\big\| \|\mathcal{R}_{1,3}^{(B1)}(M_1,\cdots, P_5)\|_{L^2(\dd\mu_s|\mathcal{B}_{\ll M_2})} \cdot \mathbf{1}_{B_R^{H^{\sigma}}	}(\mathbf{P}_{\ll M_2}w)\big\|_{L^p(\dd\mu_s)}.
		\end{align*}
		It suffices to show that
		\begin{align}\label{bdrandom:R13B1}
			\|\mathcal{R}_{1,3}^{(B1)}(M_1,\cdots, P_5)\|_{L^2(\dd\mu_s|\mathcal{B}_{\ll M_2})}\lesssim & N_{(1)}^{-\frac{1}{2}}\|w\|_{H^{\sigma}}^{8}.
		\end{align}
		Indeed, the above estimate yields
		$$ 	\|\mathcal{R}_{1,3}^{(B1)}(M_1,\cdots, P_5)\mathbf{1}_{B_R^{H^{\sigma}}	}(w) \|_{L^p(\dd\mu_s)}\lesssim pN_{(1)}^{-\frac{1}{2}}R^8.
		$$
		Since we have left a negative power of $N_{(1)}$, by interpolating with the crude deterministic bound
		which is of the form $N_{(1)}^{O(1)}$, we obtain the desired estimate.
		
		Now we prove \eqref{bdrandom:R13B1}. Thanks to the fact that $k_2\neq k_3$, we deduce that\footnote{In the summation below, we implicitly assume that $|k_2|\sim M_2, |k_3|\sim M_3.$}
		\begin{align*}
			&\|\mathcal{R}_{1,3}^{(B1)}(M_1,\cdots, P_5)\|_{L^2(\dd\mu_s|\mathcal{B}_{\ll M_2 })}\\ \lesssim &
			(M_2M_3)^{-s}\Big(\sum_{k_2\neq k_3}\Big(\!\!\!\!\!
			\sum_{\substack{k_4,k_5,k_6,p_1,\cdots,p_5\\
					p_1-\cdots+p_5-k_2+\cdots-k_6=0
			}  } \!\!\!\!\!
			\frac{\psi_{2s}(\vec{k})}{\Omega(\vec{k})}\Big(1-\chi\Big(\frac{\Omega(\vec{k})}{\lambda(\vec{k})^{\delta_0} } \Big) \Big) \ov{w}_{k_4}^{M_4}w_{k_5}^{M_5}\ov{w}_{k_6}^{M_6}w_{p_1}^{P_1}\cdots w_{p_5}^{P_5}
			\Big)^2 
			\Big)^{\frac{1}{2}}\\
			\lesssim & N_{(1)}^{-2s}\Big(\sum_{k_2\neq k_3}\sum_{\substack{k_4,k_5,k_6,p_1,\cdots,p_5\\
					p_1-\cdots+p_5-k_2+\cdots -k_6=0 }  } \frac{|\psi_{2s}(\vec{k})|^2}{1+|\Omega(\vec{k})|^2}
			|w_{k_{(3)}}^{N_{(3)}}|^2
			\Big)^{\frac{1}{2}}\Big(
			\sup_{k_2\neq k_3}\sum_{\substack{k_4,k_5,k_6,p_1,\cdots,p_5\\
					p_1-\cdots+p_5-k_2+\cdots-k_6=0
			} } |w_{k_{(4)}}^{N_{(4)}}\cdots w_{k_{(10)}}^{N_{(10)}}|^2\Big)^{\frac{1}{2}} \\
			\lesssim & N_{(1)}^{-2s}\Big(\prod_{j=4}^{10}\|w^{N_{(j)}}\|_{L^2}\Big)\cdot \Big(\sum_{\substack{k_2\neq k_3,k_4,k_5,k_6,p_1\cdots,p_5\\
					p_1-\cdots+p_5-k_2+\cdots -k_6=0
			} } \frac{M_{(1)}^{4(s-1)}(M_{(3)}^4+|\Omega(\vec{k})|^2 ) }{1+|\Omega(\vec{k})|^2}  |w_{k_{(3)}}^{N_{(3)}}|^2 
			\Big)^{\frac{1}{2}}.
		\end{align*}
		By Lemma \ref{threevectorcounting},  the last sum on the right hand side can be estimated as
		\begin{align*}
			&\sum_{\substack{k_4,k_5,k_6,p_1,\cdots,p_5  }  }|w_{k_{(3)}}^{N_{(3)}}|^2\!\!\!\!\!\sum_{\substack{k_2\neq k_3\\
					k_2-k_3=p_1-\cdots+p_5-k_4+k_5-k_6 } }\!\!\!\!\!
			\frac{N_{(1)}^{4(s-1)}( N_{(3)}^4+|\Omega(\vec{k})|^2) }{1+|\Omega(\vec{k})|^2}\\
			\lesssim & \sum_{k_4,k_5,k_6,p_1,\cdots, p_5}|w_{k_{(3)}}^{N_{(3)}}|^2N_{(1)}^{4(s-1)}(N_{(3)}^4N_{(1)}^2+N_{(1)}^3 ) \\
			\lesssim &\|w^{N_{(3)}}\|_{L^2}^2N_{(1)}^{4(s-1)}(N_{(3)}^4N_{(1)}^2+N_{(1)}^3)\prod_{j=4}^{10}N_{(j)}^3,
		\end{align*}
		thus
		\begin{align*}
			\|\mathcal{R}_{1,3}^{(B1)}(M_1,\cdots,P_5) \|_{L^2(\dd\mu_s|\mathcal{B}_{\ll M_2})} \lesssim & N_{(1)}^{-2}(N_{(1)}^{\frac{3}{2}}+N_{(3)}^2N_{(1)})
			\Big(\prod_{j=3}^{10}\|w^{N_{(j)}}\|_{L^2} \Big)
			\Big(\prod_{j=4}^{10}N_{(j)}^{\frac{3}{2}}\Big)\\
			\lesssim & N_{(1)}^{-\frac{1}{2}}N_{(3)}^{-\sigma}\prod_{j=4}^{10}N_{(j)}^{-\sigma+\frac{3}{2}}\cdot  \|w\|_{H^{\sigma}}^8+N_{(1)}^{-1}N_{(3)}^{-\sigma+2}\prod_{j=4}^{10} N_{(3)}^{-\sigma+\frac{3}{2}}\cdot \|w\|_{H^{\sigma}}^8\\
			\lesssim & N_{(1)}^{-\frac{1}{2}} \|w\|_{H^{\sigma}}^8,
		\end{align*}
		which is conclusive,
		thanks to the fact that $s\geq 10$ and that $\sigma$ is close to $s-\frac{3}{2}$.
		
		\vspace{0.3cm}
		\noi
		$\bullet$ {\bf Subcase: Contribution $\mathcal{R}_{1,3}^{(B2)}$: }
		Next we estimate $\mathcal{R}_{1,3}^{(B2)}$ for which $k_{(1)},k_{(2)}\in\{p_1,p_2,p_3,p_4,p_5\}$. By symmetry of indices, we assume that $k_{(1)}=p_1,k_{(2)}=p_2$, then
		$$ \sum_{j=3}^5|p_j|+\sum_{j=2}^6|k_j|\leq |p_1|^{\theta}+|p_2|^{\theta}
		$$
		on $\Lambda_{B2}$. Similarly, we decompose $\mathcal{R}_{1,3}^{(B2)}$ by dyadic sum
		$$ \sum_{M_1,\cdots,M_6, P_1,\cdots,P_5}\mathcal{R}_{1,3}^{(B2)}(M_1,\cdots,P_5),
		$$
		and this time, $P_1\sim P_2\sim N_{(1)}$ and $\max\{M_1,M_2,\cdots,M_6\}\lesssim N_{(3)}$. In particular, the energy weight $\psi_{2s}(\vec{k})$ satisfies $|\psi_{2s}(\vec{k})|\lesssim N_{(3)}^{2s}$ which is much smaller than in the previous case. 
		
		The pairing contribution $p_1=p_2$ in $\mathcal{R}_{1,3}^{(B2)}(M_1,\cdots,P_5)$ can be controlled in the same way by the same bound as \eqref{innerpairingbound}. We omit the detail.
		
		For the non-pairing contribution in $\mathcal{R}_{1,3}^{(B2)}$, again we apply the Wiener chaos estimate. 
		Denote $\mathcal{B}_{\ll P_{1}}$ the $\sigma$-algebra generated by $(g_k)_{|k|\leq P_1/100}$. Without loss of generality, we assume that $P_1\sim N_{(1)}$ is large enough such that $N_{(3)}\lesssim N_{(1)}^{\theta}\ll \frac{P_1}{100}$. Consequently, with respect to $\mu_s$, $w^{M_2},\cdots, w^{M_6}, w^{P_3},w^{P_4}, w^{P_5}$ are independent of $w^{P_1},w^{P_2}$. Recall that
		\begin{align*} &\mathcal{R}_{1,3}^{(B2)}(M_1,\cdots,P_5)\\= &\sum_{\substack{p_1\neq p_2\\
					|p_1|\sim P_1,|p_2|\sim P_2 } }\!\!\!\!w_{p_1}^{P_1}\ov{w}_{p_2}^{P_2}\!\!\!\!\!\!\sum_{\substack{k_2,\cdots,k_6, p_3,p_4,p_5
					\\ 
					p_1-\cdots+p_5-k_2+\cdots-k_6=0  } }\!\!\!\!\!\!\frac{\psi_{2s}(\vec{k})}{\Omega(\vec{k})}\Big(1-\chi\Big(\frac{\Omega(\vec{k})}{\lambda(\vec{k})^{\delta_0}}\Big)\Big)w_{p_3}^{P_3}\ov{w}_{p_4}^{P_4}w_{p_5}^{P_5}\ov{w}_{k_2}^{M_2}\cdots \ov{w}_{k_6}^{M_6}.
		\end{align*}
		By conditional Wiener chaos estimate, we have
		\begin{align*}
			\|\mathcal{R}_{1,3}^{(B2)}(M_1,\cdots, P_5)\mathbf{1}_{B_R^{H^{\sigma}}	}(w)\|_{L^p(\dd\mu_s)}\leq &\|\mathcal{R}_{1,3}^{(B2)}(M_1,\cdots, P_5)\mathbf{1}_{B_R^{H^{\sigma}}	}(\mathbf{P}_{\ll P_1}w)\|_{L^p(\dd\mu_s)}\\
			=&\big\| \|\mathcal{R}_{1,3}^{(B2)}(M_1,\cdots, P_5)\mathbf{1}_{B_R^{H^{\sigma}}	}(\mathbf{P}_{\ll P_1}w)\|_{L^p(\dd\mu_s|\mathcal{B}_{\ll P_1})} \big\|_{L^p(\dd\mu_s)}\\
			\lesssim &\, p\big\| \|\mathcal{R}_{1,3}^{(B2)}(M_1,\cdots, P_5)\|_{L^2(\dd\mu_s|\mathcal{B}_{\ll P_1})} \cdot \mathbf{1}_{B_R^{H^{\sigma}}	}(\mathbf{P}_{\ll P_1}w)\big\|_{L^p(\dd\mu_s)}.
		\end{align*}
		As in the estimate of $\mathcal{R}_{1,3}^{(B1)}$, here it suffices to show that
		\begin{align}\label{bdrandom:R13B2}
			\|\mathcal{R}_{1,3}^{(B2)}\|_{L^2(\dd\mu_s|\mathcal{B}_{\ll P_1})}
			\lesssim N_{(1)}^{-\frac{1}{2}}\|w\|_{H^{\sigma}}^8.
		\end{align}
		
	Thanks to the non-pairing condition $p_1\neq p_2$, we deduce that\footnote{In the summation below, we implicitly assume that $|p_1|\sim P_1, |p_2|\sim P_2.$}
		\begin{align*}
			\|\mathcal{R}_{1,3}^{(B2)}(M_1,\cdots,P_5)&\|_{L^2(\dd\mu_s|\mathcal{B}_{\ll P_1})}\\ \lesssim & (P_1P_2)^{-s}\Big(\sum_{ p_1\neq p_2}\Big(\sum_{\substack{k_2,\cdots,k_6,p_3,p_4,p_5\\
					p_1-\cdots+p_5-k_2+\cdots-k_6=0 } } N_{(3)}^{2s}|w^{P_3}_{p_3}w^{P_4}_{p_4}w^{P_5}_{p_5}w_{k_2}^{M_2}\cdots w_{k_6}^{M_6}|
			\Big)^2\Big)^{\frac{1}{2}}\\
			\lesssim & N_{(1)}^{-2s}N_{(3)}^{2s}\Big(\sum_{p_1-p_2+\cdots+p_5-k_2+\cdots-k_6=0} |w^{P_3}_{p_3}w^{P_4}_{p_4}w^{P_5}_{p_5}w_{k_2}^{M_2}\cdots w_{k_6}^{M_6}|^2 \Big)^{\frac{1}{2}}\\
			&\hspace{1.5cm}\times \Big(\sup_{p_1\neq p_2}\sum_{\substack{k_2,\cdots,k_6,p_3,p_4,p_5\\ 
					p_1-\cdots+p_5-k_2+\cdots-k_6=0
			} } \!\!\!\!\!1\Big)^{\frac{1}{2}}\\
			\lesssim & N_{(1)}^{-2s+\frac{3}{2}}N_{(3)}^{2s}\|w\|_{H^{\sigma}}^{8}\lesssim N_{(1)}^{-2s+\frac{3}{2}+2s\theta}\|w\|_{H^{\sigma}}^{8}\lesssim N_{(1)}^{-\frac{1}{2}}\|w\|_{H^{\sigma}}^8,
		\end{align*}
		thanks to the fact that $s\geq 10$ that $\theta$ is close to $\frac{1}{3}$ and the restriction $N_{(3)}\lesssim N_{(1)}^{\theta}$. 
		%The obtained bound is conclusive, since for the range of $s$, the power of $N_{(1)}$ is negative. 
		
		%%%%%%%%%%%%%%%%%%%%%%%%%%%%%%%%%%%%%%%%%%%%%%%%%%%5
		\vspace{0.3cm}
		
		\begin{flushleft}
			$\bullet$ Estimate of Type C contribution:
		\end{flushleft}	
		
		Without loss of generality, we assume that $k_{(1)}=p_1$ and $k_{(2)}=k_2$, since the other cases can be treated in the same way. In particular, $P_1\sim M_2\sim N_{(1)}$. We write
		\begin{align*}
			&\mathcal{R}_{1,3}^{(C)}(M_1,\cdots,P_5)\\
			=&\!\!\!\!\!\!\sum_{\substack{p_1\neq k_2\\
					|p_1|\sim P_1, |k_2|\sim M_2 }  }\!\!\!\! w_{p_1}^{P_1}\ov{w}_{p_2}^{P_2}\!\!\!\!\!\!
			\sum_{\substack{p_2\cdots,p_5, k_3,\cdots,k_6\\
					p_1-\cdots+p_5-k_2+\cdots-k_6=0 } }\!\!\!\!\!\!
			\frac{\psi_{2s}(\vec{k}) }{\Omega(\vec{k}) }\Big(1-\chi\Big(\frac{\Omega(\vec{k})}{\lambda(\vec{k})^{\delta_0} } \Big) \Big)\ov{w}_{p_2}^{P_2}\cdots w_{p_5}^{P_5}w_{k_3}^{M_3}\cdots \ov{w}_{k_6}^{M_6}.
		\end{align*}
		
		Denote $\mathcal{B}_{\ll P_{1}}$ the $\sigma$-algebra generated by $(g_k)_{|k|\leq P_1/100}$. Without loss of generality, we assume that $P_1\sim N_{(1)}$ is large enough such that $N_{(3)}\lesssim N_{(1)}^{\theta}\ll \frac{P_1}{100}$. Consequently, with respect to $\mu_s$, $w^{M_3},\cdots w^{M_6}, w^{P_2},\cdots, w^{P_5}$ are independent of $w^{P_1},w^{M_2}$.

		By conditional Wiener chaos estimate, we have
		\begin{align*}
			\|\mathcal{R}_{1,3}^{(C)}(M_1,\cdots, P_5)\mathbf{1}_{B_R^{H^{\sigma}}	}(w)\|_{L^p(\dd\mu_s)}\leq &\|\mathcal{R}_{1,3}^{(C)}(M_1,\cdots, P_5)\mathbf{1}_{B_R^{H^{\sigma}}	}(\mathbf{P}_{\ll  P_1}w)\|_{L^p(\dd\mu_s)}\\
			=&\big\| \|\mathcal{R}_{1,3}^{(C)}(M_1,\cdots, P_5)\mathbf{1}_{B_R^{H^{\sigma}}	}(\mathbf{P}_{\ll P_1}w)\|_{L^p(\dd\mu_s|\mathcal{B}_{\ll P_1})} \big\|_{L^p(\dd\mu_s)}\\
			\lesssim &\, p\big\| \|\mathcal{R}_{1,3}^{(C)}(M_1,\cdots, P_5)\|_{L^2(\dd\mu_s|\mathcal{B}_{\ll P_1})} \cdot \mathbf{1}_{B_R^{H^{\sigma}}	}(\mathbf{P}_{\ll P_1}w)\big\|_{L^p(\dd\mu_s)}.
		\end{align*}
		As in the estimate for Type (B) terms, it suffices to show that
		\begin{align*}
			\|\mathcal{R}_{1,3}^{(C)}(M_1,\cdots, P_5)\|_{L^2(\dd\mu_s|\mathcal{B}_{\ll P_1})}\lesssim & N_{(1)}^{-\frac{1}{2}}\|w\|_{H^{\sigma}}^{8}.
		\end{align*}
		Since $p_1\neq k_2$(recall that the contribution where $p_1=k_2$ is contained in $\mathcal{S}_{1,1}$), we estimate\footnote{In the summation below, we implicitly assume that the sum is taken in the range $|p_1|\sim P_1, |k_2|\sim M_2$. } 
		\begin{align}\label{bd:R13C}
			&\|\mathcal{R}_{1,3}^{(C)}(M_1,\cdots, P_5)\|_{L^2(\dd\mu_s|\mathcal{B}_{\ll P_1})}\notag  \\
			\lesssim & (P_1M_2)^{-s}\Big( 
			\sum_{p_1\neq k_2}\Big(\!\!\!\!\!\!
			\sum_{\substack{p_2,\cdots,p_5,k_3,\cdots,k_6\\
					p_1-\cdots+p_5-k_2+\cdots-k_6=0 } }\!\!\!\! \frac{\psi_{2s}(\vec{k})}{\Omega(\vec{k})}\Big(1-\chi\Big(\frac{\Omega(\vec{k})}{\lambda(\vec{k})^{\delta_0} } \Big) \Big)\ov{w}_{p_2}^{P_2}\cdots w_{p_5}^{P_5}w_{k_3}^{M_3}\cdots \ov{w}_{k_6}^{M_6}
			\Big)^2
			\Big)^{\frac{1}{2}} \notag  \\
			\lesssim &N_{(1)}^{-2s}\Big(
			\sum_{p_1\neq k_2 } \sum_{\substack{p_2\cdots,p_5,k_3,\cdots,k_6\\
					p_1-\cdots+p_5-k_2+\cdots-k_6=0
			} } \frac{|\psi_{2s}(\vec{k})|^2}{1+|\Omega(\vec{k})|^2}
			|w_{k_{(3)}}^{N_{(3)}}|^2
			\Big)^{\frac{1}{2}}\Big(\sup_{p_1\neq k_2}
			\sum_{\substack{p_2\cdots,p_5,k_3,\cdots,k_6\\
					p_1-\cdots+p_5-k_2+\cdots-k_6=0
			} } \prod_{j=4}^{10}|w_{k_{(j)}}^{N_{(j)}}|^2
			\Big)^{\frac{1}{2}} \notag 
			\\
			\lesssim & N_{(1)}^{-2s}\Big(\prod_{j=4}^{10}\|w^{N_{(j)}}\|_{L^2}\Big)\cdot \Big(\sum_{\substack{p_1\neq k_2,k_3,\cdots,k_6,p_2\cdots,p_5\\
					p_1-\cdots+p_5-k_2+\cdots -k_6=0
			} } \frac{|\psi_{2s}(\vec{k})|^2 }{1+|\Omega(\vec{k})|^2}  |w_{k_{(3)}}^{N_{(3)}}|^2 
			\Big)^{\frac{1}{2}}.
		\end{align}
		The estimate the last sum on the right hand side is very similar as for the estimate of $\mathcal{R}_{1,3}^{(B1)}$. The only difference here is that 
		we might have the pairing of $k_1=p_1-p_2+p_3-p_4+p_5$ and $k_2$, although $p_1\neq k_2$.  Note that in the case of pairing 
		$$
		k_1=p_1-(p_2-p_3+p_4-p_5)
		%\mathbf{p}
		=k_2
		$$
		we have $|\psi_{2s}(\vec{k})|\lesssim N_{(3)}^{2s}$, thus we control 
		the paired contribution crudely by
		\begin{align*}
			\sum_{p_1\neq k_2}\sum_{\substack{p_2\cdots,p_5,k_3,\cdots,k_6\\
					p_1-\cdots+p_5=k_2\\
					k_3-k_4+k_5-k_6=0}}  N_{(3)}^{4s}|w_{k_{(3)}}^{N_{(3)}}|^2\lesssim 	N_{(1)}^3N_{(3)}^{4s}\|w^{N_{(3)}}\|_{L^2}^2\prod_{j=4}^{10}N_{(j)}^3
		\end{align*}
		For the non-pairing contribution, we can argue exactly as the last part of the estimate of $\mathcal{R}_{1,3}^{(B1)}$ by using Lemma \ref{threevectorcounting}:
		\begin{align}\label{lastsumTypeC}
			\sum_{k_3\cdots,k_6,p_2\cdots,p_5}|w_{k_{(3)}}^{N_{(3)}}|^2\sum_{\substack{k_1\neq k_2\\
					k_1=p_1-p_2+p_3-p_4+p_5\\
					k_1-k_2+k_3-k_4+k_5-k_6=0
			} }\frac{M_{(1)}^{4(s-1)}(M_{(3)}^4+|\Omega(\vec{k})|^2 ) }{1+|\Omega(\vec{k})|^2 }.
		\end{align}
		For fixed $p_2,\cdots,p_5, k_3,\cdots,k_6$,
		$$ |\Omega(\vec{k})|=|p_1-\mathbf{p}|^2-|k_2|^2+\mathbf{c}
		$$
		with $\mathbf{p}=p_2-p_3+p_4-p_5$ and $\mathbf{c}=|k_3|^2-|k_4|^2+|k_5|^2-|k_6|^2.$ When $\mathbf{p}\neq p_1-k_2$, by Lemma \ref{threevectorcounting}, the choices of $p_1,k_2$ such that $p_1-\mathbf{p}-k_2=\mathbf{k}=k_3-k_4+k_5-k_6$ are bounded by $M_{(1)}^2$, hence \eqref{lastsumTypeC} is bounded by
		$$ N_{(1)}^{4(s-1)}(N_{(3)}^4N_{(1)}^2+N_{(1)}^3)\|w^{N_{(3)}}\|_{L^2}^2\prod_{j=4}^{10}N_{(j)}^3.
		$$  
		Therefore, the right hand side of \eqref{bd:R13C} can be bounded by
		$$ \big(N_{(1)}^{-2s+\frac{3}{2}}N_{(3)}^{2s}+N_{(1)}^{-2}(N_{(1)}^{\frac{3}{2}}+N_{(3)}^2N_{(1)})\big)\prod_{j=3}^{10}\|w^{N_{(j)}}\|_{L^2}\cdot \Big(\prod_{j=4}^{10}N_{(j)}^{\frac{3}{2}}\Big)\lesssim N_{(1)}^{-\frac{1}{2}}\|w\|_{H^{\sigma}}^8,
		$$
thanks to the fact that  $s\geq 10$, that  $\theta<\frac{1}{3}$, $\theta$ close to $\frac{1}{3}$ and that $N_{(3)}\lesssim N_{(1)}^{\theta}$. This completes the proof of Proposition \ref{prop:R13}.
\end{proof}	
%%%%%%%%%%%%%%%%%%%%%%%%%%%%%%%%%%%%%%%%%%%%%%%%%%%%%%%%%%%%%
%%%%%%%%%%%%%%%%%%%%%%%%%%%%%%%%%%%%%%%%%%%%%%%%%%%%%%%%%%%%%%%	
%%%%%%%%%%%%%%%%%%%%%%%%%%%%%%%%%%%%%%%%%%%%%%%%%%%%%%%%%%%%%
\appendix	
\section{Long time approximations}
In this appendix, we prove the approximation results used in Section \ref{proof:quasi}. The proof is a consequence of the global regularity theory of \cite{IoPau}.
\subsection{Ingredients in the global regularity theory for the energy critical NLS on $\mathbb{T}^3$}
In the sequel, we follow the notations in \cite{IoPau} and \cite{HTTz} about basic definitions and properties for function spaces $U^p, V^p, U_{\Delta}^p, V_{\Delta}^p$ related to critical problems. Let us briefly recall some related function spaces as well as multilinear estimates from  \cite{HTTz} and \cite{IoPau} that we will use. For $s\in\mathbb{\R}$,
\begin{align}\label{def:XsY^s} 
&\|u\|_{\widetilde{X}^s(\R)}:=\Big(\sum_{k\in\mathbb{Z}^3} \langle k\rangle^{2s}\|\e^{it|k|^2}(\mathcal{F}u)(t,k) \|_{U_t^2}^2 \Big)^{\frac{1}{2}},\\
&\|u\|_{\widetilde{Y}^s(\R)}:=\Big(\sum_{k\in\mathbb{Z}^3} \langle k\rangle^{2s}\|\e^{it|k|^2}(\mathcal{F}u)(t,k) \|_{V_t^2}^2 \Big)^{\frac{1}{2}}.
\end{align}
We have the continuous embedding property:
$$ \widetilde{X}^s(\R)\hookrightarrow \widetilde{Y}^s(\R)\hookrightarrow L^{\infty}(\R; H^s(\T^3)).
$$
For intervals $I\subset\R$, the space $X^s(I)$ is defined via the restriction norms:
$$ \|u\|_{X^s(I)}:=\sup_{J\subset I,|J|\leq 1 }\inf_{v\mathbf{1}_J(t)=u\mathbf{1}_J(t)}\|v\|_{\widetilde{X}^s}.
$$
Similarly for the space $Y^s(I)$. Note that by definition, for linear solution $u(t)=\e^{it\Delta}\phi$,
\begin{align}\label{linearestimate}
 \|u(t)\|_{X^s(I)}\leq \|\phi\|_{H^s(\T^3)}.
\end{align}
 The critical Strichartz type norm is defined via the norm
$$ \|u\|_{Z(I)}:=\sum_{p\in\{p_0,p_1\} }\sup_{J\subset I, |J|\leq 1}\Big(\sum_{N\in 2^{\mathbb{N}} }N^{5-\frac{p}{2}}\|\mathbf{P}_Nu(t)\|_{L_{t,x}^p(J\times\T^3)}^p
\Big)^{\frac{1}{p}},\quad p_0=4+1/10,\, \,p_1=100,
$$
where $\mathbf{P}_N=\mathbf{P}_{\leq N}-\mathbf{P}_{\leq N/2}$, and $\mathbf{P}_{\leq N}$ are square Littlewood-Paley projectors defined in Section~2 of \cite{IoPau}.

By definition we remark that if $T\geq 1$ and $I_T=[-T,T]$,
$$ \|u\|_{Z(I_T)}\sim_T \sum_{p\in\{p_0,p_1\} }\|u(t)\|_{L_t^p([-T,T];\widetilde{B}_{p,p}^{\frac{5}{p}-\frac{1}{2}}(\T^3) ) },
$$
while if $T<\frac{1}{2}$,
$$ \|u\|_{Z(I_T)}= \sum_{p\in\{p_0,p_1\} }\|u(t)\|_{L_t^p([-T,T];\widetilde{B}_{p,p}^{\frac{5}{p}-\frac{1}{2}}(\T^3) ) }
$$
where $\widetilde{B}_{p,q}^{s}$ are Besov spaces related to the Littlewood-Paley projectors $\mathbf{P}_N$. 

The inhomogeneous term on an interval $I=(a,b)$ will be controlled by the $N^s(I)$ norm:
$$ \|F\|_{N^s(I)}:=\Big\|\int_a^t \e^{i(t-t')\Delta}F(t')\dd t' \Big\|_{X^s(I)}.
$$
It turns out that (\cite{HTTz}, Proposition 2.11) 
\begin{align*}
\|F\|_{N^s(I)}\leq \sup_{\substack{G\in Y^{-s}(I)
\\
\|G\|_{Y^{-s}(I)}\leq 1 } }\Big|\int_I\int_{\T^3}F(t,x)\ov{G}(t,x)\dd x \dd t \Big|.
\end{align*}
Recall the key Strichartz estimate:
\begin{lem}[\cite{IoPau}, Corollary 2.2]\label{Strichartz}
Let $p\in(4,\infty)$ and $\mathbf{P}_{C}$ the frequency projector to some cube $C$ of size $N$. For any interval $I\subset\R$, $|I|\leq 1$, 
$$ \|\mathbf{P}_Cu\|_{L_{t,x}^p(I\times \T^3)}\lesssim N^{\frac{3}{2}-\frac{5}{p}}\|u\|_{U_{\Delta}^p(I;L_x^2)},
$$ 
where the implicit constant is independent of intervals $I$.
\end{lem}
As a consequence of Lemma~\ref{Strichartz} and the embedding $X^0(I)\hookrightarrow  U_{\Delta}^p(I;L_x^2)$ (basically since $U^2\hookrightarrow U^p$) for $p>2$, we have
\begin{align}\label{Strichartzcor:}
\|u\|_{Z(I)}\lesssim \|u\|_{X^1(I)}
\end{align}
for any interval $I$, where the implicit constant is uniform. The key multilinear estimate we will use reads:
\begin{lem}[\cite{IoPau}, Lemma 3.2]\label{multilinear}
Let $\sigma\geq 1$.	For $u_j\in X^1(I)$, $j=1,2,3,4,5$, $|I|\leq 1$.  the estimate
\begin{align}\label{Multilinears} \Big\|\prod_{j=1}^5u_j^{\pm} \Big\|_{N^{\sigma}(I)}\lesssim \sum_{\sigma\in\mathfrak{S}_5 }\|u_{\sigma(1)}\|_{X^{\sigma}(I)}\prod_{j\geq 2}\|u_{\sigma(j)}\|_{Z(I)}^{\frac{1}{2}}\|u_{\sigma(j)}\|_{X^1(I)}^{\frac{1}{2}}
\end{align}
holds true,
where $\mathfrak{S}_5$ is the permutation group of $5$ elements and $u_j^{\pm}\in\{u_j,\ov{u}_j \}$, and the implicit constant in the inequality is independent of intervals $I$ such that $|I|\leq  1$.
\end{lem}
We remark that \cite{IoPau} treats the case $\sigma=1$. For $\sigma>1$, the proof follows in the similar way by the more precise estimate
\begin{align*}
& \int_{I\times\T^3}\Big|\sum_{N_0,N_1\geq N_2\geq N_3\geq N_4\geq N_5 }\prod_{j=0}^5\mathbf{P}_{N_j}u_j^{\pm}  \Big|\dd x\dd t \\ \lesssim 
&\Big(\frac{N_0}{N_1}\Big)^{\sigma}\Big(\frac{N_5}{N_1}+\frac{1}{N_3} \Big)^{\delta}\Big(\frac{N_4}{N_0}+\frac{1}{N_2}\Big)^{\delta}
\|\mathbf{P}_{N_1}u_1\|_{Y^{\sigma}(I)}\|\mathbf{P}_{N_0}u_0\|_{Y^{-\sigma}(I)} \prod_{j=2}^5\|\mathbf{P}_{N_j}u_j\|_{Z(I)}^{\frac{1}{2}}\|\mathbf{P}_{N_j}u_j\|_{X^1(I)}^{\frac{1}{2}}
\end{align*}
for some $\delta>0$. We omit the details of the proof.
\vspace{0.3cm}

Finally, we recall the global regularity theory of Ionescu-Pausader. Following Section 6 of \cite{IoPau}, given $R>0$ and $\tau\geq 0$, consider the non-negative function (possibly $\infty$)
$$ \Sigma(R,\tau):=\sup\big\{\|u\|_{Z(I)}^2:\; H(u)\leq R, |I|\leq \tau \big\},
$$
where $H(u)$ is the energy of $u$ and the supremum is taken over all strong solutions of \eqref{NLS} of energy less than or equal to $R$ and all intervals $I$ of length $|I|\leq \tau$. As an increasing function in $\tau$, the limit (possibly $\infty$ in a priori)
$$ \Sigma_{*}(R):=\lim_{\tau\rightarrow 0^+}\Sigma(R,\tau)
$$
exists. Moreover, $\Sigma(R,\tau)$ is quasi subadditive in $\tau$:
$$ \Sigma(R,\tau_1+\tau_2)\lesssim \Sigma(R,\tau_1)+\Sigma(R,\tau_2)
$$ for any $\tau_1,\tau_2>0$.

The global regularity result of Ionescu-Pausader can be stated as:
\begin{thm}[\cite{IoPau}]\label{GWP}
For any $R>0$, $\Sigma_*(R)<\infty$. Consequently, for any $\tau>0$, $\Sigma(R,\tau)<\infty$ and moreover
$$\Sigma(R,\tau)\leq \Sigma(R,1)\e^{C_0(1+\tau)},
$$
where $C_0>0$ is an absolute constant. In particular, for any $\phi\in H^1(\T^3)$ of energy smaller than or equal to $R$, the strong solution $u(t)$ of \eqref{NLS} with initial data $\phi$ is global and
$$ \|u(t)\|_{Z([0,\tau])}\leq \Sigma(R,1)\e^{C_0(1+\tau)}.
$$ 
Finally,  
$$ \|u\|_{X^1([0,\tau])}\leq C(R,\|u\|_{Z([0,\tau])} ).
$$
\end{thm}
We denote $\Phi(t)$ the global flow of \eqref{NLS} in $H^1(\T^3)$. The following corollary shows that 
the $H^{\sigma}(\T^3)$, $\sigma\geq 1$ regularity is propagated by  $\Phi(t)$. 
%on $H^{\sigma}(\T^3)$
%we can extend $\Phi(t)$ on $H^{\sigma}(\T^3)$ for $\sigma\geq 1$:
\begin{cor}\label{Hsigmabound} 
Let $\sigma\geq 1$. Then $\Phi(t)(H^{\sigma}(\T^3))= H^{\sigma}(\T^3) $ for every $t\in\R$.
% and $T\geq 1$. Assume that $\phi\in H^{\sigma}(\T^3)$ such that $H(\phi)\leq R$, then the flow $\Phi(t)$ of \eqref{NLS} can be extended on $H^{\sigma}(\T^3)$ globally.
\end{cor}
%%%
\begin{proof}
Assume that $\phi\in H^{\sigma}(\T^3)$ such that $H(\phi)\leq R$. Denote $u(t)=\Phi(t)\phi$.  From Theorem~\ref{GWP}, 
%for any $I\subset[-T,T]$, 
$$
\|u\|_{Z([-T,T])}\leq \Sigma(R,T),\quad\|u\|_{X^1([-T,T])}\leq C(R,T),
$$ 
where $C(R,T)$ depends only on $R$ and $T$.   By the Duhamel formula and \eqref{Multilinears} of Lemma \ref{multilinear}, for any $I=(a,b)\subset [-T,T]$,
\begin{align*}
\|u\|_{X^{\sigma}(I)}\leq &\|e^{i(t-a)\Delta}u(a)\|_{X^{\sigma}(I)}+\||u|^4u\|_{N^{\sigma}(I)}\\ \leq &\|u(a)\|_{H_x^{\sigma}}+C\|u\|_{X^{\sigma}(I)}\|u\|_{X^1(I)}^2\|u\|_{Z(I)}^2\\
\leq & \|u(a)\|_{H_x^{\sigma}}+C(R,T)\|u\|_{Z(I)}^2\|u\|_{X^{\sigma}(I)},
\end{align*}
where $C(R,T)$ is a constant depending only on $R$ and $T$ that can change from line to line.

Next, we partition $[0,T]=\bigcup_{j=1}^{\kappa}[a_{j-1},a_j]$ such that
$ \|u\|_{Z([a_{j-1},a_j])}<\frac{1}{\sqrt{2C(R,T)}}, 
$
hence for all $j=1,2,\cdots,k$,
$$ \|u\|_{X^{\sigma}([a_{j-1},a_j] )}\leq 2\|u(a_{j-1})\|_{H_x^{\sigma}}.
$$
By the embedding property, for all $j\geq 1$,
$$ \|u\|_{X^{\sigma}([a_{j},a_{j+1}] )}\leq C\|u\|_{X^{\sigma}([a_j,a_{j-1}]) }.
$$
This shows that for all $t\in[-T,T]$,
$$ \|u(t)\|_{H^{\sigma}}\leq C^{\kappa}\|\phi\|_{H_x^{\sigma}}.
$$
Therefore  $\Phi(t)(H^{\sigma}(\T^3))\subset H^{\sigma}(\T^3) $. Using the time reversibility completes the proof.
\end{proof}
%%%%%%%%
\begin{rem}\label{RemarkHsigma}
When $\sigma>1$, the above argument does not give a uniform control of the $H^{\sigma}$ norm for the solution, since the index $\kappa$ depends on the profile of each individual global $H^1$-solution $u(t)$. Later we shall strengthen the $H^{\sigma}$-estimate uniformly on any bounded ball of $H^{\sigma}$ by  choosing a uniform partition of $[0,T]$.
\end{rem}

\subsection{Local convergence and stability}

From now on, denote $\Phi(t)$ the flow of the energy critical NLS. Denote $\Phi_N(t)$ the flow of the truncated NLS:
\begin{align}\label{5NLS-N}
	i\partial_tu_N+\Delta u_N=S_N(|S_Nu_N|^4S_Nu_N),
\end{align}
where $S_N$ is the smooth Fourier truncation at size $N$ defined at the beginning of Section \ref{Sec:modifiedenergy}.

\begin{prop}[Local convergence]\label{LWP}	
Assume that $\sigma\geq 1$.	Let $\phi,\widetilde{\phi}\in H^{\sigma}(\mathbb{T}^3)$ and $I\subset \mathbb{R}$ be an interval and $t_0\in I$. Suppose $\|\phi\|_{H_x^{\sigma}}\leq A,\|\widetilde{\phi}\|_{H_x^{\sigma}}\leq  A$. Then for any $\epsilon>0$, there exist $\delta=\delta(A,\epsilon)>0$ such that if
	$$ \|e^{i(t-t_0)\Delta}\phi\|_{Z(I)}<\delta,\; \|\phi-\widetilde{\phi}\|_{H_x^{\sigma}}<\delta,
	$$
	there exist unique solutions $u=\Phi(t-t_0)\phi $ and $u_N=\Phi_N(t-t_0)\widetilde{\phi}$ in $C(\ov{I}; H_x^{\sigma})\cap X^{\sigma}(I)$ satisfying
	$$ \|u_N\|_{X^{\sigma}(I)}+\|u\|_{X^{\sigma}(I)}\leq C_0A,\; \|u_N\|_{Z(I)}+\|u\|_{Z(I)}<\epsilon,
	$$
	where $C_0>0$ is an absolute constant.
	Moreover, 
\begin{align}\label{localboundgeneral}
 \|\Phi_N(t)\widetilde{\phi}-\Phi(t)\phi\|_{C(\ov{I};H_x^{\sigma} ) }\leq  C_0\|\phi-\widetilde{\phi}\|_{H_x^{\sigma}}+C_0\delta_{N}(A,\epsilon,\phi),
\end{align}
where $\delta_N(A,\epsilon,\phi)\rightarrow 0$ as $N\rightarrow\infty$, uniformly in $\phi$ on a compact set $K$ of $H^{\sigma}(\T^3)$ such that $\|\phi\|_{H_x^{\sigma}}\leq A$.
\end{prop}
\begin{rem}\label{rk:stability}
Consequently, taking $\widetilde{\phi}=\phi$, under the hypothesis of Proposition~\ref{LWP}, we have
$$ \|\Phi_N(t)\phi-\Phi(t)\phi\|_{X^{\sigma}(I)}\rightarrow 0,
$$
uniformly on any compact set of $H^1(\T^3)$. Taking the limit $N\rightarrow \infty$ in \eqref{localboundgeneral}, we obtain also that under the hypothesis of 
Proposition~\ref{LWP},
$$ \|\Phi(t)\widetilde{\phi}-\Phi(t)\phi\|_{X^{\sigma}(I)}\leq C_0\|\phi-\widetilde{\phi}\|_{H_x^{\sigma}}.
$$
\end{rem}

\begin{proof}
First we prove the case when $\sigma=1$. Note that the existence of solutions on $\ov{I}$ is a direct consequence of Proposition~3.3 of \cite{IoPau}. So we will only concentrate on the bounds for solutions $u_N(t)$ and $u(t)$. We argue by several steps.

\begin{flushleft}	
\underline{Step 1: Uniform bound}:
\end{flushleft}
	Without loss of generality, we assume that $t_0=0$.
	By the Strichartz estimate \eqref{Strichartzcor:} and Lemma~\ref{multilinear}, we have
	\begin{align}\label{appendix:Zbound}
		\|u\|_{Z(I)}\leq   &\|\e^{it\Delta}\phi\|_{Z(I)}
		+C\||u|^4u \|_{N^1(I)} \notag 
		\\
		\leq & \|\e^{it\Delta}\phi\|_{Z(I)}
		+C\|u\|_{Z(I)}^2\|u\|_{X^1(I)}^3,
	\end{align}
and
	\begin{align}\label{appendix:X1bound}
	\|u\|_{X^1(I)}\leq   &\|\e^{it\Delta}\phi\|_{Z(I)}
	+C\||u|^4u \|_{N^1(I)} \notag 
	\\
	\leq & \|\phi\|_{H_x^1}
	+C\|u\|_{Z(I)}^2\|u\|_{X^1(I)}^3,
\end{align}
where $C>0$ is independent of $I$. 
	Note that the same inequalities holds for $u_N$, as the smooth spectral projector
	$S_N$ is bounded from $L_x^r$ to $L_x^r$ for any $1<r<\infty$.
	Then the desired control for $\|u\|_{Z(I)}$ and $\|u\|_{X^1(I)}$, as well as $\|u_N\|_{Z(I)}, \|u_N\|_{X^1(I)}$ follow from the following elementary lemma:
	\begin{lem}\label{elementary} 
		Let $A>0, C_0>0$ and $f\in C([0,\tau_0];[0,\infty))$ be an increasing continuous function such that $f(0)=0$ and $h:[0,\tau_0]\rightarrow [0,\infty)$ is increasing. Then for any small $\epsilon>0$, there exists $\delta=\delta(A,\epsilon)>0$, such that if
		$$ f(t)\leq \delta+C_0f(t)^4h(t),\quad  h(t)\leq C_0A+C_0f(t)^4h(t),\quad \forall t\in[0,\tau_0],
		$$
		then $f(t)\leq \epsilon$ and $h(t)\leq 2C_0A$.
	\end{lem}
	
	\begin{proof}
		This follows from a standard continuity argument. Let $\tau\leq \tau_0$ be the largest time such that $f(t)\leq \epsilon$ ($\tau$ exists since $f(0)=0$). We claim that if $\epsilon\ll 1$ is such that $16C_0\epsilon^4<\frac{1}{2}$ and $32C_0^2A\epsilon^3<\frac{1}{2}$, then $\tau=\tau_0$. By contradiction, if $\tau<\tau_0$, by continuity of $f$, there exists $\tau_1\in(\tau,\tau_0)$ such that $f(\tau_1)<2\epsilon$. Then for all $0\leq t\leq \tau_1$,
		$ h(t)\leq C_0A+16C_0\epsilon^4 h(t)<2C_0A
		$, thanks to the smallness of $\epsilon$. Plugging into the inequality of $f(t)$, we deduce that for all $0\leq t\leq \tau_1$
		$$ f(t)\leq \delta+32C_0^2A\epsilon^4.
		$$
		Choosing $\delta<\frac{\epsilon}{2}$, hence, $\delta+32C_0^2A\epsilon^4<\epsilon$, we obtain a contradiction.
	\end{proof}

\begin{flushleft}	
	\underline{Step 2:  Quantitative convergence}:
\end{flushleft}
	
 By the same application of the Strichartz inequality, we get
	\begin{align*}
		\|u_N-u\|_{X^1(I)}\leq \|\phi_N-\phi\|_{H_x^1}+\|S_N^{\perp}(|S_Nu_N|^4S_Nu_N )\|_{N^1(I) }+\||S_Nu_N|^4S_Nu_N-|u|^4u\|_{N^1(I)}.
	\end{align*} 
	By splitting 
	$ S_Nu_N=S_Mu_N+S_M^{\perp}S_Nu_N,
	$ we observe that for $M=\frac{N}{16}$, $S_N^{\perp}(|S_Mu_N|^4S_Mu_N )=0$. Therefore, by the uniform boundedness of $S_N,S_M$ on $L_x^r$ ($1<r<\infty$) and on $X^1$,
	\begin{align}\label{ine:SNperp}
		\|S_N^{\perp}(
		|S_Nu_N|^4S_Nu_N
		)\|_{N^1(I)}\leq & C\|S_M^{\perp}u_N\|_{Z(I)}\|u_N\|_{Z(I)}^3\|u_N\|_{X^1(I)}^3\notag   \\
		+& C\|u_N\|_{Z(I)}^2\|S_M^{\perp}u_N\|_{X^1(I)}\|u_N\|_{X^1(I)}^2\notag \\
		\leq &C\|S_M^{\perp}u_N\|_{X^1(I)}\|u_N\|_{Z(I)}^2\|u_N\|_{X^1(I)}^2\\
		\leq & C(\|S_M^{\perp}u\|_{X^1(I)}+\|u_N-u\|_{X^1(I)})\|u_N\|_{Z(I)}^2\|u_N\|_{X^1(I)}^2\notag \\
	\leq  &CA^2\epsilon^2(\|S_M^{\perp}u\|_{X^1(I)}+\|u_N-u\|_{X^1(I)}). \notag
	\end{align}
	For the other term, algebraic manipulation yields
	\begin{align*}
		\||S_Nu_N|^4S_Nu_N-|u|^4u \|_{N^1(I)}\leq & C\|S_Nu_N-u\|_{X^1(I)}(\|S_Nu_N\|_{Z(I)}^{2}+\|u\|_{Z(I)}^2)(\|S_Nu_N\|_{X^1(I)}^2+\|u\|_{X^1(I)}^2)\\
	\leq  &CA^2\epsilon^2(\|S_N^{\perp}u\|_{X^1(I)}+\|u_N-u\|_{X^1(I)})\\
	\leq  &CA^2\epsilon^2(\|S_M^{\perp}u\|_{X^1(I)}+\|u_N-u\|_{X^1(I)}  ).
	\end{align*}
	In summary, we have
	\begin{align}\label{iteration1}
		\|u_N-u\|_{X^1(I)}\leq & \|\phi-\widetilde{\phi} \|_{H_x^1}+CA^2\epsilon^2(\|S_{N/16}^{\perp}u\|_{X^1(I)}+\|u_N-u\|_{X^1(I)}).
	\end{align}
	Furthermore, from the Duhamel formula of $u$ and the similar argument as \eqref{ine:SNperp}, we have the recursive inequality
	\begin{align}\label{iteration2} 
		\|S_N^{\perp}u\|_{X^1(I)}\leq & \|S_N^{\perp}\phi\|_{H_x^1}+CA^2\epsilon^2\|S_{N/16}^{\perp}u\|_{X^1(I)}.
	\end{align}
To conclude, we invoke the following elementary result:
	\begin{lem}
		Let $\{a_j\},\{b_j\}$ be two positive sequences and $A_0>0$ is an absolute constant. Assume that  $0<\theta<1$ and for $1\leq j\leq m$,
		$$ a_{j}\leq A_0b_{j}+\theta a_{j-1}.
		$$
		Then we have
		\begin{align*}
			a_m\leq A_0\sum_{j=0}^{m-1}\theta^{j}b_{m-j}+ \theta^{m-1}a_1.
		\end{align*}
		In particular, if $b_m\rightarrow 0$, then $a_m\rightarrow 0$.
	\end{lem}
	The proof of this elementary lemma is straightforward hence we omit the detail.
	
	Consequently, we have
	\begin{align*}
		\|S_{N/16}^{\perp}u\|_{X^1(I)}\leq A_0\sum_{j=1}^{\log_{16}N-1}(CA\epsilon)^j\|S_{N/16^j}^{\perp}\phi\|_{H_x^1}+(CA\epsilon)^{\log_{16}N-1}\|\phi\|_{H_x^1},
	\end{align*}
	where $C,A>0$ are absolute constants and $\epsilon\ll 1$ such that $CA\epsilon<1$. We denote
	\begin{align}\label{deltaN} \delta_{N}(A,\epsilon,\phi):=\sum_{j=1}^{\log_{16}N-1}(CA\epsilon)^j\|S_{N/16^j}^{\perp}\phi\|_{H_x^1}+(CA\epsilon)^{\log_{16}N-1}\|\phi\|_{H_x^1}.
	\end{align}
Since $\|S_N^{\perp}\phi\|_{H_x^1}\rightarrow 0$, uniformly on any compact set of $H^1$, we deduce that	$\delta_N(A,\epsilon,\phi)$ converges to $0$, uniformly on a compact set of $H^1(\T^3)$. 
	Plugging into \eqref{iteration1}, we obtain that
	\begin{align}\label{quantitativebound}
		\|u_N-u\|_{X^1(I)}\leq A_0\|\phi-\widetilde{\phi}\|_{H_x^1}+A_0\delta_{N}(A,\epsilon,\phi).
	\end{align}
By the embedding $X^1(I)\hookrightarrow L^{\infty}(I;H_x^1)$ and the fact that $u_N(t),u(t)\in C(\ov{I};H_x^1)$, we obtain that
$$ \sup_{t\in\ov{I}}\|u_N(t)-u(t)\|_{H_x^1}\leq C_0\|\phi-\widetilde{\phi}\|_{H_x^1}+C_0\delta_{N}(A,\epsilon,\phi),
$$
for some absolute constant $C_0>0$.
	This completes the proof of Lemma \ref{LWP} when $\sigma=1$.
	
The general case $\sigma>1$ follows from similar analysis. Here we only indicate the necessary modification in the proof. 
Indeed, in Step 1, we replace $X^1,H^1$ norms by $X^{\sigma}, H^{\sigma}$ norms in \eqref{appendix:X1bound}, thanks to  \eqref{Multilinears}.
In Step 2, all inequalities remain unchanged when replacing all $X^1, H^1, N^1$ norms by $X^{\sigma}, H^{\sigma}, N^{\sigma}$ norms (up to change the numerical constants $C$ in front of each inequality), thanks to \eqref{Multilinears} and the trivial embedding $X^{\sigma}\hookrightarrow X^1$. This completes the proof Lemma \ref{LWP}.
\end{proof}
%%%%%%%
\begin{definition}[$(A,\delta)$-partition]\label{partition}
	Let $A>\delta>0$. Given an interval $[-T,T]$ and $\phi\in H^1(\T^3)$, we define an $(A,\delta)$-partition with respect to $\phi$ a collection of finite intervals $([\tau_{j-1},\tau_j])_{j=1}^m$ such that 
	$$ -T=\tau_0<\tau_2<\cdots<\tau_m=T,\quad \|\Phi(\tau_{j-1})\phi\|_{H_x^1}\leq A,\; \|\Phi(t)\phi\|_{Z([\tau_{j-1},\tau_j])}<\delta,\;\forall j=1,\cdots,m. 
	$$
\end{definition}
We collect some basic properties. The following property is immediate:
\begin{prop}[Refinement of $(A,\delta)$-partition]\label{refinement} 
	Any refinement of an $(A,\delta)$-partition with respect to $\phi$ is an $(A,\delta)$-partition.
\end{prop}

\begin{prop}\label{integrabilitypartition}
	Assume that $([\tau_{j-1},\tau_j])_{j=1}^m$ is an $(A,\delta)$-partition with respect to $\phi$. Then for sufficiently small $\delta=\delta(A)>0$,
	$$ \|\e^{i(t-\tau_{j-1})\Delta}\Phi(\tau_{j-1})\phi\|_{Z([\tau_{j-1},\tau_j])}<2\delta,\; \forall j=1,\cdots,m.
	$$
\end{prop}
\begin{proof}
	Denote $u(t)=\Phi(t)\phi$, then
	$$ \e^{i(t-\tau_{j-1})\Delta}u(\tau_{j-1})=u(t)-\frac{1}{i}\int_{\tau_{j-1}}^t\e^{i(t-t')\Delta}(|u(t')|^4u(t'))\dd t'.
	$$
	The desired consequence follows from the Strichartz inequality as in the proof of Lemma \ref{LWP}. Hence we omit the detail.
\end{proof}
%%%%%%%%%%%%%%%%%%%%%%%%%%%%%%%%%%%%%%%%%%%%%%%%%%%%%%%%%%%%%%%%%%%%%%%%
\begin{prop}[$H^{1}$-Stability of an $(A,\delta)$-partition]\label{stabilitypartition} 
	Assume that $([\tau_{j-1},\tau_j])_{j=1}^m$ is an $(A,\delta)$-partition with respect to $\phi$. There exists $\delta_1>0$ such that for any $\delta<\delta_1$, there exists $\epsilon_0=\epsilon_0(m,A,\delta)>0$, such that for any $\phi_1$ in the $\epsilon_0$-neighborhood of $\phi$ (with respect to the $H^{1}$-topology), $([\tau_{j-1},\tau_j])_{j=1}^m$ is an $(2A,2\delta)$-partition for $\phi_1$. 
\end{prop}
\begin{proof}
	Denote $u=\Phi(t)\phi$ and $u_1=\Phi(t)\phi_1$. Fix an interval $I_j=[\tau_{j-1},\tau_j]$.
	By the Strichartz inequality,
	$$ \|\e^{i(t-\tau_0)\Delta}(\phi-\phi_1)\|_{Z(I_1)}\leq C_1\|\e^{i(t-\tau_0)\Delta}(\phi-\phi_1) \|_{X^1(I_1)} \leq  C_1\|\phi-\phi_1\|_{H_x^{1}}\leq C_1\epsilon_0
	$$
	for some absolute constant $C_1>1$.
	Pick $\delta<\delta_1 $ as in Lemma \ref{LWP}, and for $\epsilon_0\ll \delta$ (so that $C_1\epsilon_0<\delta$) and for any $\phi_1$ in an $\epsilon_0$-neighborhood of $\phi$ (with respect to the $\dot{H}^1$-topology), the solution $u_1$ (which exists thanks to Lemma \ref{LWP}) with initial data $\phi_1$ satisfies
	$$ \|u_1\|_{X^{1}(I_1)}\leq \frac{3}{2}A,\; \|u_1\|_{Z(I_1)}<\frac{3}{2}\delta.
	$$ 
	Applying Lemma \ref{multilinear}, we have
	\begin{align*}
		\|u-u_1\|_{X^{1}(I_1)}\leq &\|\e^{it\Delta}(\phi-\phi_1)\|_{X^{1}(I_{1})}+\big\|
		|u|^4u
		-|u_1|^4u_1\big\|_{N^{1}(I_1)}\\
		\leq &\|\phi-\phi_1\|_{H_x^{1}}+C\|u-u_1\|_{X^{1}(I_1)}(\|u\|_{Z(I_1) }^2+\|u_1\|_{Z(I_1)}^2)(\|u\|_{X^{1}(I_1)}^2+\|u_1\|_{X^{1}(I_1)}^2)\\
		\leq &\|\phi-\phi_1\|_{H_x^{1}}+CA^2\delta^2\|u-u_1\|_{X^{1}(I_1)}.
	\end{align*} 
	Taking $\delta>0$ small enough such that $1-CA^2\delta^2>\frac{1}{2}$, we deduce that 
	$$ \|u-u_1\|_{Z(I_1)}\leq C_1\|u-u_1\|_{X^1(I_1)}\leq 2C_1\|\phi-\phi_1\|_{H_x^{1}}<2C_2\epsilon_0.
	$$
	In particular, by the embedding property $X^{1}(I)\hookleftarrow L^{\infty}(I;H_x^{1}) $ , for almost every $\tau_1^*\in (\tau_0,\tau_1)$, $$\|u(\tau_1^*)-u_1(\tau_1^*) \|_{H_x^{1}}\leq C_2\|u-u_1\|_{X^{1}(I_1)}\leq 2C_2\|\phi-\widetilde{\phi}\|_{H_x^{1}}\leq 2C_2\epsilon_0,
	$$
	where $C_2>0$ is another absolute constant. By the continuity of the flows $t\mapsto u(t), u(t_1)$, we have 
$$
\|u(\tau_1)-u_1(\tau_1)\|_{H_x^{1}}\leq 2C_2\epsilon_0,
 \quad	
  \|u_1\|_{Z(I_1)}\leq \|u\|_{Z(I_1)}+\|u-u_1\|_{Z(I_1)}<\delta+2C_1\epsilon_0, $$
hence 
$$
	\|\e^{i(t-\tau_1)\Delta}(u(\tau_1)-u_1(\tau_1))\|_{Z(I_2)}\leq C_1\|u(\tau_1)-u_1(\tau_1)\|_{H_x^1}\leq 2C_1C_2\epsilon_0.
	$$
	By choosing $\epsilon_0$ small enough such that $\epsilon_0\sum_{j=1}^m(2C_1C_2)^j<\delta$, we can repeat the argument above until $I_m$. In particular, $\|u_1\|_{Z(I_j)}<2\delta, \|u_1\|_{X^{1}(I_j)}<2A$ for all $j=1,2,\cdots,m$ This implies that $([\tau_{j-1},\tau_j])_{j=1}^m$ is an $(2A,2\delta)$-partition with respect to $\phi_1$. The proof of Proposition \ref{stabilitypartition}  is complete.	
\end{proof}
%%%%%
Now we are ready to prove:
\begin{prop}[Long-time approximation]\label{long-timeapproximation} 
	Given $T\geq 1$ and $\phi\in H^1(\T^3)$. 
	$$ \lim_{N\rightarrow\infty}\|\Phi_N(t)\phi-\Phi(t)\phi\|_{H^1(\T^3)}=0,\forall t\in[-T,T],
	$$
	uniformly for $\phi$ on a compact set $K\subset H^1(\T^3)$. Moreover, for any $|t|\leq T$ and $N\in \N$,  the sets $\Phi(t)(K),\Phi_N(t)(K)$ are compact in $H^1(\T^3)$.
\end{prop}

\begin{proof}
	
	Fix $A>0, T>0$ and $K\subset B^{H^1}_{A/2}(0):=\{\phi:\; \|\phi\|_{H^1}\leq A/2 \}$ a compact set of $H^1(\T^3)$. Note that for any $\phi\in B_{A/2}^{H^1}$, $H[\phi]\leq C_0A^2$.
	We divide the proof into several steps. 
Set $A_1=4\sqrt{C_0}A$.
\begin{flushleft}
	\underline{Step 1: Existence of a uniform $(A_1,\delta)$-partition:} 
\end{flushleft}	
	Thanks to Theorem \ref{GWP}, for any $\phi\in K$ in $H^1(\T^3)$, $\|\Phi(t)\phi\|_{Z([-T,T]) }\leq \Lambda(C_0A^2,T)<\infty$. In particular, there exists an $(\frac{A_1}{2 },\frac{\delta}{2})$-partition $([\tau_{j-1},\tau_j])_{j=1}^{m}$ with respect to $\phi$. By stability (Proposition \ref{stabilitypartition}), there exists $\epsilon_0=\epsilon_0(m,A_1,\delta )>0$, such that $([\tau_{j-1},\tau_j])_{j=1}^m$ is an $(A_1,\delta)$-partition with respect to all $\phi_1\in B_{\epsilon_0}^{H^1}(\phi)$. Since $K$ is compact, there exist finitely many $\phi_1,\cdots,\phi_n\in K,$ $\epsilon_i>0, i=1,\cdots,n$ and $(A,\delta)$-partitions $([\tau^{(i)}_{j-1},\tau^{(i)}_j])_{j=1}^{m_i}, i=1,\cdots,n$, such that
	\begin{itemize}
		\item[(1)] $K\subset \bigcup_{i=1}^nB_{\epsilon_u}^{H^1}(\phi_i)$;
		\item[(2)] $([\tau^{(i)}_{j-1},\tau^{(i)}_j])_{j=1}^{m_i}$ is an $(A_1,\delta)$-partition for all $\phi\in B_{\epsilon_i}^{H^1}(\phi_i)$, $i=1,\cdots,n$.
	\end{itemize}
	Consider a refinement $([\tau_{j-1},\tau_j])_{j=1}^m$ of partitions $([\tau^{(i)}_{j-1},\tau^{(i)}_j])_{j=1}^{m_i}$. By Proposition \ref{refinement}, $([\tau_{j-1},\tau_j])_{j=1}^m$ is a uniform $(A_1,\delta)$-partition with respect to all $\phi\in K$.

\begin{flushleft}
	\underline{Step 2: Long-time convergence:}  
\end{flushleft}
	Now we are able to iterate Lemma \ref{LWP} (with the small parameter $2\delta$ instead of $\delta$ in the statement) from $I_1=[\tau_0,\tau_1]$ to $I_m=[\tau_{m-1},\tau_m]$. Thanks to Proposition \ref{integrabilitypartition} and the energy conservation law, we have
	$$ \|\e^{i(t-\tau_{j-1})\Delta}\Phi(\tau_{j-1})\phi\|_{Z(I_j)}<2\delta, \; \|\Phi(\tau_{j-1})\phi\|_{H_x^1}\leq A_1.
	$$
	In order to apply Lemma \ref{LWP} on each interval $I_j$ with initial data $\Phi_N(\tau_{j-1})\phi$ and $\Phi(\tau_{j-1})\phi$, we have to ensure that
	\begin{align}\label{induction} \|\Phi_N(\tau_{j-1})\phi-\Phi(\tau_{j-1})\phi\|_{H_x^1}<2\delta.
	\end{align}
	First, since $\delta_N(A_1,2\delta,\phi)\rightarrow 0$ uniformly on the compact set $K$, we can choose $N_0$ large enough, such that for all $N\geq N_0, \phi\in K$,
	$$ (C_0+C_0^2+\cdots+C_0^m)\delta_N(A_1,2\delta,\phi)<\delta.
	$$
	Now we argue by induction that 
	\begin{align}\label{induction'} \|\Phi_N(t)\phi-\Phi(t)\phi\|_{C(\ov{I_j};H_x^1)}\leq (C_0+\cdots+C_0^j)\delta_N(A_1,2\delta,\phi).
	\end{align}
	. When $j=1$, $\Phi_N(\tau_0)\phi=\Phi(\tau_0)\phi$, and from the last assertion of Lemma \ref{LWP}, we have
	\begin{align*} \|\Phi_N(t)\phi-\Phi(t)\phi\|_{C(\ov{I_1};H_x^1)}\leq C_0\delta_N(A_1,2\delta,\phi).
	\end{align*}
	Assume that \eqref{induction'} holds for some $j\geq 1$, in particular, \eqref{induction} holds thanks to our choice. Then
	we are able to apply Lemma \ref{LWP} on the time interval $I_{j+1}$ to obtain that
	\begin{align*}  \|\Phi_N(t-\tau_{j-1})&\Phi_N(\tau_{j-1})\phi-\Phi(t-\tau_{j-1})\Phi(\tau_{j-1})\phi\|_{C(\ov{I_{j}};H_x^1)} \\ \leq &C_0\|\Phi_N(\tau_{j-1})\phi-\Phi(\tau_{j-1})\phi\|_{H_x^1}+C_0\delta_N(A_1,2\delta,\phi)\\
		\leq &(C_0+C_0^2+\cdots+C_0^{j+1})\delta_N(A_1,2\delta,\phi).
	\end{align*}
	Hence  \eqref{induction'} holds for all $j=1,2,\cdots,m$. In particular, we have
	$$ \|\Phi_N(t)\phi-\Phi(t)\phi\|_{C([-T,T];H_x^1)}\leq (C_0+\cdots+C_0^m)\delta_N(A_1,2\delta,\phi)
	$$
	which converges to $0$, as $N\rightarrow\infty$, uniformly in $\phi\in K$. 
%\begin{flushleft}	
%	\underline{Step 3: Compactness of $\Phi(t)(K)$ and $\Phi_N(t)(K)$:}
%\end{flushleft}
	
%	Fix $|t_1|\leq T$, we will only prove the compactness of $\Phi(t_1)(K)$ since the proof is the same for $\Phi_N(t_1)(K)$. Let $u_n$ be a bounded sequence of $\Phi(t_1)(K)$, our goal is to show that $v_n$ has a convergent subsequence whose limit belongs to $\Phi(t_1)(K)$. Set $\phi_n=\Phi(-t_1)u_n$, by compactness of $K$, there is a convergent subsequence $\phi_{n_k}$ such that $\phi_{n_k}\rightarrow \phi$ in $H^1$ and $\phi\in H^1$. 
	
   % Recall that $([\tau_{j-1},\tau_j])_{j=1}^m$ is the uniform $(A,\delta)$-partition used in Step 1, 
    %by Remark \ref{rk:stability}, first we have
    %$$ \|\Phi(\tau_1)\phi_{n_k}-\Phi(\tau_1)(\phi)\|_{H_x^1}\leq C_0\|\phi_{n_k}-\phi\|_{H_x^1},
    %$$
    %provided that $k$ large enough such that $C_0\|\phi_{n_k}-\phi\|_{H_x^1}<\delta$.  For large $k$ such that 
    %$$ (C_0+C_0^2+\cdots +C_0^m)\|\phi_{n_k}-\phi\|_{H_x^1}<\delta,
    %$$
	%by repeatedly using Lemma \ref{LWP} we obtain that for all $j=1,\cdots,m$,
	%$$ \|\Phi(t)\phi_{n_k}-\Phi(t)\phi \|_{C(\ov{I_j};H_x^1 )}\leq (C_0+\cdots+C_0^j)\|\phi_{n_k}-\phi\|_{H_x^1}<\delta.
	%$$
	%Finally we obtain that
	%$$ \lim_{k\rightarrow\infty}\|v_{n_k}-\Phi(t_1)\phi \|_{H_x^1}\leq  \lim_{k\rightarrow\infty}\|\Phi(t)\phi_{n_k}-\Phi(t)\phi\|_{C([-T,T];H_x^1)}=0.
	%$$
	%Since $\Phi(t_1)\phi\in \Phi(t_1)(K)$, 
	%we conclude 
	This completes the proof of Proposition \ref{long-timeapproximation}.
\end{proof} 
	
Now we are ready to accomplish the proof of Proposition \ref{uniformboundHsigma} and Proposition \ref{approximation}:	
\begin{proof}[Proof of Proposition \ref{uniformboundHsigma}]

If $\sigma=1$, the $H^1$-uniform bound for $\Phi_N(t)\phi$ and $\Phi(t)\phi$ follows from the defocusing feature of \eqref{NLS} and the conservation of energy. Now we assume that $\sigma>1$. By the compact embedding $H^{\sigma}(\T^3)\hookrightarrow H^1(\T^3)$, the ball $B_R^{H^{\sigma}}$ is compact with respect to the $H^1$-topology.
	By the same argument as Step 1 in the proof of Proposition \ref{long-timeapproximation}, there exists a uniform $(A,\delta)$-partition (with $A=R$ here and $\delta<\frac{1}{\sqrt{2C(R,T)}}$ as in the proof of Corollary \ref{Hsigmabound}) $([\tau_{j-1},\tau_j])_{j=1}^m$ of $[-T,T]$, where $m$ depends only on the $R,T$ and $\sigma$.  Repeating the analysis in the proof of Corollary \ref{Hsigmabound}, we obtain that for all $|t|\leq T$ and $N\in\mathbb{N}$, $$ \|\Phi(t)\phi\|_{H_x^{\sigma}}+\|\Phi_N(t)\phi\|_{H_x^{\sigma}}\leq C^m\|\phi\|_{H_x^{\sigma}}.
	$$
	This completes the proof of Proposition \ref{uniformboundHsigma}.
\end{proof}

\begin{proof}[Proof of Proposition \ref{approximation}]

We assume that $\sigma>1$, otherwise, the proof is completed as Proposition \ref{long-timeapproximation}.
Let $K$ be a compact set of $H^{\sigma}(\T^3)$. In particular, $K$ is bounded of $H^{\sigma}(\T^3)$ and compact with respect to the $H^1(\T^3)$-topology. To prove the uniform  convergence on $K$, we follow the same scheme of analysis as in the previous section.  By Proposition \ref{uniformboundHsigma}, there exists a constant $D(K,T)$ depending only on $T>0$ and the compact set $K$ in $H^{\sigma}$, such that
\begin{align}\label{boundHsigmauniform} \sup_{t\in[-T,T]}\|\Phi(t)\phi\|_{H^{\sigma}}+\sup_{t\in[-T,T]}\|\Phi_N(t)\phi\|_{H^{\sigma}}\leq D(K,T)
\end{align}
for all $N\in\N$. At this stage, we are able to repeat the argument as in the proof of Proposition \ref{long-timeapproximation} line by line if we replace the norms  $H^1,X^1,N^1$ everywhere
by $H^{\sigma}, X^{\sigma}, N^{\sigma}$  and the constants $A$ by $D(K,T)$. We omit the details and conclude.
\end{proof}
	%%%%%%%%%%%%%%%%%%%%%%%%%%%%%%%%%%%%%%%%%%%%%%%%%%%%%%%%%%%%%%%%%%%%%%%%
	
	%%%%%%%%%%%%%%%%%%%%%%%%%%%%%%%%%%%%%%%%%%%%%%%%%%%%%%%%%%%%%%%%%%%%%

\end{document}